\documentclass[11pt,leqno]{amsart}

\usepackage[all,cmtip]{xy}
\usepackage{a4,latexsym}
\usepackage[utf8]{inputenc}
\usepackage{textcomp,fontenc,exscale,ifthen}
\usepackage[hidelinks]{hyperref}
\usepackage{amsxtra,amsmath,amsfonts,amscd, amssymb, mathrsfs, amsthm}
\usepackage{color, mathtools}
\usepackage[shortlabels]{enumitem}

\newcommand{\printmode}{0} 

\newtheorem{thm}{Theorem}[subsection]
\newtheorem{prop}[thm]{Proposition}
\newtheorem{cor}[thm]{Corollary}
\newtheorem{lem}[thm]{Lemma}
\newtheorem{theointro}{Theorem}

\theoremstyle{definition}
\newtheorem{defi}[thm]{Definition}
\newtheorem{ex}[thm]{Example}
\newtheorem{rem}[thm]{Remark}
\newtheorem{art}[thm]{}

\numberwithin{paragraph}{section}
\numberwithin{equation}{section}

\setlist{label=\it{(\roman*)}}

\usepackage{color}
\usepackage{comment}

\ifthenelse{\equal{\printmode}{0}}
{}
{}

\ifthenelse{\equal{\printmode}{0}}
{}
{}

\ifthenelse{\equal{\printmode}{0}}
{}
{}

\ifthenelse{\equal{\printmode}{0}}
{}
{}

\ifthenelse{\equal{\printmode}{0}}
{}
{}

\def\P{{\mathbb P}}
\def\N{{\mathbb N}}
\def\Z{{\mathbb Z}}

\def\R{{\mathbb R}}
\def\C{{\mathbb C}}

\def\A{{\mathbb A}}

\def\G{{\mathbb G}}
\def\K{{\mathbb K}}

\def\N{{\mathbb N}}
\def\T{{\mathbb T}}
\def\SS{{\mathbb S}}

\newcommand{\metr}{{\|\hspace{1ex}\|}}
\newcommand{\curr}{{[\hspace{1ex}]}}
\newcommand{\Hom}{{\rm Hom}}

\def\an{{\rm an}}

\DeclareMathOperator{\trop}{trop}

\newcommand{\Tan}{{T^{\rm an}}}

\newcommand{\Spec}{{\rm Spec}}

\newcommand{\cl}{{\rm cl}}
\newcommand{\av}{{\rm av}}

\newcommand{\supp}{{\rm supp}}

\newcommand{\relint}{{\rm relint}}

\newcommand{\sym}{{\rm sym}}

\newcommand{\AS}{{A}}

\newcommand{\Rsup}{\R_\infty}
\newcommand{\torus}{{\T}}
\newcommand{\Xsigmaan}{{X_\Sigma^{\rm an}}}
\newcommand{\torusan}{{\torus^{\rm an}}}
\newcommand{\dell}{{\partial}}
\newcommand{\dellbar}{{\bar \partial}}

\DeclareMathOperator{\Mon}{{\bf Mon}}
\DeclareMathOperator{\Gr}{Gr}

\title{{A} comparison of positivity  
in complex and tropical toric geometry}

\author[J.I.Burgos Gil]{Jos\'e Ignacio Burgos Gil}
\address{
J.I.~Burgos Gil,
Instituto de Ciencias Matem\'aticas (CSIC-UAM-UCM-UCM3), 
Calle Nicol\'as Cabrera 15, Campus de la Universidad 
Aut\'onoma de Madrid, Cantoblanco, 28049 Madrid, Spain}
\email{burgos@icmat.es}

\author[W.~Gubler]{Walter Gubler}
\address{W. Gubler, Mathematik, Universit{\"a}t 
Regensburg, 93040 Regensburg, Germany, ORCID: 0000-0003-2782-5611}
\email{walter.gubler@mathematik.uni-regensburg.de}

 \author[P.~Jell]{Philipp Jell}
 \address{P. Jell,  Mathematik, Universit{\"a}t 
Regensburg, 93040 Regensburg, Germany, ORCID: 0000-0003-2757-1770}
\email{philipp.jell@mathematik.uni-regensburg.de}

\author[K.~K\"unnemann]{Klaus K{\"u}nnemann}
\address{K. K{\"u}nnemann, Mathematik, Universit{\"a}t 
Regensburg, 93040 Regensburg, Germany, ORCID: 0000-0003-1109-9096}
\email{klaus.kuennemann@mathematik.uni-regensburg.de}

 \thanks{J.~I.~Burgos was partially supported by MINECO research projects MTM2016-79400-P and 
by ICMAT Severo Ochoa project SEV-2015-0554.
W.~Gubler, P.~Jell and K.~K{\"u}nnemann 
were supported by the collaborative research 
center SFB 1085 \emph{Higher Invariants - Interactions between Arithmetic Geometry and Global Analysis} funded by the Deutsche Forschungsgemeinschaft.}

\begin{document}

\begin{abstract}
Given a smooth complex toric variety 
we will compare real Lagerberg forms and currents on its
tropicalization with invariant complex forms 
and currents on the toric variety. 
Our main result is a correspondence theorem which identifies
the cone of invariant closed positive currents on the complex toric
variety with closed positive currents on the tropicalization. 
In a subsequent paper, this correspondence will be used
to develop a Bedford--Taylor theory of 
plurisubharmonic functions on the tropicalization. 
\end{abstract}

\keywords{{Toric varieties, tropicalization, positive currents, Lagerberg forms}} 
\subjclass{{Primary 14L32; Secondary 14T05, 32U05, 32U40}}

\maketitle

\tableofcontents

\section{Introduction}\label{introduction}

A smooth complex toric variety $X_{\Sigma}$ is a smooth algebraic variety over $\C$ with an open 
immersion of an algebraic torus $\T$ and an algebraic extension of the
action of $\T$ on itself to $X_{\Sigma}$. 
Such a variety is encoded by the combinatorial structure of a fan $\Sigma$ 
in the vector space $N_\R := N \otimes_\Z \R$ where $N$ is the cocharacter lattice of $\T$.
We denote by $\T^{\an} := \T(\C)$ and $\Xsigmaan \coloneqq X_{\Sigma}(\C)$ 
the respective complex analytic manifolds 
of complex points. 
Inside $\T^{\an}$, there is a maximal compact torus $\mathbb{S}$. 
The quotient $\Xsigmaan / \mathbb{S}$ is denoted by $N_{\Sigma}$. 
The topological space $N_{\Sigma}$ has a canonical stratification
\begin{displaymath}
  N_{\Sigma }=\coprod_{\sigma \in \Sigma }N(\sigma )
\end{displaymath}
indexed by the cones of $\Sigma $,
analogous to the stratification of $X_{\Sigma}$ into orbits. 
The incidence
relations between cones of $\Sigma $ translate to incidence relations
between strata with the inclusions reversed. 
If $\tau$ is a face of the cone $\sigma \in \Sigma$,
denoted as $\tau\prec \sigma$,  then the corresponding strata
satisfy $N(\sigma  )\subset \overline {N(\tau )}$. The stratum
corresponding to the cone $\{0\}$ is equal to $N_\R$. 
It is called the \emph{dense stratum}. 
Each stratum $N(\sigma)$ of $N_\Sigma$ has a canonical
structure of a finite dimensional real vector space
equipped with a $\Z$-structure.
For $\tau \prec \sigma \in \Sigma $, there is a linear projection
\begin{displaymath}
  \pi_{\sigma ,\tau }\colon N(\tau )\longrightarrow N(\sigma ). 
\end{displaymath}
The space $N_\Sigma$ is a classical object in the theory of toric varieties
where it appears also under the name \emph{manifold with corners}
(see e.g.~\cite{AMRT, oda}]). In tropical geometry, it is called 
the \emph{Kajiwara-Payne tropicalization} of $X_{\Sigma}$.  
By construction, it comes with a natural map $\trop \colon \Xsigmaan
\to N_\Sigma$ 
which is a proper map of topological spaces. 

On the spaces $\Xsigmaan$ and $N_\Sigma$, there are sheaves of
bigraded algebras of smooth differential forms. 
Both are denoted $A^{\cdot, \cdot}$ or, when we want to stress the {underlying} space, 
by $A_{\Xsigmaan}^{\cdot, \cdot}$ and $A_{N_{\Sigma }}^{\cdot, \cdot}$ respectively.  
The well-known sheaf of complex smooth
differential forms $A_{\Xsigmaan}^{\cdot, \cdot}$ is a sheaf of
$\C$-algebras and plays a central role  
in complex analysis and complex geometry. The
sheaf $A_{N_{\Sigma }}^{\cdot, \cdot}$ is a sheaf of $\R$-algebras and was 
introduced by Smacka, Shaw and the third author
\cite{jell-shaw-smacka2015},
based on work of Lagerberg \cite{lagerberg-2012}. 
The smooth differential forms on $\Xsigmaan$ are called \emph{complex forms} while the
forms on $N_\Sigma$ are called \emph{Lagerberg forms}.
In both cases, the elements of $A^{0,0}$ are called smooth functions.

We explain briefly the definition of Lagerberg forms. More 
details are given in Section \ref{section smooth forms tropical}.
If  $U$ is an open subset of the finite dimensional real
vector space $N(\sigma)$, Lagerberg \cite{lagerberg-2012} has introduced the bigraded $\R$-algebra 
\begin{align*}
A^{\cdot,\cdot}(U)\coloneqq 	
\bigoplus_{p,q \in \N} A^p(U)\otimes_{C^\infty(U)}A^q(U)
\end{align*}
where $A^\cdot(U)$ denotes the usual $\R$-algebra of real
smooth differential forms on $U$. 
For an open set   
$U\subset N_{\Sigma }$, denote by $U_{\sigma}\coloneqq N(\sigma )\cap
U$ its strata. Then a Lagerberg form on $U$ is defined as a
collection of forms 
$(\omega _{\sigma })_{\sigma \in \Sigma }$, with $\omega _{\sigma
}\in A^{\cdot,\cdot}(U_{\sigma })$, satisfying the following
compatibility conditions. For every pair of cones $\tau \prec \sigma $
and every point $p\in U_\sigma$, there is a neighborhood $V\subset U$ of
$p$ with $V_\tau = \pi _{\tau ,\sigma }^{-1}(V_\sigma)\cap V$ and
\begin{equation}
  \label{eq:3}
  \omega _{\tau }|_{V_{\tau }}=\pi _{\tau ,\sigma }^{\ast}(\omega
  _{\sigma }|_{V_{\sigma }})
\end{equation}
on $V_\tau$.  
The compatibility conditions \eqref{eq:3} are,  
roughly speaking, saying that close to the boundary, 
Lagerberg forms are constant in the direction {towards} 
the boundary. Although this condition does not seem
entirely natural from an archimedean
point of view, it is very natural from both a tropical
\cite{jell-shaw-smacka2015} and a non-archimedean point of view
\cite{jell-2019}. Moreover, it has very strong consequences. For
instance, if $\omega $ is a form of bidegree $(p,q)$,
then the support of $\omega $ is
disjoint to any stratum of dimension smaller than $\min(p,q)$.

There are natural differential operators $d',d''$ of bidegree
$(1,0)$ and $(0,1)$ turning $A_{N_{\Sigma }}^{\cdot,\cdot}$ 
into a double complex analogous to the usual differential operators 
 $\partial $ and $\bar \partial$ on $A_{\Xsigmaan}^{\cdot,\cdot}$. 
There is also a theory of integration for Lagerberg forms similarly to the complex case.

The sheaf $A^{\cdot,\cdot}_{\Xsigmaan}$ has an antilinear
involution, the \emph{complex conjugation}, that sends $A^{p,q}$ to
$A^{q,p}$. A complex form is \emph{real}, if it is invariant under
complex conjugation. The sheaf $A^{\cdot,\cdot}_{N_{\Sigma }}$ has also
a canonical involution $J$ called the \emph{Lagerberg involution}. A form $\omega $
of bidegree $(p,p)$ with $J(\omega )=(-1)^{p}\omega $ is called 
\emph{symmetric}.
In both settings, there is a notion of positivity for $(p,p)$-forms (see \ref{positivity notions for toric varieties}).
Positive forms on $\Xsigmaan$ are always real, 
while  positive forms on $N_{\Sigma }$ are always symmetric.

In \S \ref{comparison-vector-space-section} and in \S \ref{subsection invariant forms toric variety}, 
we will introduce a new antilinear involution $F$ on the sheaf of $\SS$-invariant forms on $\Xsigmaan$ 
that respects the bigrading, and anticommutes with complex conjugation on one-forms.  
We denote by $A_{\Xsigmaan}^{\SS,F}$ the subsheaf of $\SS$-invariant forms which are $F$-invariant.

\begin{theointro}\label{thm:A}
{Let $U\subset N_\Sigma$ be open and $V$ the $\SS$-invariant
open subset $V\coloneqq \trop^{-1}(U)$ of $X_\Sigma^\an$.}
There exists a unique bigraded algebra morphism
\begin{equation}\label{chija-3}
\mathrm{trop}^*\colon A^{\cdot,\cdot}(U) \longrightarrow A^{\cdot,\cdot}(V)^{\mathbb S,F}
\end{equation}
with 
\begin{math}
  \trop^{\ast}\varphi = \varphi\circ \trop
\end{math} for all $\varphi \in A^{0,0}(U)$ and which
 satisfies
\begin{equation}
  \label{eq:4}
  \trop^{\ast} \circ \,d' = \pi ^{-1/2}\partial\circ \trop^{\ast},\quad
    \trop^{\ast} \circ \,d'' = \pi ^{-1/2}i\bar\partial\circ \trop^{\ast}.
\end{equation}
Moreover, for $\omega \in \AS ^{p,q}(U)$, we have
\begin{equation}\label{eq:5}
  \trop^{\ast}(J(\omega ))=i^{p+q}\overline{\trop^{\ast}(\omega )}.
\end{equation}
Therefore $\trop^{\ast}$ sends symmetric forms to real
forms. Furthermore, this morphism respects 
positivity and integration of top dimensional forms.
\end{theointro}

If $U$ is contained in the dense stratum $N_\R$, then \eqref{chija-3}
is an isomorphism. In general, this is no longer true.  
The reason for this is the compatibility conditions \eqref{eq:3}. 
These results will be shown in {Section \ref{Complex invariant forms}}. 
The normalization factors in equation \eqref{eq:4} are almost forced
by the compatibility with integration and the compatibility between
the Lagerberg involution and complex conjugation \eqref{eq:5}.
If we do not insist on compatibility with integration or with the
bigrading, other identifications between Lagerberg forms and invariant
complex forms are possible. For instance, the map \eqref{chija-3}
differs from the interpretation of
Lagerberg forms 
as $\SS$-invariant forms given in  \cite[Remarque
(1.2.12)]{chambert-loir-ducros}.

The main interest of this paper will be  currents. 
To define currents on an open subset $U\subset N_{\Sigma}$, 
we first introduce a topology on the space of Lagerberg forms with
compact support $A_{c}^{\cdot,\cdot}(U) $. 
The definition is similar to the complex case 
with additional input  caused by the compatibility
condition \eqref{eq:3} towards the boundary 
(see Subsection \ref{subsection smooth forms tropical} for details). 
A \emph{Lagerberg current of type $(p,q)$} on $U$ is a continuous linear map
$T\colon A^{n-p,n-q}_c(U)\to \R$. 
We denote the space of currents of type $(p,q)$ by $D^{p,q}(U)$.
By duality, the involution $J$ defines an involution on
$D^{\cdot,\cdot}$, hence a notion of symmetric currents. 
We call $T\in D^{p,p}(U)$ \emph{positive}, if it is symmetric and
$T(\alpha)\geq 0$ for all positive Lagerberg forms $\alpha\in A^{n-p,n-p}_c(U)$.

Let again $V = \trop^{-1}(U)$. Since the map $\trop$ is
proper, the dual of the map $\trop^*$ from
\eqref{chija-3}
induces a $\C$-linear map 
\begin{equation}\label{chija-22}
\trop_*\colon D^{p,q}(V)\longrightarrow D^{p,q}(U)\otimes_\R \C
\end{equation}
defined by $\trop_*(T)(\alpha)=T\bigl(\trop^*(\alpha)\bigr)$
for all $T\in D^{p,q}(V)$ and $\alpha\in A^{n-p,n-q}_c(V)$.
We will show that $\trop_*(T) \in D^{p,p}(U)$ for every 
$\SS$-invariant $T \in D^{p,p}(V)$ which is also $F$-invariant. 
Since the map $\trop^{\ast}\colon A_{c}^{p,q}(U)\to A_{c}^{p,q}(V)$,
albeit injective, is
not a closed immersion, the map $\trop_{\ast}$ is in general not
surjective.
Since $\trop^*$ preserves positivity of forms, $\trop_*$ preserves
positivity of currents. 
All this will be shown in Section \ref{Complex invariant currents}.

The main result of this paper is the following Correspondence Theorem.

\begin{theointro} \label{main thm intro}
The map $\trop_*$ induces a linear isomorphism between the following cones:
\begin{enumerate}
\item\label{main thm intro 1} The cone of $\SS$-invariant positive
  complex currents
  in $D^{p,p}(V)$
that are closed with respect to $\dell$ and $\dellbar$ and that are
invariant with respect to $F$. 
\item \label{main thm intro 2} The cone of positive Lagerberg currents in
  $D^{p,p}(U)$ that are closed with respect to $d'$ and $d''$.   	
\end{enumerate} 
\end{theointro}

The proof of Theorem \ref{main thm intro} will be given in Theorem \ref{thm:5}. 
We will show in Examples \ref{exm:1} and \ref{exm:1'} that we cannot omit 
any of the conditions \emph{closed} or \emph{positive}
in Theorem \ref{main thm intro}. 

Let us give some details about the proof
of the Correspondence Theorem. 
In the case $p = n \coloneqq \dim X_{\Sigma}$, 
the cone in \ref{main thm intro 1} is the space of positive
$\SS$-invariant Radon
measures on $V$ and
the cone in \ref{main thm intro 2} is the space of positive Radon
measures on $U$,
hence Theorem \ref{main thm intro} follows readily from the fact that
$U$ is the quotient of $V$ 
by the $\SS$-action (see Corollary \ref{correspondence for top degree currents}). 
This suggests that for $p<n$, we consider coefficients of complex
and Lagerberg currents 
to follow a similar argument. 
In fact, it is more convenient to go for the dual notion of
co-coefficients  as follows:

For simplicity, we consider the case $X_\Sigma^\mathrm{an}= \C^n$
with coordinates $z = (z_1,\ldots,z_n)$, a situation that can always be achieved locally.   
Then $N_\Sigma=(\R\cup\{\infty\})^n$ has coordinates $(u_1,\ldots,u_n)$
and the tropicalization map  is given by 
$\trop(z) =( -\log \vert z_1\vert,\dots, -\log \vert z_n \vert)$.  
The cones of $\Sigma$ are the faces of $\R_{\geq 0}^n$. 
Any cone in $\Sigma$ has the form 
$\sigma_L \coloneqq \{u \in \R_{\geq 0}^n \mid u_i = 0 \ \forall { i\notin L}\}$ 
for some $L \subset \{1,\dots,n\}$ and the corresponding stratum 
 of $N_\Sigma$ is given by 
\begin{align*}
N(\sigma_L) \coloneqq \{ (u_1,\dots,u_n) \in N_\Sigma \mid u_i = \infty \text{ if and only if } i \in L\}.
\end{align*}
 For $M \subset \{1,\ldots,n\}$, we define the following union of strata of codimension $1$:
\begin{align*}
E^M:=\{(u_1,\ldots,u_n)\in N_\Sigma \mid \text{$u_i=\infty$ for some $i\in M$}\}.
\end{align*}
{Let $U$ be an open subset of $(\R\cup\{\infty\})^n$.} Given a Lagerberg current $T\in D^{p,p}(U)$ and 
 $I,J\subset \{1,\ldots,n\}$
with $|I|=|J|=n-p$, we call 
$T^{IJ}\in D^{n,n}(U\setminus E^{I\cup J})$ given by
\begin{align*}
T^{IJ}(f)=T((-1)^{q(q-1)/2}fd'u_I\wedge d''u_J)
\end{align*}
a \emph{co-coefficient} of $T$. 
If the Lagerberg current $T$ is positive, then the co-coefficients
$T^{IJ}$ are real Radon measures, 
$T^{IJ}=T^{JI}$, the $T^{II}$ are positive Radon measures and
$2|T^{IJ}|\leq T^{II}+T^{JJ}$, where $|T^{IJ}|$ denotes
the total variation measure of $T^{IJ}$ (see Subsection \ref{Co-coefficients subsection}).

A key ingredient in the proof of the Correspondence
Theorem is the 
following Decomposition Theorem along the above stratification.

\begin{theointro}  \label{stratification of currents intro}
Let $U$ be an open subset of $(\R\cup\{\infty\})^n$.
If $T$ is a positive current in  $D^{p,p}(U)$,
then there is a decomposition 
\begin{equation}\label{stratification-of-currents}
T = \sum_{\sigma \in \Sigma}T_\sigma
\end{equation}
with uniquely determined currents $T_\sigma$ 
such that $U \setminus (E^{I\cup J} \cup N(\sigma))$ 
is a null set with respect to the Radon measure ${T_\sigma^{IJ}}$ for any $\sigma = \sigma_L \in \Sigma$.
\end{theointro}

The decomposition \eqref{stratification-of-currents} does not depend on the choice of the coordinates 
$u_1, \dots, u_n$ hence gives a canonical decomposition for any positive current 
$T \in D^{p,p}(U)$ on any open subset $U$ of $N_\Sigma$ as we show in Theorem \ref{stratification of currents}.
A similar statement is well known on the complex toric manifold $\Xsigmaan$ and 
we show in Subsection \ref{subsection positive currents and the boundary} 
that both canonical decompositions are closely related via $\trop_*$.

For a positive Lagerberg current $T \in D^{p,p}(U)$, we will prove that 
$T=\trop_*(S)$ for a positive current $S$ on $V=\trop^{-1}(U)$ if and only if 
$T$ has \emph{$\C$-finite local mass}. 
The latter is a local condition on $U$ given in Definition \ref{def complex local finite mass}. 
We then show that a closed positive Lagerberg current has $\C$-finite
local mass,
completing the proof of surjectivity. 
For injectivity of $\trop_*$, an additional argument is required. 
All this is done in Subsection \ref{closed positive invariant currents}.

As an application, we will prove in Theorem \ref{tropical El Mir} 
a tropical analogue of the Skoda--El Mir Theorem for Lagerberg
currents in the toric setting: 
Let  $U$ be an open subset of $N_\Sigma$ and let $E$ be a union of
strata closures in $N_\Sigma$.
We consider a closed positive current $T \in D^{p,p}(U \setminus E)$ 
which has $\C$-finite local mass on $U$. 
Then we can extend $T$ by zero to a closed positive
Lagerberg current on $U$. 
For details, we refer to Subsection \ref{section:tropical Skoda El
  Mir}. A consequence of the Tropical Skoda--El Mir theorem is that in
the canonical decomposition \eqref{stratification-of-currents}, if $T$
is closed, then all the currents $T_\sigma$ are closed. 

The motivation for the present work is the following.
It is known that there is no way to continuously extend the wedge
product on $A^{\cdot, 
  \cdot}$ to $D^{\cdot, \cdot}$.
Bedford--Taylor theory provides a way to
define products
of certain closed positive currents   on $\Xsigmaan$. 
In a subsequent paper, using the Correspondence Theorem and
Bedford--Taylor theory for complex manifolds, we develop a 
Bedford--Taylor theory on $N_\Sigma$.

For instance, Bedford-Taylor theory on $N_\Sigma$ will have the
following application:  
Let $K$ be a field  endowed with a non-archimedean complete absolute
value $|\phantom{a}|_v$.
From the fan $\Sigma$, we may construct a toric variety $X_{\Sigma, K}$. 
We may also consider the analytification $X_{\Sigma,K}^{\rm an}$ 
of the toric variety $X_{\Sigma, K}$ as a Berkovich space and the
corresponding tropicalization map
\begin{align*}
\trop_K \colon X_{\Sigma,K}^{\rm an} \longrightarrow N_\Sigma
\end{align*}
which should be viewed as a non-archimedean analogue of the tropicalization map 
$\trop$ of the complex toric manifold ${X_{\Sigma}^\an}$ 
considered before. 
For any open subset $U$ of $N_\Sigma$,  
these tropicalization maps lead to a natural bijective correspondence 
between invariant continuous plurisubharmonic functions on $\trop^{-1}(U)$ and 
invariant continuous plurisubharmonic functions on $\trop_K^{-1}(U)$ 
in the sense of Chambert-Loir and Ducros \cite[\S 5.5]{chambert-loir-ducros}. 
Moreover, we show that the complex Bedford--Taylor theory 
corresponds to the Bedford--Taylor theory of Chambert-Loir and Ducros
\cite[\S 5.6]{chambert-loir-ducros}.

We explain in more detail the content of the paper. In Section
\ref{sec:posit-real-compl}, we introduce complex and Lagerberg multilinear forms. 
We interpretate the Lagerberg forms as the complex forms which are invariant under a natural involution $F$. 
Then we discuss several positivity notions of these 
 forms. In the complex case,  weakly
positive, positive and strongly positive multilinear forms are well-known. They form
strictly convex cones of maximal dimension. Moreover, the cone of
strongly positive forms is dual to the one of weakly positive forms, 
while the cone of positive forms is self dual. Similar notions are defined in \cite{lagerberg-2012} for 
Lagerberg multilinear forms on a real vector space. The
notions of weakly and strongly positive Lagerberg forms are 
however somewhat pathological. We show in 
Example \ref{too few strongly positive} that  
the cone of strongly positive
forms is not of full dimension and 
that the cone of weakly positive forms is
not strictly convex. 
This also implies that the notions of weakly and
strongly positive in the complex and Lagerberg case do not correspond
exactly. For this reason, we will mainly restrict
ourselves to positive forms. 

In Section \ref{section smooth forms tropical}, 
we will first introduce the partial compactification $N_\Sigma$ associated to a fan $\Sigma$ and 
we will describe its topology. 
Then we introduce Lagerberg forms on  $N_\Sigma$. 
Similarly as in complex analysis, 
we endow the space of compactly supported Lagerberg forms with a locally convex topology and 
we define the dual notion of Lagerberg currents. 
The upshot of this section is that Lagerberg forms and currents on $N_\Sigma$ 
satisfy similar properties as their complex analogues.

In Section \ref{Complex invariant forms}, we study positivity of
Lagerberg  forms and compare them to invariant complex
differential forms. In particular, we prove Theorem \ref{thm:A}. 
{We first deal with the dense torus $\mathbb T$ in Subsection \ref{subsection invariant forms torus} 
before we consider arbitrary smooth toric varieties in Subsection \ref{subsection invariant forms toric variety}.}

Section \ref{Complex invariant currents} is devoted to positivity of
Lagerberg currents and the relation with positivity of complex
currents. We define the map $\trop_{\ast}$ and discuss the 
compatibility of the different notions of positivity with respect to
this map. We also define the co-coefficients of a Lagerberg current
and show that, analogously to the complex case, the co-coefficients of
a positive Lagerberg current are Radon measures. We show in Example
\ref{measures and weakly positive Lagerberg currents2} that there are
weakly positive Lagerberg currents whose co-coefficients 
are not Radon measures. 
This is caused by the fact that the cone of
weakly positive multilinear Lagerberg forms is not strictly convex.

In Section \ref{decomposition-invariant-boundary}, 
we prove the decomposition theorem (see Theorem \ref{stratification of currents intro}). 
Moreover we introduce the concept of $\C$-finite local mass. 
Intuitively, a current has $\C$-finite local mass 
if it has local finite mass as an invariant complex current 
(see Definition \ref{def complex local finite mass} for details). 
We give in Example \ref{exm:3} a positive current $T$ that does not have
$\C$-finite local mass. Nevertheless, it is of the form
$T=\trop_{\ast}(S)$ for a complex current $S$, but this current cannot
be chosen to be positive.

Finally, Section \ref{correspondance-invariant-boundary} is devoted to
the proof of Theorem \ref{main thm intro} and of the tropical Skoda-El
Mir theorem.
In  Appendix \ref{facts from measure theory}, there is a reminder on Borel and Radon measures.

\addtocontents{toc}{\protect\setcounter{tocdepth}{0}}
\section*{{Acknowledgements}}
\addtocontents{toc}{\protect\setcounter{tocdepth}{1}}
We are deeply grateful to Florent Martin for his numerous inputs at the beginning of this project.
We are grateful to Roberto Gualdi for his comments on a previous 
version and for pointing out the reference \cite{babaee-huh2017}.
We thank the referee for helpful comments.

\addtocontents{toc}{\protect\setcounter{tocdepth}{0}}
\section*{Notation and conventions}
\addtocontents{toc}{\protect\setcounter{tocdepth}{1}}

The set $\N$ of natural numbers includes zero.
We write $\Rsup =\R \cup \{ \infty \}$ 
and $\R_{\geq t} \coloneqq \{u \in \R \mid u \geq t\}$ for any $t \in \R$.
In the notation $A \subset B$, we allow that $A=B$. 

In this paper,  $N$ usually denotes a free abelian group of rank $n$ with dual $M$. We 
denote by $N_\R$ and $M_\R$ their scalar extensions to
$\R$. By a  \emph{fan} $\Sigma$ in $N_\R$, 
we mean a fan consisting of strictly convex rational polyhedral cones in $N_\R$. 
We denote the associated (complex) toric variety by $X_\Sigma$ and 
the associated partial compactification of $N_\R$ by $N_\Sigma$ 
(see \S \ref{subsection partial compactification}).

A  topological space is called \emph{locally compact} if it has a basis consisting of relatively compact subsets; but it 
 is not necessarily  Hausdorff. 
A \emph{topological vector space} is assumed to be Hausdorff. 
Our conventions on \emph{Radon measures} are summarized in  Appendix \ref{facts from measure theory}.

\section{Positivity on real and complex vector spaces }
\label{sec:posit-real-compl}

Let $V$ be a real vector space of dimension $n$. 
Write $V_{\C}=V\otimes_\R \C$ for the associated complex vector space, 
$V^{\ast}=\Hom_{\R}(V,\C)=\Hom_{\C}(V_{\C},\C)$ for the complex dual 
and $V'=\Hom_{\R}(V,\R)$ for the real dual.  
Let $\overline{V_\C}$ denote the real vector space
underlying $V_\C$ with the complex structure determined 
by $\lambda (v\otimes \mu)=v\otimes (\bar\lambda \mu)$.
We denote the complex dual of $\overline{V_\C}$ by 
$\overline V^*$. Observe that
$\overline V^{\ast}$ agrees with the space of
antilinear maps from $V_\C$ to $\C$. 
Let $V''$ be a copy of $V'$. 
We consider the exterior algebras
\begin{displaymath}
\Lambda ^{\cdot,\cdot}V^{\ast}\coloneqq \Lambda _{\C}(V^{\ast}\oplus
\overline{V}^{\ast})=\bigoplus_{p,q \in \N} \Lambda ^{p,q}V^{\ast},
\end{displaymath}
and
\begin{displaymath}
\Lambda ^{\cdot,\cdot}V'\coloneqq \Lambda _{\R}(V'\oplus V'')=\bigoplus_{p,q \in \N} \Lambda ^{p,q}V'.
\end{displaymath}
The elements of $\Lambda ^{p,q}V^{\ast}$ are called \emph{complex
$(p,q)$-forms} while the elements of $\Lambda ^{p,q}V'$ are called
\emph{Lagerberg $(p,q)$-forms}.

\subsection{The complex situation} \label{subsection: The complex situation}
We recall some definitions from complex geometry.

The identity induces antilinear maps
$\sigma \colon V_\C\to \overline{V_\C}$ and $\sigma \colon \Lambda^pV_\C\to
\Lambda^p\overline{V_\C}$ which we denote
by $w\mapsto \bar w$. The inverse of $\sigma $ is also denoted by
$\sigma \colon \overline{V_\C}\to V_{\C}$.
The antilinear maps $\sigma $ induce antilinear maps 
$\sigma \colon  V^{\ast}\to \overline {V}^{\ast}$ and $\sigma \colon
\overline{V}^{\ast}\to V^{\ast}$, that 
extend uniquely to an antilinear involution of the $\R$-algebra $\Lambda ^{\cdot,\cdot}V^{\ast}$.
This involution sends $\Lambda ^{p,q}V^{\ast}$ to $\Lambda^{q,p}V^{\ast}$. 
We continue to write $\sigma(\omega)=\bar\omega$ for complex forms $\omega$.
A complex form $\omega$ is called \emph{real} if $\bar \omega =\omega $.

We choose a real basis $e_{1},\dots ,e_{n}$ of $V$. 
We obtain dual complex bases 
$du_{1},\dots,du_{n}$
of $V^{\ast}$ and $d\bar u_{1},\dots,d\bar u_{n}$ of 
$\overline{V}^{\ast}$ determined by $du_i(e_j\otimes 1)=\delta_{ij}=
d\bar u_i(e_j\otimes 1)$.
Then we have 
\begin{displaymath}
\sigma (\lambda du_i ) = \bar \lambda d\bar u_i. 
\end{displaymath}

\begin{defi}
The \emph{canonical orientation} of the vector space $V_\C$
is the orientation 
 determined by the 
real form
\begin{displaymath}
\omega _{n}=du_{1}\wedge i d\bar u_{1} \wedge \dots \wedge  du_{n}\wedge i d\bar u_{n}\in \Lambda ^{n,n}V^{\ast}.
\end{displaymath} 
The form $\omega _{n}$ depends on our choice of 
a basis, but the orientation  does not.

An $\omega \in \Lambda ^{p,p}V^{\ast}$ is called \emph{strongly positive} if it
is in the convex cone spanned by 
\begin{displaymath}
\{\alpha _{1}\land i \bar
\alpha _{1}\land \dots \land
\alpha _{p}\land i \bar
\alpha _{p} \mid \text{$\alpha _{j}\in \Lambda ^{1,0}V^{\ast}$ for $j=1,\dots,p$}\}.
\end{displaymath}

An $\omega \in \Lambda ^{p.p}V^{\ast}$ is called \emph{positive} if it
belongs to the convex cone spanned by 
\begin{displaymath}
\{i^{p^{2}}\alpha \land  \bar \alpha \mid \alpha \in \Lambda ^{p,0}V^{\ast}\}.
\end{displaymath}

A complex $(p,p)$-form $\omega \in \Lambda ^{p,p}V^{\ast}$ is called
\emph{weakly positive} if, for every strongly positive form $\eta$ of
type $(n-p,n-p)$, there is a real number $\gamma \ge 0$ with
\begin{displaymath}
\omega \land \eta=\gamma \omega _n.
\end{displaymath}

We denote by $\Lambda^{p,p}_{+,s} V^*$, $\Lambda^{p,p}_{+} V^*$ and $\Lambda^{p,p}_{+,w} V^*$ the cones of 
strongly positive, positive and weakly positive $(p,p)$-forms respectively. 
\end{defi}

Observe that our weakly positive forms
are called positive in \cite[\S III.1.A]{demailly_agbook_2012}.

\begin{defi}\label{def:2}
To $\eta\in \Lambda ^{p,p}V^{\ast}$ we associate a sesquilinear form 
$|\eta|$ 
  on $\Lambda ^{p}V_{\C}$ by the rule
  \begin{displaymath}
    |\eta|(x,y) =  (-1)^{\frac{p(p-1)}{2}}i^{-p}\eta(x,\bar y).
  \end{displaymath}
Moreover, for $q=n-p$, there is a duality pairing
\begin{equation}\label{eq:31}
  \langle\,\cdot\,,\,\cdot\,\rangle\colon \Lambda ^{p,p}V^{\ast}\otimes \Lambda ^{q,q}V^{\ast}
  \longrightarrow \C
\end{equation}
defined  by
\begin{displaymath}
  \eta\land \omega = \langle \eta,\omega \rangle \omega _{n}.
\end{displaymath}
We denote by $\Lambda ^{p,p}V^{\ast}_{\R}$ the subspace of real
elements. Then the pairing \eqref{eq:31} induces a pairing
\begin{equation}\label{eq:35}
  \langle\,\cdot\,,\,\cdot\,\rangle_{\R}\colon \Lambda ^{p,p}V^{\ast}_{\R}\otimes
  \Lambda ^{q,q}V^{\ast}_{\R}\to \R. 
\end{equation}
\end{defi}

Note that the assignment $\eta \mapsto |\eta|$ is canonical and gives
an isomorphism between $\Lambda ^{p,p}V^{\ast}$ and the space of
sesquilinear forms on $\Lambda ^{p}V_{\C}$. On the other hand, the
duality pairing $\langle\cdot ,\cdot \rangle$ depends on the choice of basis but only up
to a non-zero positive number.

\begin{prop}\label{prop:11} A form $\eta\in \Lambda ^{p,p}V^{\ast}$ is
  real if and only if the sesquilinear form $|\eta|$ is Hermitian. A
  form $\eta$ is positive if and only if $|\eta|$ is a positive semidefinite Hermitian form. 
\end{prop}
\begin{proof}
  The antilinear involution $\sigma $ on $\Lambda ^{p,p}V^{\ast}$
  is given by
  \begin{displaymath}
     \sigma (\eta)(x,\xi)=\bar \eta(x,\xi)=(-1)^{p}\overline {\eta(\bar
       \xi,\bar x)}
   \end{displaymath}
  for $x\in \Lambda ^{p}V_\C$ and $\xi\in \Lambda ^{p}
  \overline{V_\C}$. Assume that $\eta$ is real. Then 
  for $x,y\in \Lambda ^{p}V_\C$ we have  
  \begin{multline*}
    \overline{|\eta|(x,y)}=
    \overline{(-1)^{\frac{p(p-1)}{2}}i^{-p}\eta(x,\bar y)}=
    (-1)^{\frac{p(p-1)}{2}}(-i)^{-p}\overline{\eta(x,\bar y)}\\=
    (-1)^{\frac{p(p-1)}{2}}(-i)^{-p}(-1)^{p}\bar \eta(y,\bar x) =
     (-1)^{\frac{p(p-1)}{2}}i^{-p}\eta(y,\bar x) =
     |\eta|(y,x).
  \end{multline*}
Thus, the sesquilinear form $|\eta|$ is hermitian. The converse is
proved analogously.

Assume now that $\eta\in \Lambda ^{p,p}V^{\ast}$ is a positive
form and $x\in \Lambda ^{p} V_\C$. Then we can write
\begin{displaymath}
  \eta= \sum_{j}i^{p^{2}}\gamma _{j}\alpha _{j}\land \bar \alpha
  _{j}, 
\end{displaymath}
with  $\gamma _{j}\in \R_{\ge 0}$ and  $\alpha _{j}\in \Lambda
^{p,0}V^{\ast}$. Then, since
$i^{p^{2}}(-1)^{\frac{p(p-1)}{2}}i^{-p}=1$, we have    
\begin{displaymath}
  |\eta|(x,x)=(-1)^{\frac{p(p-1)}{2}}i^{-p}\eta(x,\bar x)
  =\sum_{j}\gamma _{j}(\alpha _{j}\land \bar
  \alpha_{j})(x,\bar x)
  =\sum_{j}\gamma _{j}\alpha _{j}(x)\overline{\alpha _{j}(x)}
  \ge 0,
\end{displaymath}
proving that $|\eta|$ is positive semidefinite. 

Conversely, assume that
$|\eta|$ is positive semidefinite. Then by the spectral theory of
Hermitian forms, there are  $\gamma _{j} \in \R_{\geq 0}$ and
$\alpha _{j}\in \Lambda ^{p,0}V^{\ast}=(\Lambda ^{p}V_{\C})^{\ast}$
such that
\begin{displaymath}
  |\eta|=\sum \gamma _{j}\alpha _{j}\otimes \bar \alpha _{j}.  
\end{displaymath}
This implies that
\begin{displaymath}
  \eta=\sum i^{p^{2}}\gamma _{j}\alpha _{j}\land \bar \alpha _{j}
\end{displaymath}
showing that $\eta$ is a positive form.
\end{proof}

\begin{lem} \label{lemm:13}
For every $p\ge 0$, the complex vector space $\Lambda ^{p,p}V^{\ast}$
admits a $\C$-basis of strongly positive forms.
\end{lem}

\proof
\cite[Lemma III.1.4]{demailly_agbook_2012}.
\qed

\begin{lem} \label{lemm:12}
Any weakly positive form in  $\Lambda ^{p,p}V^{\ast}$ is real. 
\end{lem}
\begin{proof}
  The fact that $\omega _{n}$ and any positive or strongly positive
  form are real follows directly from the definition. Let now
  $\omega $ be a weakly positive form in $\Lambda
  ^{p,p}V^{\ast}$. This means that for each strongly positive form
  $\eta$ in $\Lambda^{q,q} V^*$, there is a non-negative real number
  $\gamma $ such that
\begin{displaymath}
\omega \wedge \eta = \gamma \omega _{n}.
\end{displaymath}
Since $\eta$, $\gamma $ and $\omega _{n}$ are real, this implies that
\begin{displaymath}
\bar \omega \wedge \eta = \gamma \omega _{n}.
\end{displaymath}
Hence $(\omega -\bar \omega)\wedge \eta =0 $ for any strongly positive
form $\eta$ in $\Lambda^{q,q} V^*$. By Lemma \ref{lemm:13},
$\Lambda ^{q,q}V^{\ast}$ admits a basis of strongly positive
elements. Therefore $(\omega -\bar \omega)\wedge \eta =0 $ holds for
any $\eta \in \Lambda ^{q,q}V^{\ast}$. By duality, we get
$\omega -\bar \omega=0$ and hence $\omega $ is real.
\end{proof}

\begin{cor}\label{cor:1} 
	For $p \in \N$, there are inclusions of closed convex cones
  \begin{displaymath}
  \Lambda^{p,p}_{+,s} V^* \subset \Lambda^{p,p}_{+} V^* \subset
  \Lambda^{p,p}_{+,w} V^* 
  \end{displaymath}
in  $\Lambda ^{p,p}V^{\ast}_{\R}$. 
For $q \coloneqq n-p$,  the cones  $\Lambda^{p,p}_{+,s} V^*$ and $\Lambda^{q,q}_{+,w}
V^*$ are dual to each other and
the cone $\Lambda^{p,p}_{+} V^*$ is the dual of $\Lambda^{q,q}_{+}
V^*$ with respect to the real duality pairing \eqref{eq:35}.
\end{cor}
\begin{proof}
  By definition, the spaces of strongly positive forms and of  positive
  forms are convex cones contained in $\Lambda ^{p,p}V^{\ast}_{\R}$.
   Lemma \ref{lemm:12} implies that
  $\Lambda^{p,p}_{+,w}V^*$ is contained in $\Lambda
  ^{p,p}V^{\ast}_{\R}$. Then $\Lambda^{p,p}_{+,w}V^*$ is a closed convex cone as the dual of the  convex cone $\Lambda^{p,p}_{+,w}V^*$.

  We next show that the convex cone of strongly positive forms is
  closed. Choose any  hermitian metric in the complex vector space
  $\Lambda ^{p}V^{\ast}$ and let $S\subset \Lambda ^{p}V^{\ast}$ be
  the unit sphere. The set 
  $K\subset S$ of totally decomposable elements is closed as it is the
  preimage of the Grassmanian {$\Gr(p,V^\ast)$} under the projection $S\to \P(\Lambda
  ^{p}V^{\ast})$. Since $S$ is compact, so is $K$. Let $K'\subset
  \Lambda ^{p,p}V^{\ast}$ be the image of $K$ under the continuous map
  \begin{displaymath}
    \Lambda ^{p}V^{\ast}\to \Lambda ^{p,p}V^{\ast},\quad \alpha
    \mapsto i^{p^{2}}\alpha \land \bar \alpha.
  \end{displaymath}
  Then $K'$ is a compact set that does not contain $0$. Thus the convex cone
  over $K'$ is closed
  \cite[Cor.~9.6.1]{rockafellar1970}. 
  Since the {convex} cone over $K'$ is
  $\Lambda^{p,p}_{+,s} V^*$, we deduce that the space of strongly
  positive forms is a closed convex cone. Being  $\Lambda^{p,p}_{+,s}
  V^*$ closed and convex, it agrees with its double dual. Therefore
  $\Lambda^{p,p}_{+,s} V^*$ is the dual of $\Lambda^{p,p}_{+,w} V^*$.  
  
  We next prove that $\Lambda^{p,p}_{+} V^*$ is self dual. Let $\gamma
  \in \Lambda ^{n}V^{\ast}$. Then $\gamma =zdu_{1}\wedge\dots\wedge
  du_{n}$ for some  $z \in \C$ and hence
  $i^{n^{2}}\gamma \wedge \bar \gamma=z\bar z \omega _{n}$. 
  For $\alpha \in \Lambda
  ^{p}V^{\ast}$ and $\beta \in \Lambda ^{q}V^{\ast}$, we conclude that
  \begin{displaymath}
    i^{p^{2}}i^{q^{2}}\alpha \wedge \bar \alpha \wedge
    \beta\wedge 
    \bar \beta = i^{n^{2}} \alpha \wedge \beta \wedge
    \overline{\alpha \wedge \beta} 
  \end{displaymath}
  is a positive multiple of $\omega _{n}$. Here we have used the identity
   $ i^{p^{2}}i^{q^{2}}(-1)^{pq}=i^{n^{2}}$.
  This proves that the cone $\Lambda^{p,p}_{+} V^*$ is contained the
  dual of $\Lambda^{q,q}_{+} V^*$. To prove the converse, 
  we introduce the {isomorphism} $\varphi\colon \Lambda ^{q}V^{\ast}\to \Lambda
  ^{p}V$ defined, for $\alpha \in \Lambda ^{q}V^{\ast}$, by
  \begin{displaymath}
    \beta \wedge \alpha  = \beta (\varphi(\alpha ))\,
    du_{1}\wedge\dots\wedge du_{n} \quad  \quad {(\beta \in \Lambda ^{p}V^{\ast})}. 
  \end{displaymath}
  Then, for any pair of forms $\eta\in \Lambda ^{p,p}V^{\ast}$ and $\alpha \in
  \Lambda ^{q}V^{\ast}$, the equality
  \begin{displaymath}
    \eta\wedge i^{q^{2}}\alpha \wedge \bar \alpha =
    |\eta|(\varphi(\alpha ),\varphi(\alpha )) \omega _{n}
  \end{displaymath}
  is satisfied. Therefore, if $\eta$ belongs to the dual cone of
  $\Lambda^{q,q}_{+} V^*$, then the sesquilinear form $|\eta|$ is
  positive semidefinite. By Proposition \ref{prop:11}, the form $\eta$
  is positive proving the reverse inclusion. We deduce that the cone 
  $\Lambda^{p,p}_{+} V^*$ is closed because it is a dual cone.
\end{proof}

\begin{rem} \label{rem:6} 
For $p=0,1,n-1,n$, the notions of strong positivity,
positivity and weakly positivity agree (see \cite[Corollary III.1.9]{demailly_agbook_2012}). 
{In \cite[Remark III.1.10]{demailly_agbook_2012},  there are 
examples of positive forms that are not strongly positive for any
$2\le p \le n-2$.}
\end{rem}

\subsection{The real situation}

We now shift to positivity of  Lagerberg forms following
\cite{lagerberg-2012}.  
We consider again a real basis $e_{1},\dots ,e_{n}$ of $V$ which
induces dual bases
 $ d'u_{1},\dots,d'u_{n}$ of $V'$ and $d''u_{1},\dots,d''u_{n}$
 of $V''$. 

We denote by $JV$ another copy of our $n$-dimensional $\R$-vector space $V$ and
 let us denote by $J \colon V \mapsto JV$ the identity map. 
There is a unique involution on $V \oplus JV$ that extends $J$ and which we also denote by $J$. 
From now on, we make the identification $V''\coloneqq \Hom(JV,\R)$ and 
then duality yields an involution on $V' \oplus V''$ which we also call $J$. 
There is a unique algebra homomorphism on $\Lambda^{\cdot,\cdot} V'$ that extends $J$. 
It is again an involution  mapping $\Lambda ^{p,q}V'$ onto $\Lambda ^{q,p}V'$. 
This map, also denoted by $J$, is called the \emph{Lagerberg involution}.

\begin{defi} \label{Lagerberg involution and orientation}
The \emph{Lagerberg
    orientation} is the orientation on the vector space
    $V'\oplus V''$ defined by 
  \begin{displaymath}
    \tau _{n}=d'u_{1}\wedge d'' u_{1}\wedge \dots \wedge
    d'u_{n}\wedge d'' u_{n} \in \Lambda^{n,n}(V).
  \end{displaymath}
Consider a form $\omega \in \Lambda ^{p,p}V'$. 
We call $\omega$ \emph{symmetric} if
  \begin{equation}
    \label{eq:24}
    J(\omega)=(-1)^{p}\omega. 
  \end{equation}
We call $\omega$ \emph{strongly positive} if it belongs to
  the convex cone spanned by
  \begin{displaymath}
    \{\alpha _{1}\land J(
    \alpha _{1})\land \dots \land
    \alpha _{p}\land J(
    \alpha _{p})\mid \text{$\alpha _{j}\in \Lambda ^{1,0}V'$ for $j=1, \dots, p$}\}.
  \end{displaymath}
We call $\omega$ \emph{positive} if it belongs to
  the convex cone spanned by 
  \begin{displaymath}
    \{(-1)^{\frac{p(p-1)}{2}}\alpha \land J(
    \alpha ) \mid \alpha \in \Lambda ^{p,0}V'\}.
  \end{displaymath}

  A symmetric Lagerberg $(p,p)$-form $\omega \in \Lambda ^{p,p}V'$ is called
  \emph{weakly positive} if for every strongly positive form $\eta$ of
  type $(n-p,n-p)$, there is a
  real number $\gamma \ge 0$ with
  \begin{displaymath}
    \omega \land \eta =\gamma \tau_n.
  \end{displaymath}
  We will also denote as $\Lambda^{p,p}_{+,s} V'$, $\Lambda^{p,p}_{+} V'$ and
$\Lambda^{p,p}_{+,w} V'$ the spaces of Lagerberg $(p,p)$-forms that are strongly
positive forms, positive and weakly positive respectively.
\end{defi}

\begin{rem} \label{rem:5}
Strongly
positive, positive or weakly positive Lagerberg forms and $\tau_n$ are symmetric. 
In the definition of weakly positive Lagerberg forms, it is
necessary to impose the symmetry because the real analogue of 
Lemma \ref{lemm:13} is not true,  see Example \ref{too few strongly positive}.

{It follows from Corollary
\ref{cor:2} below} that our definition of positive Lagerberg forms 
agrees with the definition given by Lagerberg in \cite[Definition
2.1]{lagerberg-2012}.

From 
our definition, we deduce that the product of positive Lagerberg forms 
is  positive.
\end{rem}

\begin{defi}\label{def:4}
  To $\eta\in \Lambda ^{p,p}V'$, we associate a bilinear form $|\eta|$
  on $\Lambda ^{p}V$ by the rule
  \begin{displaymath}
    |\eta|(x,y) =  (-1)^{\frac{p(p-1)}{2}}\eta(x,J(y)).
  \end{displaymath}

Moreover, for $q=n-p$, there is a duality pairing
\begin{displaymath}
  \langle\,\cdot\,,\,\cdot\, \rangle\colon \Lambda ^{p,p}V'\otimes \Lambda ^{q,q}V'\longrightarrow \R, \quad 
\end{displaymath}
defined  by
$ \eta\land \omega = \langle \eta,\omega \rangle \tau _{n}$ {for $\eta \in \Lambda ^{p,p}V'$ and $\omega \in \Lambda ^{q,q}V'$.}

\end{defi}

Note that, again, the assignment $\eta \mapsto |\eta|$ is canonical and gives
an isomorphism between $\Lambda ^{p,p}V'$ and the space of
bilinear forms on $\Lambda ^{p}V$. On the other hand, the
duality pairing
$\langle\cdot ,\cdot \rangle$ depends on the choice of a basis, but only up
to a positive number.

\begin{prop}\label{prop:12} A form $\eta\in \Lambda ^{p,p}V'$ is
  symmetric if and only if the bilinear form $|\eta|$ is symmetric. A
  form $\eta$ is positive if and only $|\eta|$ is a 
  positive semidefinite symmetric form. 
\end{prop}
\begin{proof}
  The proof is similar to the complex case and is given in
  \cite[Proposition 2.1]{lagerberg-2012}
\end{proof}

Denote now by $\Lambda_{\sym}^{p,p}V'$ the subspace of
symmetric elements and let  $q \coloneqq n-p$. The duality pairing of
Definition \ref{def:4} induces a real
duality pairing, denoted by the same symbol,
\begin{displaymath}
  \langle\cdot,\cdot \rangle  \colon  \Lambda_{\sym}^{p,p}V'\otimes \Lambda^{q,q}_{\sym}V' \to \R.
\end{displaymath}

{We  give  the  analogue of Corollary \ref{cor:1} which was stated before \cite[Lemma 2.2]{lagerberg-2012}.}

\begin{cor}\label{cor:2} For $p \in \N$, there are inclusions of closed convex cones
  \begin{displaymath}
    \Lambda^{p,p}_{+,s} V' \subset \Lambda^{p,p}_{+} V' \subset
    \Lambda^{p,p}_{+,w} V'
  \end{displaymath}
in $ \Lambda_{\sym}^{p,p}V'$. For $q \coloneqq n-p$, the cones  $\Lambda^{p,p}_{+,s}
V'$ and $\Lambda^{q,q}_{+,w} V'$  
are dual to each other and the cone $\Lambda^{p,p}_{+} V'$ is the
dual of $\Lambda^{q,q}_{+} V'$ with respect to the above real duality pairing.
\end{cor}

\begin{proof}
{The arguments are as in Corollary \ref{cor:1} replacing complex by real numbers 
  	and sesquilinearforms by symmetric bilinear forms.}
\end{proof}

\begin{rem} \label{rem:7} As in the complex case, strong positivity agrees with positivity and weak positivity for $p=0,1,n-1,n$.
	Similarly as   in \cite[Example III 
  1.10]{demailly_agbook_2012}, there are 
  positive Lagerberg forms that are not strongly positive for any
  $2\le p \le n-2$.
\end{rem}

\subsection{Comparison between the real and complex situations}
\label{comparison-vector-space-section}

We aim for an identification 
of the space of Lagerberg forms $\Lambda ^{\cdot,\cdot}V'$
with a subspace of the space of complex forms 
$\Lambda ^{\cdot,\cdot}V^{\ast}$ that preserves positivity 
as much as possible.

\begin{defi}\label{def:5}
The vector space $V^{\ast}=\Hom_{\R}(V,\C)$ has an antilinear
involution $F$ coming from complex conjugation in $\C$. We extend
$F$ to $V^{\ast}\oplus \overline {V}^{\ast}$ in such a way that $F$
and $\sigma $ (the complex conjugation from \ref{subsection: The complex situation}) anticommute:
\begin{equation}\label{eq:12}
  F(\bar \alpha )=-\overline {F(\alpha )}
  \quad (\text{for }\alpha\in {V^*} ).
\end{equation}
{There is a unique antilinear involution of the $\R$-algebra  $\Lambda ^{\cdot,\cdot}V^{\ast}$ which extends $F$. We denote this extension also by $F$.}   
\end{defi}

{Note that $F$ induces an antilinear involution of any  $\Lambda ^{p,q}V^{\ast}$.} In coordinates, we have 
\begin{equation}\label{def:555}
F(\lambda)=\bar\lambda,\,\,\,\,\,
F (du_i ) = du_i \quad\text{ and } \quad
F (d\bar u_i ) = - d \bar u_i 
\end{equation}
where $\lambda\in \Lambda^{0,0}V^\ast=\C$.

Note that, if $\eta\in \Lambda ^{p,q}V^{\ast}$, then
\begin{equation}\label{eq:13}
  \overline {F(\eta )}=(-1)^{p+q}F(\bar \eta ).
\end{equation}

We see $V'\subset V^{\ast}$  and $V''\subset \overline{V}^{\ast}$  as
the subspaces of $F$-invariant elements. Note that $V'$ looks like the space of 
real elements of $V^{\ast}$, while, due to the twisted definition
\eqref{eq:12} of $F$ in $\overline{V}^{\ast}$, the elements of $V''$ look like the
\emph{imaginary} elements of $\overline{V}^{\ast}$.

We extend the inclusion $V'\oplus V'' \to
V^{\ast}\oplus \overline{V}^{\ast}$ to an {$\R$-algebra homomorphism}
\begin{equation}\label{eq:15}
   \Lambda ^{\cdot,\cdot}V'\longrightarrow  \Lambda ^{\cdot,\cdot}V^{\ast}.
 \end{equation}
 This inclusion sends $d'u_{j}$ to $du_{j}$ and $d''u_{j}$ to
 $id\bar u_{j}$. 

\begin{prop}\label{prop:13}
  We have the following compatibilities for the inclusion \eqref{eq:15}. 
  \begin{enumerate}
    \item The space of complex forms invariant
    under $F$ is the image of $\Lambda ^{\cdot,\cdot}V'$. 
    \item     
    The Lagerberg involution $J$ agrees on {$V' \oplus V''$} 
        with the map $\alpha \mapsto i\bar \alpha $.  
    \item \label{item:18} A Lagerberg $(p,p)$-form is symmetric if and
      only if it is real as a complex form.
    \item The image of $\tau _{n}$ 
      is $\omega _{n}$. 
\end{enumerate}
\begin{proof}
  Since $F$ is an antilinear involution, the space of $F$-invariant
  elements of $\Lambda ^{\cdot,\cdot}V^{\ast}$ has real dimension
  $2^{2n}$ which agrees with the dimension of $\Lambda
  ^{\cdot,\cdot}V'$. Since, by construction, the Lagerberg forms are
  invariant under $F$, both spaces agree. The remaining statements are
  direct computations. 
\end{proof}
\end{prop}

\begin{lem} \label{lemm:15}
The involution $F$ maps strongly positive (resp.~weakly
positive, resp.~positive) complex forms to strongly positive
(resp.~weakly positive, resp.~positive) complex forms. 
\end{lem}

\begin{proof} 
For any complex $(p,0)$-form $\alpha$, antilinearity of $F$ and \eqref{eq:13} give
\begin{align} \label{formula with F}
F(\alpha \wedge i^p \bar \alpha) = F\alpha \wedge (-i)^p F(\bar
\alpha)=F\alpha \wedge i^p \overline{ F(\alpha)}.
\end{align}
We conclude that $F$ preserves  positivity of
complex forms.
	Multiplicativity of $F$ and \eqref{formula with F}  show also that $F$
preserves strong positivity of forms.  
	
Since  $\omega_n$ is the image of  $\tau_n$ under the above identification, it is
fixed under $F$. Using that $F$ is an involution, we deduce from
duality that $F$ preserves weak positivity as well.
\end{proof}

\begin{prop} \label{prop:15} The following compatibility conditions hold. 
\begin{enumerate}
\item \label{simpl} A strongly positive Lagerberg form is also a strongly positive
  complex form.
\item \label{wimpl} An $F$-invariant weakly positive complex form is a
  weakly positive Lagerberg form.
\end{enumerate}  
\end{prop}
This follows easily from the definitions. 
For positive forms, we have a stronger result.

\begin{prop}\label{prop:14}
    A Lagerberg form is positive if and only if it is positive as a complex form.
\end{prop}
\begin{proof}
It is easily seen that if $\omega\in \Lambda ^{p,p}V'$ 
is a positive Lagerberg form, then it is also  a positive complex form 
by using the embedding \eqref{eq:15}
 and that $J\omega = i^p \bar \omega$.

Conversely, assume that $\omega$ is a Lagerberg $(p,p)$-form that is positive as a complex form.  
We claim that is enough to show that if
$\eta = (-1)^{\frac{p(p-1)}{2}}\alpha \wedge i^p \bar \alpha$ for some
complex $(p,0)$-form $\alpha$,
then $\eta +F(\eta )$ is a positive Lagerberg form. 
Assuming this claim, we can write 
  \begin{displaymath}
    \omega=\sum_{s}\gamma _{s} (-1)^{\frac{p(p-1)}{2}}\alpha_{s}
    \wedge i^p \bar  \alpha_{s} 
\end{displaymath}
with {$\gamma_s \in \R_{\geq 0}$ and $\alpha_s \in \Lambda^{p,0}V^*$.} Using the claim and the $F$-invariance of $\omega$, we have that 
\begin{displaymath}
  \omega =\frac{1}{2}(\omega  + F(\omega ))=
  \sum_{s} \frac{\gamma _{s}}{2} (-1)^{\frac{p(p-1)}{2}} (\alpha_{s}
  \wedge i^p \bar
  \alpha_{s}+ F(\alpha_{s} \wedge i^p \bar \alpha_{s}))
\end{displaymath}
is a positive Lagerberg form.

To prove our claim, we write $\alpha = a + ib$ with $a,b\in \Lambda^{p,0}(V')$. Then we have  
\begin{displaymath}
  \alpha \wedge i^p \bar \alpha = (a+ib) \wedge i^p(\bar a -
  i \bar b)= i^p(a \wedge \bar a - i a \wedge \bar b +
  i b \wedge \bar a + b \wedge  \bar b).
\end{displaymath}
Using $F(\alpha)=a-ib$, we get
\begin{displaymath}
  F\alpha \wedge F(i^p \bar\alpha)) = (a-ib) \wedge F(i^p(\bar a - i
  \bar b)) = i^p(a-ib) \wedge (\bar a +i \bar b),
\end{displaymath}
where we have used for the last equality that 
\begin{displaymath}
  F(i^p\bar a)=(-1)^{p}i^{p}F(\bar a)
  =(-1)^{p}i^{p}(-1)^{p}\overline{F(a)}= i^p\bar a
\end{displaymath}
by equation \eqref{eq:13} and  analogously $F(i^{p+1}\bar b)
= -i^{p+1}\bar b$. Since
\begin{displaymath}
  \frac{1}{2}(\eta + F(\eta))=
   (-1)^{\frac{p(p-1)}{2}}\frac{1}{2} \left( \alpha
    \wedge  i^p\bar \alpha +  F\alpha \wedge F(i^p
    \bar \alpha)) \right),
\end{displaymath}
we deduce that
\begin{displaymath}
\eta +F(\eta )  = 2 \cdot (-1)^{\frac{p(p-1)}{2}}i^p(a \wedge \bar a +
b \wedge\bar b) = 
2 \cdot (-1)^{\frac{p(p-1)}{2}} (a \wedge J a + b \wedge J b) 
\end{displaymath}
which is a positive Lagerberg form.
\end{proof}

The next example shows that the analogues of 
Proposition \ref{prop:14} for strongly and weakly  
positive forms do not hold. 
Hence the converse of Proposition \ref{prop:15} \ref{simpl} and \ref{wimpl} is wrong.

\begin{ex} \label{too few strongly positive}
We show that the space of strongly 
positive Lagerberg forms does not span the space of 
symmetric Lagerberg forms. 
In particular, the cone of strongly positive Lagerberg 
forms has empty interior and 
not every symmetric Lagerberg form can be written as a 
difference of strongly positive ones. 

Indeed, let $V$ be a real vector space of dimension four. 
Then we have
\begin{displaymath}
  \Lambda^{2,2}_{\sym} V'=\Bigl\{\sum_{i<k, j<l}
  w_{ijkl}d'u_i \wedge d''u_j \wedge d'u_k \wedge d''u_l
  \,\Big|\,  w_{ijkl}=w_{jilk}\Bigr\}.
\end{displaymath}
Moreover, 
we have the following identity of cones:
\begin{align*}
\Lambda^{2,2}_{+,s} V' = \langle  a \wedge J(a) \wedge b \wedge J(b)
  \mid a,b \in \Lambda^{1,0} V'  \rangle_{\R_{\geq 0}}.
\end{align*}
Given $a,b \in \Lambda^{1,0} V'$, let $w_{ijkl}(a,b) \in \R$ be such that
\begin{align*}
a \wedge J(a) \wedge b \wedge J(b) = \sum_{i<k, j<l}
  w_{ijkl}(a,b) d'u_i \wedge d''u_j \wedge d'u_k \wedge d''u_l.
\end{align*}
A direct computation shows  
\begin{align*}
  w_{1,3,2,4}(a,b) - w_{1,2,3,4}(a,b) + w_{1,2,4,3}(a,b) = 0,
\end{align*}
and, by symmetry, we have
  \begin{displaymath}
    w_{3,1,4,2}(a,b) - w_{2,1,4,3}(a,b) + w_{2,1,3,4}(a,b) = 0,
  \end{displaymath}
Hence 
$\Lambda^{2,2}_{+,s}V'$ is in a proper linear subspace of 
$\Lambda^{2,2}_{\sym} V'$ and so it has empty interior.

As a consequence, the cone {$\Lambda^{2,2}_{+,w}V'$}  is not {strictly} convex. Namely, the form
\begin{align*}
  \omega  =
    &d'u_3 \wedge d''u_1\wedge d'u_4 \wedge d''u_2
   - d'u_2 \wedge d''u_1\wedge d'u_4 \wedge d''u_3
   +d'u_2 \wedge d''u_1\wedge d'u_3 \wedge d''u_4\\
   +&d'u_1 \wedge d''u_3\wedge d'u_2 \wedge d''u_4
   - d'u_1 \wedge d''u_2\wedge d'u_3 \wedge d''u_4
    +d'u_1 \wedge d''u_2\wedge d'u_4 \wedge d''u_3
\end{align*}
{satisfies $\eta\wedge \omega =0$ for every  $\eta \in\Lambda^{2,2}_{+,s}V'$, and hence 
 $C\omega \in \Lambda^{2,2}_{+,w}V'$ for every $C \in \R$.}

\end{ex}

\begin{rem} \label{strongly positive not compatible}
We will deduce from Example \ref{too few strongly positive}  the existence of Lagerberg forms that are strongly positive 
as complex forms but not strongly positive as Lagerberg forms. 

By Lemma \ref{lemm:13},  every complex form $\omega$ can be written as a complex
linear combination
\begin{align} \label{strongly positive not compatible eq}
\omega = \sum_i \lambda_i \omega_i 
\end{align}
 of strongly positive complex forms $\omega_i$ . 
If $\omega$ is real, then applying $\sigma$ to (\ref{strongly positive not compatible eq}) and using that the $\omega_i$ are real, 
we see that we may take the $\lambda_i$ to be real. 
If $\omega$ is further invariant under $F$, then applying $F$ to (\ref{strongly positive not compatible eq}) 
we see that replacing $\omega_i$ by $1/2 (\omega_i + F(\omega_i))$ we may assume 
the $\omega_i$ to be $F$-invariant. 
Note that $F(\omega_i)$ is strongly positive by Lemma \ref{lemm:15}.

We have just shown $\Lambda^{p,p}_{\sym} V' = \langle (\Lambda^{p,p}_{+,s} V^*)^{F} \rangle_\R$. 
Since we showed in Example \ref{too few strongly positive} that 
$\langle \Lambda^{p,p}_{+,s} V' \rangle_{\R} \subsetneq \Lambda^{p,p}_{\sym} V'$,  
we find that $(\Lambda^{p,p}_{+,s} V^*)^{F} \not\subset \Lambda^{p,p}_{+,s} V'$. 
By Proposition \ref{prop:13}, this means exactly that not every Lagerberg form
that is strongly positive as a complex form is strongly positive as a Lagerberg form. 

The next example illustrates this phenomenon.  
\end{rem}

\begin{ex} \label{strongly positive not compatible explicit}
Let $V$ still be a real vector space of dimension four. 
Choose a basis $e_{1},\dots,e_{4}$ and corresponding bases of $V^{\ast}$,
$\overline {V}^{\ast}$, $V'$ and $V''$ as before. The complex form
\begin{displaymath}
\eta = (du_{1}+idu_{2})\land i(d\bar u_{1}-id\bar u_{2})\land
(du_{3}+idu_{4})\land i(d\bar u_{3}-id\bar u_{4})
\end{displaymath}
is strongly positive by definition. 
By Lemma \ref{lemm:15}, the form $\omega = \frac{1}{2}(\eta + F(\eta))$ is strongly positive. 
Moreover, it is $F$-invariant and hence $\omega$ may be seen as  a Lagerberg form by Proposition \ref{prop:13}. 
We claim that $\omega$ is not strongly positive as a Lagerberg form.

A direct computation shows that
\begin{align*}
\omega = & 
d'u_{1}\land d''u_{1}\land d'u_{3}\land d''u_{3}
+  d'u_{1}\land d''u_{1}\land d'u_{4}\land d''u_{4}\\
& + d'u_{2}\land d''u_{2}\land d'u_{3}\land d''u_{3}
+  d'u_{2}\land d''u_{2}\land d'u_{4}\land d''u_{4}\\
&-d'u_{1}\land d''u_{2}\land d'u_{3}\land d''u_{4}
+ d'u_{2}\land d''u_{1}\land d'u_{3}\land d''u_{4}\\
&+ d'u_{1}\land d''u_{2}\land d'u_{4}\land d''u_{3}
-d'u_{2}\land d''u_{1}\land d'u_{4}\land d''u_{3}.\\
=&(-1)(d'u_{1}\land d'u_{3}-d'u_{2}\land d'u_{4})\land
J(d'u_{1}\land d'u_{3}-d'u_{2}\land d'u_{4})\\
&+(-1)(d'u_{1}\land d'u_{4}+d'u_{2}\land d'u_{3})\land
J(d'u_{1}\land d'u_{4}+d'u_{2}\land d'u_{3}).
\end{align*}
This shows that $\omega $ is a  positive Lagerberg form, that the
associated symmetric 
bilinear form $|\omega |$ has rank 2 and that $(\ker |\omega
|)^{\perp}$ is the 2-dimensional subspace 
\begin{displaymath}
(\ker |\omega |)^{\perp}=\R(d'u_{1}\land d'u_{3}-d'u_{2}\land
d'u_{4}) + \R(d'u_{1}\land d'u_{4}+d'u_{2}\land d'u_{3}).
\end{displaymath}
We pick any decomposition
\begin{equation}\label{eq:16}
\omega =\sum_j \alpha _{j}\land J(\alpha _{j})
\end{equation}
with $\alpha \in \Lambda^{2,0} V'$. 
Then  $\alpha _{j}\in (\ker |\omega |)^{\perp}$, since for any 
$v\in \ker |\omega | \subset \Lambda ^{2}V$, we have 
\begin{displaymath}
0=|\omega |(v,v)=\sum_{j}\langle \alpha _{j},v\rangle ^{2}.
\end{displaymath}

If $\omega $ were strongly positive, we would have a decomposition
like \eqref{eq:16} where $\alpha _{j}\in  \Lambda ^{2,0}V'$ is 
{a product of $(1,0)$-forms}.
However, this is not possible because $(\ker |\omega |)^{\perp}$ 
does not contain any non-zero real decomposable element as the
following argument shows. 
Assume that
\begin{align*}
\rho \coloneqq (\alpha\,d'u_{1}+\beta \,d'u_{2}+\gamma\,d'u_{3}+\delta\,d'u_{4})\land (\alpha'\,d'u_{1}
+ \beta'\,d'u_{2}+\gamma'\,d'u_{3}+\delta'\,d'u_{4}) \in (\ker |\omega |)^{\perp}.
\end{align*}
This implies the equations
\begin{align*}
\alpha\gamma'-\alpha'\gamma & = \beta'\delta-\beta\delta',\\
\alpha\delta'-\alpha'\delta &=\beta\gamma'-\beta'\gamma,\\
\alpha\beta'-\alpha'\beta&=0,\\
\gamma\delta'-\gamma'\delta&=0.
\end{align*}
The point $\rho$ determines a point $\rho'$ in the Grassmannian $\Gr(2,4)$
with Pl\"ucker coordinates
\begin{alignat*}{3}
x&=
    \begin{vmatrix}
      \alpha & \beta\\
      \alpha' & \beta'
    \end{vmatrix},
    &\quad
    y&=
    \begin{vmatrix}
      \alpha & \gamma\\
      \alpha' & \gamma'
    \end{vmatrix},
    &\quad
    z&=
    \begin{vmatrix}
      \alpha & \delta\\
      \alpha' & \delta'
    \end{vmatrix},
    \\
    u&=
    \begin{vmatrix}
      \beta & \gamma\\
      \beta' & \gamma'
    \end{vmatrix},
    &
    v&=
    \begin{vmatrix}
      \beta & \delta\\
      \beta' & \delta'
    \end{vmatrix},
    &
    w&=
    \begin{vmatrix}
      \gamma & \delta\\
      \gamma' & \delta'
    \end{vmatrix}.
  \end{alignat*}
  The previous equations imply that the Pl\"ucker coordinates 
  of $\rho'$ satisfy the equations
  \begin{displaymath}
    y=-v,\quad z=u,\quad x=w=0.
  \end{displaymath}
  Moreover, the Pl\"ucker equations for $\Gr(2,4)$ are reduced to the
  single equation
  \begin{displaymath}
    xw-yv+zu=0.
  \end{displaymath}
  We conclude that the Pl\"ucker coordinates of $\rho'$ satisfy the
  equation
  \begin{equation}\label{eq:17}
    y^{2}+z^{2}=0
  \end{equation}
  that has no real solutions except the trivial one. The fact that
  $\omega $ is strongly positive as a complex form is reflected  
  by the
  fact that  \eqref{eq:17} has non-trivial complex
  solutions.   
\end{ex}

\section{Lagerberg forms and Lagerberg currents on 
partial compactifications}
\label{section smooth forms tropical}

For convex geometry we will use the notation and conventions
set up in \cite[Appendix]{gubler-kuenne2017}.
Let $N$ be a free abelian group of rank {$n$}, 
$M=\Hom_\Z(N,\Z)$ its dual and 
denote by $N_\R$ resp.~$M_\R$ the respective scalar extensions to $\R$.

\subsection{Partial compactifications}
\label{subsection partial compactification}

A strictly convex rational 
polyhedral cone $\sigma \in N_\R$ is a polyhedron defined by
finitely many equations of the form $\varphi(\,.\,) \geq 0$ with $\varphi \in M$,  
that does not contain a positive dimensional linear subspace. 
A rational polyhedral fan $\Sigma$ in $N_\R$ is a 
polyhedral complex all of whose polyhedra are  
strictly convex rational cones. 
In this paper we make the convention that 
a \emph{fan} is always a rational polyhedral fan. 

A cone is called \emph{smooth} if it is generated by a subset of
a $\Z$-basis of $N$. A fan $\Sigma$ is called smooth 
if each cone of $\Sigma$ is smooth. 
For $\sigma \in \Sigma$ we define the monoid 
${S_\sigma}\coloneqq \{ \varphi \in M | \; \varphi(v) \geq 0
\text{ for all } v \in \sigma \}$. 
For $\sigma \in \Sigma$, write 
$N(\sigma)\coloneqq N_\R / \langle\sigma\rangle_\R$
where $\langle\sigma\rangle_\R$ denotes the real vector space generated by $\sigma$. 
Given $\sigma,\tau\in \Sigma$, we write $\tau\prec\sigma$
if $\tau$ is a face of $\sigma$.
We have projection maps $\pi_{\sigma} \colon N_{\R}
\to N(\sigma)$ and $\pi_{\sigma, \tau} \colon N(\tau)
\to N(\sigma)$ for $\tau \prec \sigma$.

\begin{defi}\label{definition partial compactification}
Let $\Sigma \subset N_\R$ be a rational 
polyhedral fan. 
We consider the disjoint union
\begin{align*}
N_\Sigma \coloneqq  \coprod \limits_{\sigma \in \Sigma } N(\sigma)
\end{align*}
and call $N_\Sigma$ equipped with the topology introduced
in Remark \ref{def-topology-partial-compactification} the 
\emph{partial compactification of $N_\R$
associated to $\Sigma$}.
\end{defi}

\begin{rem}\label{def-topology-partial-compactification}
The partial compactification $N_\Sigma$ carries the following
topology which is Hausdorff.
It is also locally compact and has a countable basis and hence it is metrizable. 
Let us briefly recall its definition.
		
First, we define the partial compactification of $N(\sigma)$ 
for a single cone $\sigma \in \Sigma$ by setting
\[
{N_{\sigma}} 
\coloneqq \coprod \limits_{{\tau \prec \sigma}} N({\tau}).
\] 
The set ${N_\sigma}$ is naturally identified with the monoid homomorphisms
$\Hom_{\Mon}(S_\sigma, \Rsup)$. We equip it with the subspace topology of $\Rsup^{S_\sigma}$. 
Using a finite set of generators {$\varphi_1, \dots, \varphi_k$} for the monoid $S_\sigma$, we can 
realize $\Hom_{\Mon}(S_\sigma, \Rsup)$ as a closed subspace of $\Rsup^k$ with the induced topology 
(see \cite[Remark 3.1]{payne-2009a} and use \cite[Lemma 2.1]{payne-2009a}).

For a face $\rho$ of $\sigma$, we note that  ${N_\rho}$ is an open subset 
of ${N_\sigma}$. 
This is used to define a topology on the partial compactification 
$N_\Sigma$ by gluing the partial compactifications 
${N_\sigma}$, $\sigma \in \Sigma$, along the open subsets 
induced by common faces.

We give a second description of the topology of $N_{\Sigma }$.
To this end, we fix an Euclidean metric in $N_{\R}$. 
For a cone $\nu$ of $\Sigma$, the Euclidean metric allows us
to identify
$\nu^{\perp}$ with a subspace of $N_{\R}$ and, through the
projection $\pi _{\nu}$, with the space $N(\nu)$.
Again, we consider first the case of a single cone $\sigma
\in \Sigma$.  
For  a point $u\in {N_\sigma}$, 
there is a unique face $\nu$ of $\sigma$ with $u \in N(\nu)$.    
Let $u_{0}\in \nu^{\perp}\subset N_\R$ be the corresponding point and 
let $U$ be a neighborhood of $u_0$ in $\nu^{\perp}$. 
For each face $\tau \prec \nu$, the cone $\nu$
induces a cone $\pi _{\tau }(\nu)$ contained in $N(\tau)$. 
For each $p\in \nu$, we write  
\begin{equation}\label{eq:20}
	W(\nu ,U,p)=\coprod _{\tau  \prec \nu  }\pi _{\tau }(U+p+\nu).
\end{equation}
The topology of ${N_\sigma}$ is defined by the fact that 
$\{W(\nu,U,p)\}_{U,p}$ is a basis of neighbourhoods of 
$u$ in ${N_\sigma}$ for any $u \in {N_\sigma}$.
As before the topology of $N_\Sigma$ is defined by gluing
along the open subsets ${N_\tau}$  of ${N_\sigma}$
whenever $\tau\prec\sigma$. 

The first definition of the topology is 
given by Kajiwara \cite{kajiwara2008} and by Payne \cite{payne-2009a},  
the second definition is from \cite[I.1]{AMRT}. 
The topologies coincide as the above basis of neighbourhoods 
works also for the first definition by \cite[Remark 3.4]{payne-2009a}. 

To prove that $N_{\Sigma }$ is Hausdorff, we use that the quotient of a 
topological space (in this case
the disjoint union of the $N_{\sigma }$) by an equivalence relation is
Hausdorff if 
the canonical map to the quotient is open and 
the graph of the equivalence relation is closed
\cite[Ch.~I \S 8.3 Prop.~8]{bourbaki-topologie-generale-1-4}.  
As the map to the quotient is open
by construction, it is enough to show that, for cones $\sigma_{1}$ and  
$\sigma _{2}$ with $\tau =\sigma_{1}\cap \sigma_{2}$,
the map
\begin{displaymath}
N_{\tau}\longrightarrow N_{\sigma_{1}}
\times  N_{\sigma_2}
\end{displaymath}
is a closed immersion. This follows easily from the first description
of the topologies of $N_{\sigma_{1}}$ and $N_{\sigma_2}$ by choosing a
finite set of generators of $S_{\sigma _{1}}$ and $S_{\sigma_2}$ and
observing that the union of both sets is a set of generators of
$S_{\tau }$. 

Note that $N_{\Sigma }$
is locally compact because the $N_{\sigma }$
provide an open covering of $N_{\Sigma }$ and each of them is locally
compact. 
Finally every $N_\sigma$ has a countable basis and hence also $N_\Sigma$.
\end{rem}

\begin{rem}\label{ass-toric-variety}
Let $\torus=\Spec\, \C[M]$ be the split complex torus with cocharacter lattice $N$.
Let $X_\Sigma$ denote the toric variety over $\C$ 
with dense torus $\torus$ determined by 
the  fan $\Sigma$ in $N_\R$. 
Let $X_\Sigma^\an$ denote the analytification of $X_\Sigma$,
i.e.~the set of complex points $X_\Sigma(\C)$ with its structure of an
analytic space.

There is a well-known continuous map (see for example \cite[I.1, p.2]{AMRT}, 
\cite[Definition 1.2]{kajiwara2008} or \cite[Section 4.1]{bps-asterisque}) 
\begin{align}\label{ass-toric-variety-eq1}
\trop \colon X_\Sigma^{\an} \longrightarrow N_{\Sigma},
\end{align}
which is nowadays called \emph{tropicalization map}
as it is given by glueing on the affine open subsets 
${U_\sigma=\Spec\, \C[S_\sigma]}$ for $\sigma\in \Sigma$ of $X_\Sigma$
the tropicalization maps
\begin{align*}
\trop \colon U_\sigma^{\an} \longrightarrow 
{N_\sigma} = \Hom_{\Mon}(S_\sigma, \Rsup),\,\,
y \longmapsto (m \longmapsto - \log \vert \chi^m (y)\vert)
\end{align*}
where $\chi^m\colon \T\to \G_m$ is the character associated with $m$.
It follows from {\cite[Sections 4.1,  4.2]{bps-asterisque}} that the tropicalization 
map \eqref{ass-toric-variety-eq1}
is a proper continuous map that identifies 
$N_{\Sigma }$ with $\Xsigmaan/\SS$
where $\SS$ denotes the real compact torus
\[
\SS=\bigl\{p\in \T^\an\,\big|\,|\chi^m(p)|=1\,\mbox{for all}\,m\in
M\bigr\}\subset \T^\an.
\]
There is a continuous proper section
$\rho_\Sigma\colon N_\Sigma\rightarrow X_\Sigma^\an$
of the tropicalization map \eqref{ass-toric-variety-eq1}
given as the unique continuous extension
of the section
\[
\rho\colon N_\R\longrightarrow \T^\an=\Hom(M,\C^\times),\,\,\,
n\longmapsto\bigl[m\longmapsto \exp(-\langle m,n\rangle)\bigr]
\]
of \eqref{ass-toric-variety-eq1} (see \cite[Remark 4.1.3]{bps-asterisque}). 
\end{rem}

{To illustrate the topology on the partial compactification $N_\Sigma$, we give 
the following lemma where the notation coincides with the one 
in \cite[I.1, p.5]{AMRT}.}

\begin{lem}\label{lemma definition infinity}
Let $\Sigma \subset N_\R$ be a  fan, and let 
$N_\Sigma$ be its associated partial compactification.
Given $p \in N_\R$ and $v \in |\Sigma|$, the limit 
$p + \infty v \coloneqq \lim_{\mu \to + \infty} p + \mu v $ exists in $N_\Sigma$.
Moreover, $p + \infty v \in N(\sigma)$ for the unique cone  $\sigma \in \Sigma$ such that 
$v \in \relint(\sigma)$.
\end{lem}

\begin{proof}
We use the description of the topology of $N_{\Sigma }$ given by the
basis of neighborhoods $W(\sigma ,U,q)$, so we fix an
Euclidean metric in $N_{\R}$. Let $p_{0}\in \sigma ^{\perp}$ be the
point corresponding to $\pi _{\sigma }(p)$, $U$ a neighborhood of
$p_{0}$ in $\sigma ^{\perp}$ and $q\in \sigma $. It is enough to show
that there is a $\mu _{0}>0$ and for all $\mu\ge \mu_{0}$, the
condition $p+\mu v\in W(\sigma ,U,q)$ holds. Since $\pi _{\sigma
}(p)=\pi _{\sigma }(p_{0})$ and $q\in \sigma $, we deduce that
$p-p_{0}-q\in \left< \sigma \right >_{\R }$. Since $v\in
\relint(\sigma ) $, there is a $\mu _{0}$ such that for all $\mu \ge \mu _{0}$,
we have $p-p_{0}-q+\mu v\in \sigma$ and hence
 $p+\mu v \in q+U+\sigma =W(\sigma ,U,q)$. 
\end{proof}

\begin{ex}
{As an example, we consider the toric variety $\mathbb P^2$ and its tropicalization shown in Figure \ref{Figure tropicalization}. Note that for $v=(0,1)$, the point $p+\infty v$ is the point in $N({\sigma_3})$ lying vertically above $p$.}
\end{ex}

\begin{figure}
	\includegraphics[width=.8\linewidth]{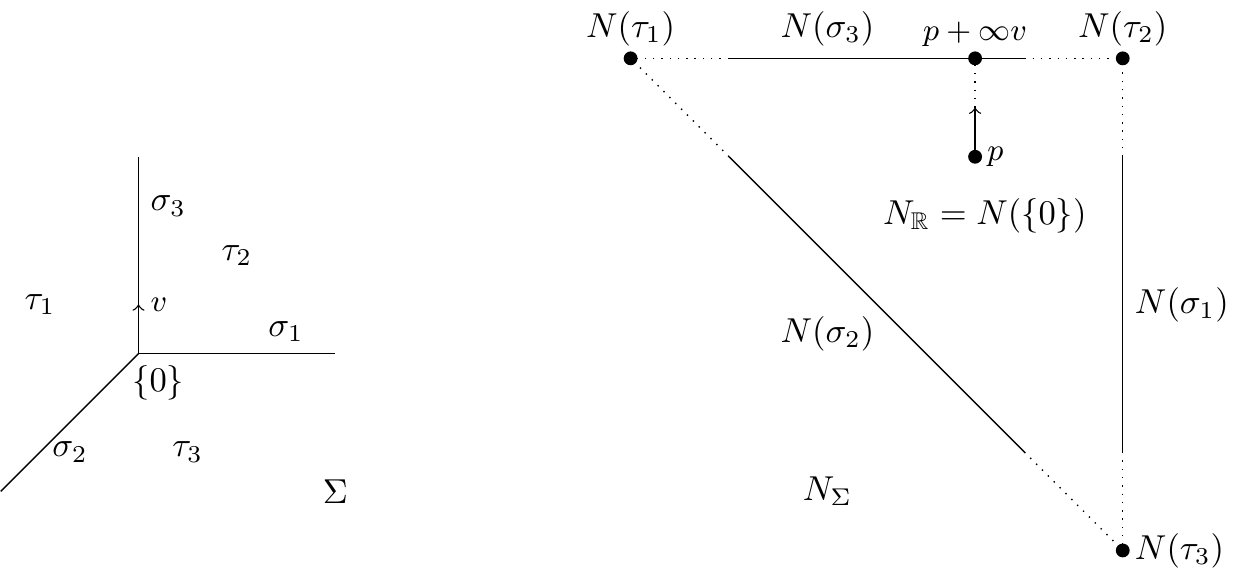}
	\caption{The fan and partial compactification of the toric variety
		$\mathbb{P}^2$.}
	\label{Figure tropicalization}
\end{figure}

\subsection{Lagerberg forms and Lagerberg currents}
\label{subsection smooth forms tropical}

Recall from \cite{lagerberg-2012} 
that for every open subset $U$ of $N_\R$, 
there is a bigraded $\R$-algebra 
of \emph{Lagerberg forms} $A^{\cdot,\cdot}(U)$
with differentials $d'$ and $d''$ of bidegree $(1,0)$ and $(0,1)$.
Lagerberg forms were introduced by Lagerberg  in \emph{loc.~cit.}~under 
the name superforms. They are defined
as 
\begin{align*}
A^{p,q}(U)=A^p(U)\otimes_{C^\infty(U)}A^q(U) 
\end{align*}
where $A^\cdot(U)$ denotes the $\R$-algebra of real
valued smooth differential forms on $U$.

{We choose a basis of $N$  which defines coordinates $u_1,\ldots,u_n$ on $N_\R$.} 
Then we may write a Lagerberg form $\alpha$ as 
\begin{align*}
\alpha = \sum_{I, J} f_{IJ} d'u_I \wedge d''u_J 
\end{align*}
where {$I=\{i_1< \dots < i_p\}$ and $J= \{j_1< \dots < j_q\}$} range over all subsets of $\{1,\dots,n\}$,
{where $f_{IJ}$ are smooth real functions on $U$ and we use the multi-index notation}
\begin{align*}
d'u_I \wedge d''u_J := du_{i_1} \wedge \dots \wedge du_{i_p} \otimes du_{j_1} \wedge \dots \wedge du_{j_q}.
\end{align*} 
There are natural differentials
$d'\colon A^{p,q}(U)\to A^{p+1,q}(U)$ and $d''\colon A^{p,q}(U)\to A^{p,q+1}(U)$, 
which are in coordinates given by 
\begin{align*}
d' ( f d'u_I \wedge d''u_J ) = 
\sum_{i_0 = 1}^n \frac{\partial f}{\partial u_{i_0}} d'u_{i_0} \wedge d'u_I \wedge d''u_J
\end{align*}
and
\begin{align*}
d'' ( f d'u_I \wedge d''u_J ) = 
(-1)^p \sum_{j_0 = 1}^n \frac{\partial f}{\partial u_{j_0}} d'u_I \wedge d''u_{j_0}  \wedge d''u_J.
\end{align*}

{The product of the  bigraded $\R$-algebra $A^{\cdot,\cdot}(U)$ is alternating and we denote it by $\wedge$.} 
The algebras $A^{\cdot,\cdot}(U)$ 
form a sheaf on $N_{\R}$ that is denoted  by $A^{\cdot,\cdot}$
or  by $A^{\cdot,\cdot}_{N_{\R}}$. 

The algebra $A^{\cdot,\cdot}(U)$  carries a natural
involution $J$ that permutes bidegrees and is determined
by $J(\alpha\otimes \beta)=\beta\otimes \alpha$
for all $\alpha,\beta\in A^\cdot(U)$.
A Lagerberg form $\alpha\in A^{p,p}(U)$ of type $(p,p)$
is called \emph{symmetric} 
if it satisfies $J(\alpha)=(-1)^p\alpha$.

Let us fix a  fan $\Sigma \subset N_\R$. {Following
\cite[Definition 2.4]{jell-shaw-smacka2015}, smooth forms on open subsets of 
$N_{\Sigma}$ are defined as follows.}

\begin{defi}\label{def-lagerberg-partial-compactification}
Let $U \subset N_\Sigma$ be an open subset. 
{For every $\sigma \in \Sigma $,} 
we write $U_\sigma \coloneqq U \cap {N(\sigma)}$. 
A \emph{Lagerberg form of type $(p,q)$ on $U$} is given by a family
$\omega=(\omega_\sigma)_{\sigma \in \Sigma }$ with $\omega_\sigma \in  A^{p,q}(U_\sigma)$ satisfying the following local condition. 
For each $p \in U_\sigma$, there exists a neighborhood $V$ of $p$ in $U$ such that for 
all $\tau \prec \sigma$ we have
\begin{displaymath}
 \omega_\tau |_{V_\tau\cap \pi_{\sigma,\tau}^{-1}(V_\sigma)} 
= \pi_{\sigma, \tau}^* 
(\omega_\sigma|_{V_\sigma})|_{V_\tau\cap \pi_{\sigma,\tau}^{-1}(V_\sigma)}. 
\end{displaymath}
We denote by $A^{p,q}(U)$ the real vector space of 
Lagerberg forms of type $(p,q)$ on $U$.
There are unique differentials
$d'\colon A^{p,q}(U)\to \AS^{p+1,q}(U)$
and $d''\colon A^{p,q}(U)\to A^{p,q+1}(U)$ such that
$(d'\omega)_\sigma=d'(\omega_\sigma)$ and 
$(d''\omega)_\sigma=d''(\omega_\sigma)$ for each 
$\omega\in \AS^{p,q}(U)$ and each $\sigma\in \Sigma$.
A \emph{smooth function} $f\colon U\to \R$ is
a Lagerberg form of type $(0,0)$ in $U$.
The assignment $U\mapsto A^{p,q}(U)$ defines a sheaf $A^{p,q}$ of
real vector spaces on the topological space $N_\Sigma$. 
{If we want to stress the fact that $A^{p,q}$ is a sheaf on $N_{\Sigma }$,  
we will denote it by $A_{N_{\Sigma }}^{p,q}$},
The \emph{support} $\mathrm{supp}(\omega )$ 
of a Lagerberg form $\omega \in A^{p,q}(U)$ is 
the closed subset of points of $U$ where $\omega$
has a non-zero germ in the stalk. 
The space of Lagerberg forms of type $(p,q)$ on 
an open subset $U$ of $N_\Sigma$ with 
\emph{compact support} is denoted $A^{p,q}_{c}(U)$.
\end{defi}

\begin{rem}\label{rem-support-forms}
\begin{enumerate} 
\item \label{item:3} The stalk $A^{p,q}_{N_{\Sigma },x}$ of the sheaf
  $A_{N_{\Sigma }}^{p,q}$ in a point $x \in 
  N(\sigma)\subset N_\Sigma$
can be identified with the stalk $A^{p,q}_{N(\sigma ),x}$ in $x$ of the sheaf of Lagerberg forms on 
the real vector space $N(\sigma)$. 
{This follows from Definition \ref{def-lagerberg-partial-compactification}.}
\item
Let $U\subset N_{\Sigma }$ be an open subset and $\omega  \in
A^{p,q}(U)$ a {Lagerberg form}. It follows from statement \ref{item:3} that  
{$\supp(\omega) = \cup_{\sigma \in \Sigma} \supp(\omega_\sigma)$, see \cite[Lemma 2.17]{jell-shaw-smacka2015}.}
\item 
Let $U\subset N_{\Sigma }$ be an open subset.
There is an involution $J$ on
$A^{\cdot,\cdot}(U)=\oplus_{p,q\geq 0}A^{p,q}(U)$ 
such that given a Lagerberg form
$\omega=(\omega_\sigma)_{\omega\in \Sigma}$ 
as in Definition \ref{def-lagerberg-partial-compactification},
$J(\omega)$
is determined by $(J\omega)_\sigma =J(\omega_\sigma)$.
A Lagerberg form $\omega\in A^{p,p}(U)$ of type $(p,p)$
is called \emph{symmetric} if it satisfies $J(\omega)=(-1)^p\omega$.
\item
For  $U\subset N_\Sigma$ open, the restriction map $A^{p,q}(U)\rightarrow A^{p,q}(U \cap N_\R)$ is injective.
\end{enumerate}
\end{rem}

\begin{rem}\label{rem-not-free}
Let $U$ be an open subset of $N_\Sigma$.
Observe that $A^{p,q}_{c}(U)$
is in general not a finitely generated $A^{0,0}_{c}(U)$-module.
This is caused by the fact that forms of large degree
vanish automatically at the boundary. 
\end{rem}

We next discuss a topology on the space $A^{p,q}_{c}(U)$. This
topology is modeled on the topology of the space of test forms used
in analysis, see for instance
\cite[\S 6]{rudin73:_funct_analy}. 
In fact, for an open subset $U$ of $N_\Sigma$, we will
define topologies on certain subspaces of $A^{p,q}_{c}(U)$ and use a
limit process to define a topology
on $A^{p,q}_{c}(U)$.   
Moreover, we shall describe the convergent sequences in $A^{p,q}_{c}(U)$.
In the following, we fix  a basis $u_1,\ldots,u_n$ of $M$ which defines coordinates $(u_1,\ldots,u_n)\colon N_\R\stackrel{\sim}{\rightarrow}\R^n$
and allows to write Lagerberg forms on $U$ in terms of standard
forms $d'u_I\wedge d''u_J$ for subsets $I,J\subset \{1,\ldots,n\}$.

\begin{defi}\label{new-definition-topology}
Let $U$ be an open subset of $N_\Sigma$.
For each compact subset $K\subset U$ and each
finite open covering $(V_i)_{i}$ of $K$, 
we denote by $A_{K}^{p,q}(U,(V_i)_{i})$ the subset of all
Lagerberg forms $\omega=(\omega_\sigma)_{\sigma\in \Sigma}$ in $A^{p,q}_c(U)$ 
such that $\supp(\omega)\subset K$ and
\begin{equation}
\omega_{\tau }|_{V_{i,\tau}\cap 
\pi _{\sigma ,\tau }^{-1}(V_{i,\sigma})}=
\pi _{\sigma ,\tau }^{\ast}\bigl(\omega_{\sigma }|_{V_{i,\sigma}}\bigr)
\big|_{V_{i,\tau}\cap \pi_{\sigma,\tau}^{-1}(V_{i,\sigma})}
\end{equation}
holds for all $i$ and all cones $\sigma,\tau\in \Sigma$ with $\tau
\prec \sigma $
and $V_{i,\sigma}=V_{i}\cap {N(\sigma)}\neq \emptyset$.
Given $\omega\in A_{K}^{p,q}(U,(V_{i})_{i})$,
we write $\omega =\sum_{I,J\subset\{1,\ldots,n\}} f_{I,J}d'u_{I}\land
d''u_{J}$ and define
\begin{equation}\label{new-distribution-norm}
\|\omega\|_{m}\coloneqq
\sum_{\substack{\alpha\in\N^n\\|\alpha |\le m}}\,\sum_{I,J\subset\{1,\ldots,n\}} 
\left\|\frac{\partial^{|\alpha |}f_{I,J}}{\partial u^{\alpha}}\right \|_{K}
\end{equation}
for each $m\in \N$ using the supremum norm $\metr_K$ of 
continuous real functions on the
compact set $K$. The family of norms \eqref{new-distribution-norm},
where $m\in \N$ varies, 
defines on $A^{p,q}_{K}(U,(V_{i})_{i})$ the
structure of a locally convex topological vector
space {which is complete with respect to a translation invariant metric and hence it is a \emph{Fr\'echet space}.} The induced topology is denoted by $\tau _{K,(V_{i})_{i}}$.

We put on $A^{p,q}_{c}(U)$ the topology
$\tau $,  defined as the limit topology 
\begin{displaymath}
(A^{p,q}_{c}(U),\tau )=\lim_{\substack{\longrightarrow\\K,(V_{i})_{i}}}
\left(A^{p,q}_{K}(U,(V_{i})_{i}),\tau_{K,(V_{i})_{i}}\right)
\end{displaymath}
in the category of locally convex topological vector spaces. 
Note that this may be different from the direct limit in the category of topological {vector} spaces. 
\end{defi}

As mentioned previously, the topology of $A^{p,q}_{c}(U)$ 
is modeled on the classical topology
on the space of test {functions} in 
\cite[\S 6]{rudin73:_funct_analy} and
its formal properties are very similar.  
For instance, if $U$ is not compact, then
$A^{p,q}_{c}(U)$ is not metrizable. 
Nevertheless the topology on $A^{p,q}_{c}(U)$ has many nice properties and the 
fact that is not metrizable is only a minor issue.

\begin{rem} \label{LF-spaces}
The spaces $A^{p,q}_{c}(U)$ have the same properties 
as the test function spaces in \cite[Chapter 6]{rudin73:_funct_analy}. 
This is a  consequence from the fact that $E \coloneqq A^{p,q}_{c}(U)$ 
is an \emph{$LF$-space} as introduced by Dieudonn\'e and Schwartz \cite{DS}. 
This means that the vector space $E$ is a countable union of strictly increasing Fr\'echet spaces $E_k$ 
such that the topology on $E_k$ agrees with the induced topology from $E_{k+1}$. 
Indeed, using that $U$ has a countable basis, 
it is clear that the direct limit can be by described by using countable many $(K_k,(V_{k,i})_{i})$ 
such that the compact subset $K_k$ lies in the interior of $K_{k+1}$ and 
such that a subfamily of  
$(V_{k+1,i})_{i}$ is a refinement of the open covering $(V_{k,i})_{i}$ of $K_k$. 
Setting 
\[
E_k \coloneqq \bigl(A^{p,q}_{K}(U,(V_{i})_{i}),\tau_{K,(V_{i})_{i}}\bigr),
\]
we see that $E$ is an $LF$-space. 
It is shown in \cite{DS} that $E$ has a canonical structure as a locally convex space 
which is the finest structure such that the topology on $E_n$ agrees with the induced topology and 
it follows that $E$ is the direct limit of the $E_n$ in the category of locally convex spaces. 

All properties of test function spaces from \cite[Chapter 6]{rudin73:_funct_analy} 
were shown in \cite{DS} more generally for  $LF$-spaces and 
so they apply to $A_c^{p,q}(U)$. 
In fact, an $LF$-space is not only sequentially complete, 
but a complete Hausdorf space \cite[Corollary  of Theorem 6]{DS}. 
For our paper, we need mainly the following results about sequences.
\end{rem}

\begin{prop}\label{convergent sequences}
Let $U$ be an open subset of $N_\Sigma$. 
A sequence $(\omega _{k})$ in $A^{p,q}_{c}(U)$
converges to $\omega \in A^{p,q}_{c}(U)$ if and
only if there is a compact subset $K\subset U$ and a finite open covering
$(V_{i})_{i}$ of $K$ such that all $\omega_k$ and $\omega$ are contained in $ A^{p,q}_{K}(U,(V_{i})_{i})$ 
and for every $m \in \N$, we have
$\lim_{k \to \infty}\|\omega_{k}- \omega\|_{m}=0$.
\end{prop}

\begin{proof}
A convergent sequence is bounded and hence the result follows from the fact 
that every bounded subset of an $LF$-space is contained in some $E_n$ \cite[Proposition 4]{DS}.
\end{proof}

Even if the space $A^{p,q}_{c}(U)$ is not metrizable, for
many purposes, sequences are enough.

\begin{prop}\label{test-proposition}
Let $T\colon A^{n-p,n-q}_{c}(U) \to \R$ be a linear functional. Then
the following conditions are equivalent.
\begin{enumerate}
\item\label{item:4}
The map $T$ is continuous.
\item \label{item:5}
If a sequence $(\omega _{k})_{k\in \N}$ converges to zero, 
then $(T(\omega _{k}))_{k\in \N}$ converges to zero.
\item \label{item:6}
The restriction of $T$ to each subspace
$A^{n-p,n-q}_{K}(U,(V_{i})_{i})$ is continuous.
\end{enumerate}
\end{prop}

\begin{proof}
The equivalence of \ref{item:4} and \ref{item:6} follows from
\cite[Proposition 5]{DS}, while \ref{item:5} and \ref{item:6} 
are equivalent by Proposition \ref{convergent sequences}.
\end{proof}

As in the classical case of distributions, we now define
currents as the topological dual of the space of smooth forms with
compact support.

\begin{defi}\label{new-definition-currents}
	A \emph{Lagerberg current of type $(p,q)$ on $U$} 
	is a  continuous linear  functional $T\colon A^{n-p,n-q}_{c}(U)\to \R$.
	The space of such Lagerberg currents is denoted by $D^{p,q}(U)$.
	A \emph{Lagerberg distribution} is a Lagerberg current of type $(n,n)$.
\end{defi}

\begin{rem}\label{product-partial-comp}
Let $U$ denote an open subset of $N_\Sigma$.
\begin{enumerate}
\item\label{product-partial-comp1}  If $U$ is contained in the open generic stratum $N_\R$, then
the definition of $D^{p,q}(U)$ above coincides 
with the definition of $D^{p,q}(U)$ given by
Lagerberg \cite[2.1]{lagerberg-2012}.
\item\label{product-partial-comp2} By the usual methods, 
the spaces of Lagerberg currents inherit operators 
\[
d'\colon D^{p,q}(U)\longrightarrow D^{p+1,q}(U),\,\,\,
d''\colon D^{p,q}(U)\longrightarrow D^{p,q+1}(U)
\] 
and a product 
\begin{displaymath}
A^{p',q'}(U)\otimes D^{p,q}(U)\longrightarrow D^{p+p',q+q'}(U)
\end{displaymath}
such that 
\begin{eqnarray*}
d'(T)(\omega')&=&(-1)^{p+q+1}T(d'\omega'),\\
d''(T)(\omega'')&=&(-1)^{p+q+1}T(d''\omega''),\\
(\beta\wedge T)(\omega)&=&(-1)^{(p'+q')(p+q)}T(\beta\wedge \omega)
\end{eqnarray*}
for all $T\in D^{p,q}(U)$, $\beta \in A^{p',q'}(U)$ and all
$\omega',\omega'',\omega \in 
A^{\cdot,\cdot}(U)$ of suitable bidegree.
\item \label{product-partial-comp3} 
There is an involution $J$ on
$D^{\cdot,\cdot}(U)=\oplus_{p,q\geq 0}D^{p,q}(U)$ 
such that given 
$T\in D^{p,q}(U)$ and $\omega\in A^{n-q,n-p}_c(U)$ the
Lagerberg current $J(T)\in D^{q,p}(U)$
is given by
\begin{displaymath}
  J(T)(\omega)=(-1)^{n}T(J(\omega)).
\end{displaymath}
A Lagerberg current $T\in D^{p,p}(U)$ 
is called \emph{symmetric} if it satisfies
$J(T)=(-1)^{p}T$.
\end{enumerate}
\end{rem}

Integration of Lagerberg forms gives examples of currents.  
We start by recalling the integration theory of Lagerberg forms. If $U\subset
  N_{\Sigma }$ is an open subset and $\eta\in
  \AS^{n,n}(U)$ is a Lagerberg form with compact support, using the
  chosen basis of $M$, we write 
  \begin{displaymath}
    \eta = f d'u_{1}\land d''u_{1}\land \dots \land d'u_{n}\land d''u_{n}.
  \end{displaymath}
  Denote by $d \lambda$ the Lebesgue measure on $N_\R$ induced by the
  lattice $N$.
  The integral of $\eta$ is defined as
  \begin{displaymath}
    \int_{U} \eta = \int_{N_\R} f d\lambda.
  \end{displaymath}
  Since the support of any compactly supported Lagerberg form of type
  $(n,n)$ is a compact subset of $N_\R$, the integral is finite.
  Since two isomorphisms $N\cong \Z^{n}$ differ by a matrix of
  determinant $\pm 1$, the integral does not depend on the choice of
  coordinates.

\begin{ex}\label{current of integration}
Let $U\subset N_{\Sigma }$ be an open subset. 
We will use the map
\begin{displaymath}
{\curr\colon A^{p,q}(U)\longrightarrow D^{p,q}(U), \quad 
\eta \longmapsto [\eta ]( \cdot )=\int_{U} \eta \wedge \,\cdot\,.}
\end{displaymath}
This map is a morphism of $A^{\cdot,\cdot}(U)$-modules compatible with
the actions of $d'$, $d''$ and $J$. 
\end{ex}     

\begin{ex} \label{measure gives current in top degree}
Let $U$ be an open subset of $N_{\Sigma }$. 
For every real Radon measure $\mu$ on $U$ 
(see Appendix \ref{facts from measure theory}), 
there exists a unique Lagerberg current
$T \in D^{n,n}(U)$ such that 
\begin{equation} \label{current equation}
T(f)=\int_U f \, d\mu \quad \forall f \in {A_c^{0,0}(U)}.
\end{equation}
Indeed, for a compact subset $K$ of $U$ and a finite covering 
$(V_i)_{i\in I}$ of $K$, the canonical maps
$A^{0,0}_K(U,(V_i)_{i\in I})\to C_K^0(U,\R)$ and
$C_K^0(U,\R)\to C_c^0(U),\R)$ are morphisms of locally
convex vector spaces. 
By the universal property
of the direct limit, the composition of these maps induces  a continuous map $A^{0,0}_c(U)\to C_c^0(U,\R)$.
This shows our claim.
\end{ex}

\begin{prop}\label{extended-partition-unity}
Let $(U_i)_{i\in I}$ be an open cover
of an open subset $U$ of $N_\Sigma$.
Then there exists a \emph{partion of unity subordinate
to the given cover $(U_i)_{i\in I}$}, i.e.~a
countable, locally finite open cover $(V_j)_{j\in J}$ of $U$
together with a map $s\colon J\to I$ such that $V_j\subset U_{s(j)}$
for all $j\in J$ and
a collection of non-negative functions $f_j\in A_c^{0,0}(V_j)$
such that $\sum_{j\in J}f_j \equiv 1$.
\end{prop}

\begin{proof}
By general arguments, see \cite[Theorem 1.11]{Warner}, 
it is sufficient to show that, given a point $x$ and a neighborhood $V$ of $x$ 
in $U$, there exists a function $f \in A^{0,0}_c(V)$ 
that is constant equal to $1$ on a neighborhood of $x$. 
This statement is clearly local, so we may assume that 
our fan $\Sigma$ is generated by a single cone $\sigma$ and
$N_\Sigma=N_\sigma$. 
We have seen in Remark \ref{def-topology-partial-compactification} 
that a choice of $k$ generators of the cone $\sigma$ leads to a
realization of $N_\sigma$ as a closed (polyhedral) subset of $\Rsup^k$. Now the existence 
of a partition of unity for an open subset of $\Rsup^k$  from \cite[Lemma 2.7]{jell-shaw-smacka2015} 
readily shows the claim. 
\end{proof}

Let $U$ be an open subset of $N_\Sigma$.
Recall from Appendix \ref{facts from measure theory} that the space $C_c^0(U,\R)$
of real valued continuous functions on $U$ with compact support has a canonical
structure of a locally convex topological vector space.

\begin{cor}\label{apply-stone-weierstrass}
For $U \subset N_\Sigma$ open, the image of  
$A^{0,0}_c(U)\hookrightarrow C_c^0(U)$ is sequentially dense.
\end{cor}

\begin{proof}
It is enough to show that, for any continuous function $f\in
C_c^0(U)$, there is a sequence of functions $g_{k}\in
A^{0,0}_c(U)$, $k\ge 0$, that converges to $f$ in the topology of
$C_c^0(U)$. Let $K$ be the support of $f$.
Using that $N_\Sigma$ is locally compact,
we can easily find open
subsets $U_{1}, U_{2}$ and compact subsets $K_{1},K_{2}$ with
\begin{displaymath}
K \subset U_{1}\subset K_{1}\subset U_{2}\subset K_{2} \subset U.   
\end{displaymath}
By Proposition \ref{extended-partition-unity}, the Stone--Weierstrass
Theorem \cite[Theorem 7.32]{rudin1964} implies that 
the $\R$-algebra $\{h|_{K_1}\mid h \in A^{0,0}(U),\, \supp(h) \subset K_2\}$
is dense in $C^0(K_1)$. 
Hence there is a sequence of smooth functions $h_{k}\in A^{0,0}_c(U)$ with 
$\supp(h_k) \subset K_2$ such that the $h_k|_{K_1}$ converge uniformly to $f$ in $C^0(K_{1})$. 
Again by Lemma \ref{extended-partition-unity},
there is a smooth function $0\le \rho \le 1 $, 
whose support is contained in
$U_{1}$ and with $\rho|_K \equiv 1$. 
Then the sequence of smooth functions given by
$g_{k}=\rho h_{k}$ converges to $f$ in $C_c^0(U)$.  
 \end{proof}

\section{Positivity for complex invariant forms and Lagerberg forms} 
\label{Complex invariant forms} 

In this section we study positive forms on a 
smooth complex toric variety
that are invariant under the action of the compact 
torus and compare them to positive Lagerberg forms.
We keep the setting from Section \ref{section smooth forms tropical}.

\subsection{Invariant forms in the case of the torus} \label{subsection invariant forms torus}
We start with the case of the complex algebraic torus 
$\torus = \Spec\,\C[M]$ of dimension $n$ with character lattice $M$ 
and cocharacter lattice $N$. 
We fix a splitting $N\cong \Z^{n}$ that induces  holomorphic
coordinates $z_1, \dots, z_n$ on $T$ as well as linear coordinates
$u_{1},\dots,u_{n}$ on $N_{\R}$. 
As before we
denote the associated complex manifold 
$M\otimes_\Z\C^\times\cong(\C^\times)^n$ by $\torusan$. 
We will also consider the real compact torus 
\begin{align*}
\SS \coloneqq 
\{z \in {\torusan} \mid  |z_j|=1 \,\,\,\forall j\in \{1,\ldots,n\}\} = 
\{z \in {\torusan} \mid |\chi^u(z)|=1 \,\,\,\forall u\in M\}.
\end{align*}

We will denote by  $\AS$ either the sheaf of complex 
differential forms on ${\torusan}$ or the sheaf of 
(real) Lagerberg forms on $N_{\R}$.  
The context will always allow us to distinguish between
them. If not we will denote the former as $\AS_{{\torusan}}$ 
and the latter as $\AS_{N_{\R}}$.  
For an $\SS$-invariant subset $V$ of ${\torusan}$,  let 
$\AS(V)^\SS$ denote the subalgebra of $\AS(V)$  given by the
$\SS$-invariant forms.

\begin{rem} \label{rem:2}
Let $V$ be an $\SS$-invariant subset of ${\torusan}$.
The subalgebra  $\AS(V)^\SS$ is a direct factor of 
$\AS(V)$ because averaging with
  respect to the  {Haar probability measure $\mu_\SS$} of 
  $\SS$ defines a canonical projection 
\[
\AS(V)\longrightarrow \AS(V)^\SS, \quad \omega  \longmapsto \omega^{\rm av}
{\coloneqq \int_\SS a^*(\omega) \, d\mu_\SS(a)}
\]  
where $a\colon \torusan\to \torusan$ denotes translation by $a\in {\torusan}$. 
\end{rem}

\begin{defi}
Let $F$ be the antilinear involution of the sheaf of 
$\C$-algebras ${\AS=}\AS_{{\torusan}}$
determined by $F(f)=\bar f$ for $f\in \AS^{0,0}$ and by
  \begin{align*}
    F(dz_{j}/z_{j})= dz_{j}/z_{j},\quad 
    F(id\bar z_{j}/\bar z_{j})&= i d\bar z_{j}/\bar z_{j}
\end{align*}
for $j=1,\dots,n$. The $\SS$-invariant one-forms
$dz_j/z_j$ and $id\bar z_j/\bar z_j$ $(j=1,\ldots,n)$
generate a subsheaf of $\R$-algebras $B$ of $\AS$ such that
$\AS^{\cdot,\cdot}= \AS^{0,0} \otimes_\R B^{\cdot,\cdot}$.
Observe that this definition of $B$ does not depend on the
choice of the splitting $N\cong \Z^n$ which induces the coordinates
$z_1,\ldots,z_n$.
The antilinear involution $F$ above is the $\R$-linear endomorphism 
on $\AS_{{\torusan}}$ given as the tensor product of complex conjugation  
on $\AS^{0,0}$ with the identity on $B$.
We conclude that the involution $F$ is indeed well defined.
\end{defi}

\begin{lem}\label{lemm:11} 
The involution $F$ is ${\torusan}$-equivariant. 
In particular $F$ induces an antilinear
involution, also denoted by $F$, of $\AS(V)^{\SS}$ for any
$\SS$-invariant open subset $V$ of ${\torusan}$.
\end{lem}

\begin{proof}
For $f\in \AS^{0,0}(V)$ and $j=1,\dots,n$, we have
\begin{displaymath}
a^{\ast}\bar f=\overline{a^{\ast}f},\  
a^{\ast}(dz_{j}/z_{j})=dz_{j}/z_{j},\text{ and }
a^{\ast}(id\bar z_{j}/\bar z_{j})=i d\bar z_{j}/\bar z_{j}.
\end{displaymath}
We deduce that $a^{\ast}F(\omega )=F(a^{\ast}\omega )$ 
for any $\omega \in \AS(V)$ and the claim follows.
\end{proof}

\begin{lem}\label{lemm:14}
  Let $V\subset {\torusan}$ be an $\SS$-invariant open subset. If $\omega \in
  \AS(V)^{\SS}$, then
  \begin{displaymath}
    F(\partial\omega )=\partial F(\omega ),\quad
    F(i \bar \partial\omega )=i \bar \partial F(\omega ).
  \end{displaymath}
\end{lem}

\begin{proof}
We note that the $\SS$-invariant forms $i^{|K|}dz_{I}\wedge d\bar z_{K}/(z_{I}\bar z_{K})$ give a frame in   $\AS(V)$. 
Using that $r \to \log(r^2)$ is a diffeomorphism from $\R_{>0}$ onto $\R$,  any $\omega \in \AS(V)^{\SS}$ can be written as
  \begin{equation} \label{added eqnr}
    \omega =\sum_{I,K}f _{I,K}(\log(z_{1}\bar
    z_{1}),\dots,\log(z_{n}\bar z_{n}) )\frac{i^{|K|}dz_{I}\wedge d\bar 
      z_{K}}{z_{I}\bar z_{K}}
  \end{equation}
  where the functions $f_{I,K}$ are smooth complex valued
  functions of $n$ real variables.
  We denote by $\partial_{j}f_{I,K}$ the partial derivative with respect
  to the $j$-th variable. Clearly
  \begin{equation*}
    \label{eq:7}
    \partial_{j}\bar f_{I,K}=\overline {\partial_{j}f_{I,K}}.
  \end{equation*} 
  Then $F(\partial \omega)$ is equal to 
  \begin{displaymath}
    F\left(
    \sum_{I,K}\sum_{j\not \in I}\partial_{j}
    f_{I,K}\frac{dz_{j}}{z_{j}} \wedge \frac{i^{|K|} dz_{I}\wedge d\bar  
      z_{K}}{z_{I}\bar z_{K}} \right) =
        \sum_{I,K}\sum_{j\not \in I}\overline{\partial_{j}
    f_{I,K}}\frac{dz_{j}}{z_{j}} \wedge \frac{i^{|K|}dz_{I}\wedge d\bar  
      z_{K}}{z_{I}\bar z_{K}},
  \end{displaymath}
  and 
  \begin{displaymath}
    \partial F(\omega) = \partial \left(
    \sum_{I,K}
    \bar f_{I,K}\frac{i^{|K|} dz_{I}\wedge d\bar  
      z_{K}}{z_{I}\bar z_{K}} \right) =
        \sum_{I,K}\sum_{j\not \in I}\partial_{j}
    \bar f_{I,K}\frac{dz_{j}}{z_{j}} \wedge \frac{i^{|K|}dz_{I}\wedge d\bar  
      z_{K}}{z_{I}\bar z_{K}}.
  \end{displaymath}
Thus the commutativity between $\partial$ and $F$ follows from $\partial_{j}\bar f_{I,K}=\overline {\partial_{j}f_{I,K}}$.

Similarly, $F(i\bar \partial \omega)$ is equal to 
  \begin{displaymath}
     F \left(
    \sum_{I,K}\sum_{j\not \in K}\partial_{j}
    f_{I,K}\frac{i d\bar z_{j}}{\bar z_{j}} \wedge \frac{i^{|K|} dz_{I}\wedge d\bar  
      z_{K}}{z_{I}\bar z_{K}} \right) =
        \sum_{I,K}\sum_{j\not \in K}\overline{\partial_{j}
    f_{I,K}}\frac{i d\bar z_{j}}{\bar z_{j}} \wedge \frac{i^{|K|}dz_{I}\wedge d\bar  
      z_{K}}{z_{I}\bar z_{K}},
  \end{displaymath}
  and
  \begin{displaymath}
    i\bar \partial F(\omega) = i\bar \partial \left(
    \sum_{I,K}
    \bar f_{I,K}\frac{i^{|K|} dz_{I}\wedge d\bar  
      z_{K}}{z_{I}\bar z_{K}} \right) =
        \sum_{I,K}\sum_{j\not \in K}\partial_{j}
    \bar f_{I,K}\frac{i d\bar z_{j}}{\bar z_{j}} \wedge \frac{i^{|K|}dz_{I}\wedge d\bar  
      z_{K}}{z_{I}\bar z_{K}}.
  \end{displaymath}
Therefore the commutativity of $F$ and $i\bar \partial$ also follows from $\partial_{j}\bar f_{I,K}=\overline {\partial_{j}f_{I,K}}$.
 
\end{proof}

The next example shows that $\SS$-invariance of $\omega$ is necessary  in Lemma \ref{lemm:14}.

\begin{ex} 
Consider the case $n=1$ and let $f(z)=z+\bar z$. 
This function is real but not $\SS$-invariant. 
Furthermore  
\[
\partial F(f) =\partial f =dz
\]
and
\[
F(\partial f)  =F(dz)=F(z dz/z)=\bar z dz/z.
\]
\end{ex}

\begin{defi}
  For any $\SS$-invariant open subset $V\subset {\torusan}$, we  
  denote by
  $\AS(V)^{\SS,F}$ the algebra of forms that are simultaneously 
  $F$-invariant and $\SS$-invariant. Since $F$ and the action of $\SS$
  both respect the bigrading of $\AS$, we deduce that $\AS(V)^{\SS,F}$
  is a bigraded algebra. By Lemma \ref{lemm:14}, 
  the operators $\partial$ and $i\bar \partial$  induce operators  on $\AS(V)^{\SS,F}$. 
\end{defi}

The next goal is to give an identification between the algebra of $\SS$-invariant
forms that are also $F$-invariant, and the algebra of real  Lagerberg forms on $N_\R$. 
This identification will respect the natural differential operators. 
Recall that $J$ denotes the Lagerberg involution introduced in 
Remark \ref{rem-support-forms}.

\begin{prop}\label{prop:4} 
Let $U\subset N_{\R}$ be an open subset
and $V=\trop^{-1}(U)$ the corresponding $\SS$-invariant subset of ${\torusan}$.
Then there is a unique homomorphism 
$\trop^{\ast} \colon \AS(U)\to \AS(V)$
of $\R$-algebras such that 
\begin{equation}\label{eq:21}
  \trop^{\ast} (\varphi) = \varphi \circ \trop
\end{equation}
for all $\varphi \in \AS^{0,0}({U})$ and such that
\begin{equation}\label{eq:19}
  \trop^{\ast}(d'\omega)= {\pi^{-1/2}}\partial\,\trop^{\ast} (\omega)  
  \quad \text{and}
\quad \trop^{\ast} (d''\omega)= {\pi^{-1/2}} i\overline \partial\,\trop^{\ast} (\omega)
\end{equation}
for all $\omega \in \AS(U)$. 
Moreover, this homomorphism is injective with image $\AS(V)^{\SS,F}$ 
inducing an isomorphism 
$\AS(U)\simeq\AS(V)^{\SS,F}$. 
For $\omega \in \AS ^{p,q}(U)$, we have
\begin{equation}
  \label{eq:18}
  \trop^{\ast}(J(\omega ))=i^{p+q}\overline{\trop^{\ast}(\omega )}.
\end{equation}
\end{prop}

\begin{proof} We first recall some formulas from complex analysis in one
variable. If $z=re^{i\theta}$ and $u=-\log|z|=-\log(r)$, since $r^2=z
\bar{z}$, we have
\begin{displaymath}
  \partial(u)= -\frac{dz}{2z} \quad \text{and} \quad \bar{\partial}{u}
  = -	\frac{d\bar{z}}{2\bar{z}}.
\end{displaymath}
  For $j=1,\dots,n$ write $r_{j}=|z_{j}|=(z_{j}\bar
    z_{j})^{1/2}$. Then
  \begin{displaymath}
    \trop^{\ast}(u_{j})=-\log(r_{j}).
  \end{displaymath}
Therefore, equations \eqref{eq:21} and \eqref{eq:19} imply
\begin{equation}\label{eq:25}
  \trop^{\ast}(d'u_{j})=-\frac{dz_{j}}{2 {\sqrt{\pi}}  z_{j}},\quad
  \trop^{\ast}(d''u_{j})=-\frac{id\bar z_{j}}{2{\sqrt{\pi}} \bar z_{j}}.
\end{equation}
Since $\trop ^{\ast}$ is an algebra homomorphism,
we deduce that, if $\omega \in \AS(U)$ is written as
\begin{equation}\label{eq:23}
  \omega = \sum_{I,K} f_{I, K}(u_{1},\dots,u_{n}) \, d' u_I \wedge
  d'' u_K,
\end{equation}
 then the corresponding form on the torus is given by
\begin{equation}\label{eq:22}
  \trop^{\ast}(\omega)= \sum_{I,K}
  \left(\frac{-1}{2{\sqrt{\pi}}}\right)^{|I|+|K|}i^{|K|}
  f_{I, K}(-\log(r_{1}),\dots,-\log(r_{n}))
  \, \frac{dz_{I}\wedge d\bar z_{K}}{z_{I}\wedge \bar z_{K}}.
\end{equation}
This proves the uniqueness of the map $\trop^{\ast}$. 
Conversely, using
equation \eqref{eq:22} as the definition of $\trop^{\ast}$, the
uniqueness of the decomposition \eqref{eq:23} shows that $\trop ^{\ast}$
is well defined and it is immediate to verify \eqref{eq:21},
\eqref{eq:19} and \eqref{eq:18}.
{Clearly, the form in \eqref{eq:22} is $F$- and $\SS$-invariant 
and every element in $\AS(V)^{\SS,F}$ has this form. 
Obviously, the map $\trop^*$ is injective.
Hence it induces an isomorphism $\AS(U)\simeq\AS(V)^{\SS,F}$.}
\end{proof}

\begin{cor}\label{top-coro}
{Let $U\subset N_{\R}$ be an open subset and $V=\trop^{-1}(U)$.
Then $\trop^*$ induces an isomorphism 
of topological vector spaces $\AS_c(U)\simeq\AS_c(V)^{\SS,F}$. }
\end{cor}

\proof
{The isomorphism of vector spaces follows from Proposition \ref{prop:4}
as the map trop is proper.
Using that trop is a submersion, the topological statement
can be checked locally in coordinates.}
\qed

\begin{rem}\label{sheaf-interpretation}
The $\SS$-invariant open subsets of $\torusan$ are precisely
the preimages of open subsets of $N_\R$.
Hence $\SS$ and $F$ act in a natural way on the
sheaf $\trop_*A$. We denote by
$(\trop_*A)^{\SS,F}$ the subsheaf of sections invariant
under $\SS$ and $F$.
Then Proposition \ref{prop:4} yields a monomorphism
of sheaves of $\R$-algebras
\begin{equation}\label{sheaf-interpretation-map}
\trop^*\colon A\longrightarrow (\trop_*A)^{\SS,F}
\end{equation}
that preserves differentials as in \eqref{eq:19}.
\end{rem}

{The next remark shows that Lagerberg forms and the involution $F$ are pointwise described by the linear algebra in Section \ref{sec:posit-real-compl}.}

\begin{rem} \label{pointwise comparision}
Choose a point $x \in {\torusan}$ and let
$y=-\log|x|=\trop(x)$ denote its image under the tropicalization map. 
Write $V=T_{y}N_{\R}$ for the tangent space to $N_{\R}$ at $y$. 
Let $V'$, $V''$, $V^{\ast}$ and $\overline {V^{\ast}}$ be defined as Section \ref{sec:posit-real-compl}. 
There is an isomorphism
\begin{align*} 
\trop^{\ast}_{x}\colon V^{\ast}\oplus \overline {V^{\ast}}\longrightarrow T^{1,0}_{x}\torusan \oplus
T^{0,1}_{x}\torusan
\end{align*}
given by 
\begin{displaymath}
  d u_{j} \longmapsto -\frac{dz_{j}} {2{\sqrt{\pi}} x_j}, \quad\text{and}\quad
  i d \bar u_{j} \longmapsto \frac{-i d\bar z_{j}} {2{\sqrt{\pi}}\bar x_j}.
\end{displaymath}
This isomorphism is compatible with the map $\trop^{\ast}$ in the
sense that
\begin{displaymath}
  \trop^{\ast}(\omega )(x)=\trop^{\ast}_{x}(\omega (y)).
\end{displaymath}

Let $\tau_n\in\Lambda^{n,n}V'$ be the Lagerberg orientation 
as in Definition \ref{Lagerberg involution and orientation}. 
Then
\begin{equation}\label{eq:11}
  \trop^{\ast}_x (\tau _{n})= 
\frac{i^{n}   dz_{1}\land d \bar z_{1}\land \dots \land dz_{n}\land d \bar z_{n}}
 {(4\pi)^{n}|x_1|^{2}\dots |x_n|^{2}}.
\end{equation}
Note that the denominator of this form is strictly positive in the
torus ${\torusan}$ therefore, any positivity notions on ${\torusan}$ defined using
the orientation $\trop^{\ast}_{x}(\tau _{n})$ or the orientation
\begin{displaymath}
  i^{n}  dz_{1}\land d \bar z_{1}\land \dots \land dz_{n}\land d \bar z_{n}
\end{displaymath}
agree. 

Let $\omega\in A^{p,q}(W)$ for some open subset $W$ of $\T^\an$ and $x \in W$. 
We get from \eqref{def:555} that
\begin{equation}\label{F-comparison}
F(\omega)(x)=F(\omega(x))
\end{equation}
where $F$ on the righthand side is the involution from Definition \ref{def:5}.
\end{rem}

\begin{defi} \label{def: positivity of Lagerberg superforms}
  Let $U$ be an open set of $N_{\R}$. A Lagerberg form
  $\omega \in \AS(U)$ is called \emph{strongly positive (respectively
  positive, weakly
  positive)} if
  for all $x\in U$, $\omega (x)$ is strongly positive (respectively
  positive, weakly positive).

  Let $V$ be an open set of ${\torusan}$. A complex differential form
  $\omega \in \AS(V)$ is called strongly positive (respectively
  positive, weakly
  positive) if
  for all $y\in V$, $\omega (y)$ is strongly positive (respectively
  positive, weakly positive).
\end{defi}

As we have seen in Example \ref{too few strongly positive},
it is reasonable to restrict our attention to positive Lagerberg forms. 
Nevertheless we add the other positivity notions for further reference.

\begin{lem}\label{lemm:4} 
Let $U$ be an open set of $N_{\R}$ and 
$\omega \in \AS^{p,p}(U)$ be a Lagerberg form on $U$.
  \begin{enumerate}
  \item The Lagerberg form $\omega $ is symmetric if and only if
    $\trop^{\ast}\omega $ is real.
  \item If $\omega $ is strongly positive, then $\trop^{\ast}\omega $
    is strongly positive.
  \item If $\trop^{\ast}\omega $ is weakly positive then $\omega $ is
    weakly positive. 
  \item The Lagerberg form $\omega $ is positive if and only if $\trop^{\ast}\omega $ is
    positive.
  \end{enumerate}
\end{lem}
\begin{proof}
  The Lagerberg form $\omega $ is symmetric if and only if $J\omega
  =(-1)^{p}\omega $. By equation \eqref{eq:18} this is equivalent to 
 $\trop^{\ast}(\omega )=\overline{\trop^{\ast}(\omega )}$. This proves
 the first statement.

 The remaining assertions follow from the fact that 
 positivity is checked pointwise, 
 {Remark \ref{pointwise comparision}} and 
 Propositions \ref{prop:15} and \ref{prop:14}.  
\end{proof}

\begin{rem} \label{about the choice of the corresponding differential operators1}
Remark \ref{pointwise comparision} shows that the correspondence 
\begin{equation}\label{correspondence-choice}
{\pi^{-1/2}}\partial \longleftrightarrow d'\,\mbox{ and }\,
{\pi^{-1/2}}i\overline \partial \longleftrightarrow d''
\end{equation}
from Proposition \ref{prop:4} was already used implicitly in 
Section \ref{sec:posit-real-compl}  in order to preserve positivity 
and the bigrading between the complex and the Lagerberg forms. 
\end{rem}

\begin{rem} \label{about the choice of the corresponding differential operators2}
Chambert-Loir and Ducros \cite[Remarque 1.2.12]{chambert-loir-ducros} 
mention the identifications 
$d=\partial + \overline \partial \leftrightarrow d'$ and 
$4\pi d^c= \frac{1}{i}(\partial - \overline \partial) \leftrightarrow d''$
{(observe that $d\,\mathrm{Arg}(z)=4\pi d^c\log |z|$)} 
which have the disadvantage that they do not respect the bigrading and 
do not allow the nice interpretation of symmetric Lagerberg forms as 
real complex forms as in Lemma \ref{lemm:4}. 
\end{rem}

\begin{rem} \label{about the choice of the corresponding differential operators3}
Let $U$ be an open subset of $N_\R$ and $V\coloneqq \trop^{-1}(U)$. 
We have introduced the map $\trop^*\colon A(U) \to A(V)$ in order
to establish later correspondence results between Lagerberg forms 
and currents on $U$ and invariant complex forms and currents on $V$.

Let us explain our choices which lead to 
the achieved correspondence \eqref{correspondence-choice}: 
Our first condition is that $\trop^*$ should be a differential 
homomorphism of $\R$-algebras with respect to the differential 
operators $d'$ and $d''$ of $A(U)$ and  
natural differential operators $\dell'$	and $\dell''$ on $A(V)$. 
Naturality means here that $\dell',\dell''$ should be in the 
two-dimensional $\C$-vector space spanned by $\dell$ and $\dellbar$. 
The second condition is that $\trop^*$ respects 
the bigrading which implies $\dell'=a \dell$ and $\dell''= bi \dellbar$ for some $a,b \in \C$. 
Third, the range of $\trop^*$ should be contained in 
the $F$-invariant forms in $V$ which yields $a,b \in \R$. 
Observe that these three conditions imply already 
that $\trop^*$ preserves positivity of forms 
(see Lemma \ref{lemm:4}).  
Our fourth condition is \eqref{eq:18} which forces $a=b$. 
The fifth condition is that $\trop^*$ is 
compatible with integration 
(see Lemma \ref{comparison of integration} below). 
This gives $ab=1/\pi$. Hence we have seen that
the five conditions given above fix our choices 
in Proposition \ref{prop:4} up to a sign.
\end{rem}

\subsection{Invariant forms in the case of a toric variety} \label{subsection invariant forms toric variety}

The next goal is to partially extend Proposition \ref{prop:4} to toric varieties.   
Let $\Sigma $ be a smooth 
fan in $N_\R$.
Let $X_{\Sigma }$ be the corresponding smooth complex toric variety
and let $N_{\Sigma }$ be the corresponding partial compactification of $N_\R$ 
as in Section \ref{section smooth forms tropical}. 
We denote by $\Xsigmaan$ the complex manifold associated to $X_\Sigma$. 
Recall from Remark \ref{ass-toric-variety}  
that the tropicalization map is a proper continuous map 
$\trop\colon \Xsigmaan \to N_{\Sigma }$ that identifies $N_{\Sigma }$
with $\Xsigmaan/\SS$.

\begin{rem}
  Let $V\subset \Xsigmaan$ be an $\SS$-invariant open subset. 
  Then, in general, the involution $F$
  does not induce an involution of $A^{\cdot,\cdot}(V)$. In fact, if
  $\omega \in A^{\cdot,\cdot}(V)$, then $F(\omega |_{V\cap \T^{\an}})$ may
  not extend to a smooth form on $V$, as the
  following example shows. 
  
  We set $X_{\Sigma }=\A^{1}_{\C}$ and 
  $V=\Xsigmaan =\C$. Then $dz\in
  A^{1,0}(V)$, but
  \begin{displaymath}
    F(dz)=F\Bigl(z\frac{dz}{z}\Bigr)=\frac{\bar z}{z}dz
  \end{displaymath}
  is not a smooth form in $0$. Nevertheless, the next result implies that $F$ can be
  extended to smooth $\SS$-invariant forms.
\end{rem}

  \begin{lem}\label{lemm:1} Let $V\subset \Xsigmaan$ be an
    $\SS$-invariant open subset. 
    Let $\omega \in A^{\cdot,\cdot}(V)^{\SS}$ be a smooth
    $\SS$-invariant form on $V$. Then 
    \begin{math}
      F(\omega|_{V\cap \T^{\an}} )
    \end{math}
    extends uniquely to a smooth form $\SS$-invariant form on $V$ and hence $F$ induces an antilinear involution on
  $A^{\cdot,\cdot}(V)^{\SS}$ also denoted by $F$.
  \end{lem}

  \begin{proof}
   Since the statement is local we may assume that $V$ is
      contained in the affine toric variety $X_{\sigma }$ for some
      $\sigma \in \Sigma$. We choose a system of toric coordinates
      $(z_{1},\dots, z_{n})$ and write
      \begin{displaymath}
        \omega = \sum f_{I,K}(z_{1},\dots,z_{n})dz_{I}\wedge
        d\bar z_{K}.
      \end{displaymath}
      Moreover, the $\SS$-invariance of $V$ yields that $V$ is invariant
      under complex conjugation of the coordinates. Since $\omega $ is
      $\SS$-invariant, 
      each summand
    $f_{I,K}dz_{I}\wedge
    d\bar z_{K}$ is also $\SS$-invariant and hence we can write
    \begin{displaymath}
      f_{I,K}(z_{1},\dots,z_{n})dz_{I}\wedge d\bar z_{K}
      =g_{I,K}(|z_{1}|,\dots,|z_{n}|)\frac{i^{|J|}
        dz_{I}\wedge d\bar z_{K}}{z_{I}\bar z_{K}},
    \end{displaymath}
   for a unique smooth function $g_{I,K} \in C^\infty(V)$.  On $V\cap \T^{\an}$, this yields
    \begin{multline*}
      F(f_{I,K}(z_{1},\dots,z_{n})dz_{I}\wedge d\bar z_{K})
      =
      F\left(g_{I,K}(|z_{1}|,\dots,|z_{n}|)\frac{i^{|K|}
      dz_{I}\wedge d\bar z_{K}}{z_{I}\bar z_{K}}\right)
      =\\
      \overline{g_{I,K}(|z_{1}|,\dots,|z_{n}|)}\frac{i^{|K|}
        dz_{I}\wedge d\bar z_{K}}{z_{I}\bar z_{K}}
      =
      (-1)^{|K|}\overline{
    f_{I,K}(\bar z_{1},\dots,\bar z_{n})}dz_{I}\wedge
    d\bar z_{K}
  \end{multline*}
  which can be uniquely extended to a smooth form on $V$ because $\omega $ is
  smooth. By Lemma \ref{lemm:11},  the 
  restriction $F(\omega )|_{V\cap \T^{\an}}$ is $\SS$-invariant. By
  continuity, $F(\omega )$ is $\SS$-invariant.
  \end{proof}

For any open $\SS$-invariant subset $V\subset
\Xsigmaan$, we will denote by $\AS(V)^{\SS,F}$ the $\R$-subalgebra of forms that are at the 
same time invariant under $\SS$ and under $F$. 
As before we obtain a sheaf of bigraded $\R$-algebras
$(\trop_*\AS_{\Xsigmaan})^{\SS,F}$ on $N_\Sigma$.

\begin{prop}\label{prop:5}
There is a unique morphism of sheaves of bigraded $\R$-algebras
\begin{equation}\label{trop-invariant-lagerberg-forms}
\trop ^{\ast}\colon \AS_{N_{\Sigma }} \longrightarrow
(\trop_*\AS_{\Xsigmaan})^{\SS,F}
\end{equation}
that extends the morphism \eqref{sheaf-interpretation-map}. 
Moreover, $\trop^*$ satisfies
\begin{equation}
\label{trop-invariant-lagerberg-forms2}
\pi^{-1/2}\dell\circ \trop^*=\trop^*\circ d' \quad \text{and} \quad \
\pi^{-1/2}i\dellbar\circ \trop^*=\trop^*\circ d''.
\end{equation}
{For an open subset $U$ of $N_\Sigma$ and $\omega \in \AS ^{p,q}(U)$, we have}
\begin{equation}
\label{eq:18'}
{\trop^{\ast}(J(\omega ))=i^{p+q}\overline{\trop^{\ast}(\omega )}.}
\end{equation}
\end{prop}
\begin{proof}
This follows easily from the definitions {and Proposition \ref{prop:4}}.
\end{proof}

\begin{rem}\label{rem:4}
{If $U$ is an open subset of $N_\R$, then we have seen in 
Proposition \ref{prop:4} that the map \eqref{trop-invariant-lagerberg-forms} 
is an isomorphism. 
For an open subset $U$ of $N_\Sigma$, the map is obviously still 
injective, but in general no longer surjective.}
The latter can be seen already in the one dimensional case. Let
$N=\Z$ and $\Sigma $ the fan with a single maximal cone $\sigma
  =\R_{\ge 0}$. Then $X_{\Sigma}=\A_\C^{1}$ and $N_{\Sigma }=\R\cup
  \{\infty\}$.  Consider the smooth
  function $\varphi(z)=z\bar z$ on $\Xsigmaan=\C$ and the function
  $f(u)=e^{-2u}$ on $N_{\R}=\R$. Then, on $\C^{\times}\subset \C$ 
we have $\varphi = \trop^{\ast} f$, 
but $f$ can not be extended to a smooth function on $N_{\Sigma }$,
because, by definition, a smooth function on $N_{\Sigma }$ has to
be constant in a neighborhood of $\infty$.
This gives an example of a non-surjective
\[
\trop ^{\ast}\colon \AS^{0,0}(U) \longrightarrow \AS^{0,0}(\trop^{-1}(U))^{\SS,F}.
\]
\end{rem}

\begin{lem} \label{comparison of integration}
  Let $U$ be an open set of $N_{\Sigma }$ and $V=\trop^{-1}(U)$. Let
  $\eta$ be a Lagerberg $(n,n)$-form with compact support contained in
  $U$. Then
  \begin{displaymath}
    \int_{U}\eta =\int_{V}\trop^{\ast}\eta.
  \end{displaymath}
\end{lem}
\begin{proof}
Since $\eta$ is a top degree Lagerberg form with compact support in $U$, 
it has compact support contained in $U\cap N_{\R}$. 
Choosing integral linear coordinates in $N_\R$, we have
\begin{displaymath}
  \eta = f(u_{1},\dots,u_{n}) d'u_{1}\land d''u_{1}\land \dots \land
  d'u_{n}\land d''u_{n} 
\end{displaymath}
for a smooth function $f$ on $\R^n$. 
{By Proposition \ref{prop:4}, we have}
\begin{displaymath}
  \trop^{\ast}\eta=
  f(-\log|z_{1}|,\dots,-\log|z_{n}|)
  \frac{i^{n}dz_{1}\land  d \bar z_{1}\land \dots \land dz_{n}
    \land d \bar z_{n}}{(4\pi)^nz_{1}\dots z_{n} \bar z_{1}\dots \bar z_{n}}.
\end{displaymath}
Writing $\trop^{\ast}\eta$ in polar coordinates $z_j=r_je^{i\theta_j}$ 
and using that
\begin{displaymath}
  \frac{i dz_j\land d\bar z_j}{2 z_j\overline z_j}=\frac{dr_j\land d\theta_j }{r_j},
\end{displaymath}
we deduce
\begin{eqnarray*}
  \int_{V}\trop^{\ast}\eta
  &=&\int _{V}
  f(-\log r_{1},\dots,-\log r_{n})\frac{dr_{1}\dots dr_{n}d\theta
    _{1}\dots d\theta _{n}}{(2\pi)^{n}r_{1}\dots
    r_{n}}\\
  &=&\int _{U} f(x_{1},\dots,x_{n})dx_{1}\dots dx_{n}
  =\int _{U}\eta
\end{eqnarray*}
which proves the claim.
\end{proof}

\begin{art} \label{positivity notions for toric varieties} 
We next discuss the notions of positivity for toric varieties. In the
case of the complex smooth toric variety $\Xsigmaan$ the different
notions of positivity are the usual ones: A complex differential form is \emph{strongly positive}
(resp. \emph{positive}, \emph{weakly positive}) if it is so pointwise.

In the case of $N_{\Sigma }$ a little bit more has to be said because
the different fibers of the sheaves of forms $\AS$ over a point $p\in N_{\Sigma }\setminus
N_{\R}$ have a different nature.

Recall from Subsection \ref{subsection partial compactification} 
that given a point $v\in N_{\Sigma }$, there is an 
unique orbit $N(\sigma )$ with $v\in N(\sigma )$ 
corresponding to a cone $\sigma $ of the fan. 
Then there is an identification of fibers 
\begin{displaymath}
\AS_{N_{\Sigma }}(v)= \AS_{N(\sigma )}(v).
\end{displaymath}
Therefore, the different notions of positivity 
for Lagerberg forms make sense for this
fiber by using Definition 
\ref{Lagerberg involution and orientation} for $V=N(\sigma)$. 
Then for $U\subset N_{\Sigma }$, a Lagerberg form 
is {\it strongly positive}
(resp. {\it positive}, {\it weakly positive}) if it is 
so fiber by fiber.
\end{art}

\begin{lem}\label{lemm:16}
Let $U\subset N_{\Sigma }$ (respectively $V\subset \Xsigmaan$)
be an open subset and $\omega \in \AS(U)$ (respectively 
$\omega \in\AS(V)$) then $\omega $ is strongly positive, 
positive or weakly positive if and only if the same if true for 
$\omega |_{U\cap N_{\R}}$ (respectively 
$\omega |_{V\cap \torusan}$).
\end{lem}

\begin{proof}
This follows from a continuity argument using that the cones of strongly positive, positive and weakly
positive Lagerberg forms in each fiber are closed by Corollary \ref{cor:2}.
\end{proof}

\begin{lem}\label{lemm:7}
Let $U$ be an open subset of $N_{\Sigma }$ and
$V=\trop^{-1}(U)\subset \Xsigmaan$. Let $\eta\in
\AS^{p,p}(U)$.
\begin{enumerate}
\item \label{item:12} If $\eta$ is strongly positive, 
then $\trop^{\ast}(\eta)$ is strongly 
positive.
\item \label{item:13} $\eta$ is positive if and only if 
$\trop^{\ast}(\eta)$ is  
positive.
\item \label{item:14} If $\trop^{\ast}(\eta)$ is weakly  
positive, then $\eta$ is weakly positive.
\end{enumerate}
\end{lem}

\begin{proof}
The result follows from  Lemma \ref{lemm:16} and Lemma \ref{lemm:4}.
\end{proof}

One can deduce from Example \ref{strongly positive not compatible explicit}
that the converses of \ref{item:12} and \ref{item:14} in the previous
lemma are not always true.

\begin{lem}\label{lemm:5} 	
Let $U$ be an open subset of $N_{\Sigma }$ and $V=\trop^{-1}(U)\subset \Xsigmaan$. 
Every strongly positive complex differential form on $V$ (resp. strongly positive Lagerberg form on $U$) is positive and
every positive complex differential form on $V$ (resp. positive Lagerberg form on $U$) is weakly 
positive. 
Moreover, if $p=0,1,n-1,n$, then all three positivity notions agree on $A^{p,p}(U)$ (resp. on $A^{p,p}(V)$).
\end{lem}

\begin{proof}
Since all notions of positivity of forms are checked fiber by fiber,
the result follows from Corollary \ref{cor:1} and Remark \ref{rem:6}
in the complex case and Corollary \ref{cor:2} and Remark \ref{rem:7}
in the Lagerberg case. 
\end{proof}

\section{Positivity for complex invariant currents and Lagerberg currents} 
\label{Complex invariant currents}

Throughout this section, $\Sigma$ will be a {smooth} fan in $N_\R$,
$X_{\Sigma}$ will denote the associated toric variety, $N_\Sigma$ will
denote the partial compactification of $N_\R$ and
$\Xsigmaan = X_{\Sigma}(\C)$ will denote the complex {manifold} 
associated to $X_\Sigma$, with tropicalization map
$\trop \colon \Xsigmaan \to N_\Sigma$.

\subsection{Invariant currents} \label{subsection invariant currents}

Let $U$ be an open subset of $N_{\Sigma }$ and write
$V=\trop^{-1}(U)$. Let $\AS_{c}^{p.q}(U)$ be the space of 
Lagerberg forms on $U$ with compact support and similarly $\AS_{c}^{p,q}(V)$ is the
space of complex forms on $V$ with compact support. 
Since  $\trop \colon V \to U$ is proper, the map $\trop^*$ 
in \eqref{trop-invariant-lagerberg-forms} induces a map
\begin{equation}\label{proper-pullback}
\trop^{\ast}\colon \AS_{c}^{p,q}(U) \longrightarrow \AS_{c}^{p,q}(V).
\end{equation} 
We now compare the topologies on the space of Lagerberg forms and the
space of complex forms through the map $\trop^*$.
The definition of the $C^{\infty}$-topology on both spaces is slightly
different.
The topology of $\AS_{c}^{p,q}(U)$ has
been described in Definition \ref{new-definition-topology}.  The
topology of $\AS_{c}^{p,q}(V)$ is defined similarly. One first defines
a topology on $\AS_{K}^{p,q}(V)$ for each compact $K$ using norms
similar to those in \eqref{new-distribution-norm} taking into account
all complex derivatives and then define the topology of
$\AS_{c}^{p,q}(V)$ as the direct limit in the category of locally
convex topological vector spaces. So the main difference is the use of
finite coverings in Definition \ref{new-definition-topology}.
From the definition of the topologies, it is easy to check that the map 
$\trop^{\ast} \colon \AS_{c}^{p,q}(U) \to\AS_{c}^{p,q}(V)$ is
injective and  continuous.
Nevertheless, as the following examples show, 
$\trop^*$ is in general not a homeomorphism onto its image endowed with the subspace topology. 
Moreover, the image of $\trop ^{\ast}$ is not closed in $\AS_{c}^{p,q}(V)$.

\begin{ex}\label{exm:2} 
As in Remark \ref{rem:4}, we consider the case $ X_\Sigma= \mathbb A_\C^1$, 
thus $N_\Sigma = \Rsup$.
Consider   $U=\{u \in \Rsup \mid u>0\}\subset N_\Sigma$ and 
$V=\trop^{-1}(U)=\{z \in \C \mid z\bar z < 1\}$.
Let $\rho \colon \R \to \R$ be a smooth function with $0\le
\rho \le 1$ and $\rho (x)=0$ for $|x|>2$ and $\rho (x)=1$ for
$|x|<1$. Define the sequence of functions
\begin{displaymath}
 f_{n}(u)=\frac{1}{e^{n^{2}}}\rho (u-n).
\end{displaymath}
The sequence $(f_{n})_{{n \geq 3}}$ does not converge to zero in $A^{0,0}_c(U)$ 
because for any compact subset $K$ of $U$ there is no finite covering
$\{V_i\}_{i}$ of $K$ with $f_{n}\in A^{0,0}_K(U,\{V_{i}\})$ for all
$n$ (see Proposition \ref{convergent sequences}). 
Indeed, assume that such a compact $K$ and covering $\{V_{i}\}$
exists. Since the point $n\in \Rsup$ is in the support of $f_{n}$, and
these points converge to $\infty$, then $\infty\in K$. Therefore one
$V_i$ would contain $\infty$ and hence 
by definition of $A^{0,0}_K(U,\{V_{i}\})$
all the $f_n|_{V_i}$ 
would have to be constant, which is not the case.

Write $\varphi_{n}=\trop^{\ast}f_{n}$. Then
\begin{displaymath}
\varphi_{n}(z)=\frac{1}{e^{n^{2}}}\rho (-\log |z|-n).
\end{displaymath}
The support of $\varphi_{n}$ is contained in the closed annulus 
$\{z \in \C \mid e^{n-2}\le{|z|^{-1}}\le e^{n+2}\}$. 
Since $\rho $ is smooth with compact support, all the
derivatives of $\rho $ are bounded. From this and the
condition above on the support of $\varphi_{n}$, we deduce that 
for every pair of 
integers $a$, $b$, there is a constant $C_{a,b}$  such that
      \begin{displaymath}
      \left| \left(\frac{\partial}{\partial z}\right)^a \left(\frac{\partial}{\partial\overline z}\right)^b \varphi_n\right|
      \le \frac{1}{e^{n^{2}}}C_{a,b}\sup_{z \in
      	\supp(\varphi_{n})}\frac{1}{|z|^{a+b}}\le
      \frac{1}{e^{n^{2}}}C_{a,b}e^{(a+b)(n+2)}.
      \end{displaymath}
It follows that the sequence $(\varphi_{n})_{n\geq 3}$ converges to 
      the function 
      $0$ in $\AS_c^{0,0}(V)$.
We conclude that the topology of
$\AS_c^{0,0}(U)$ is not induced by the topology of
$\AS_c^{0,0}(V)$.
  \end{ex}

\begin{ex}\label{exm:2'} We keep the setting from Example \ref{exm:2}. 
Choose now $\rho \in A^{0,0}(\R)$ with $0\le \rho \le 1$ and 
$\rho (x)=1$ for $x\le 0$ and $\rho(x)=0$ for $x>1$. 
Further pick $\psi \in A_{c}^{0,0}(U)$ with $\psi
  (x)=1$ for $x \geq 1$.
Consider the sequence of differential forms
\begin{displaymath}
\eta_{n}= \psi(- \log \vert z \vert) e^{-1/|z|}\rho (-\log|z|-n) 
\frac{idz\land d\bar z}{{4}\pi z\bar z}\in A_c^{1,1}(V)\,\,\,\,\,(n\in \N)
\end{displaymath}
and the differential form $\eta= \psi (- \log \vert z \vert) e^{-1/|z|}{i}dz\land 
d\bar z/(4{\pi}z\bar z)$. 
Similarly as in Example \ref{exm:2} the forms $\eta_{n}$
converge to $\eta$ in $\AS^{1,1}_c(V)$. 
The forms $\eta_{n}$ are in the image of $\trop ^{\ast}$ as 
\begin{displaymath}
\eta_{n}=\trop^{\ast} \left( \psi(u) \exp(-e^{u})\rho (u-n)d'u\land d''u\right),
\end{displaymath}
but the form $\eta$ is not in the image of $\trop^{\ast}$
because the function $\psi (u)e^{-e^{u}}$ is non-constant close
to $\infty$. 
Hence the image of $\trop^{\ast} \colon A^{1,1}_c(U) \to A^{1,1}_c(V)$
is not closed.
\end{ex}

Now we shift our attention to currents.

\begin{rem} \label{F on currents}
Let $V\subset \Xsigmaan$ be an $\SS$-invariant open subset.
Similarly as in Remark \ref{rem:2}, the average with respect to the
probability Haar measure on $\SS$  
leads to a canonical projection $A_c(V) \to A_c(V)^{\SS}$ that we denote
$\omega \mapsto \omega ^{\av}$. A current $S$ is called
\emph{$\SS$-invariant} if 
\begin{displaymath}
  a_{\ast}S = S\quad \forall a \in \SS.
\end{displaymath}
The space of $\SS$-invariant currents is denoted $D(V)^{\SS}=\bigoplus
D^{\cdot,\cdot}(V)^{\SS}$. There is a canonical projection $D(V) \to D(V)^{\SS}$,  
where $S$ is mapped to the $\SS$-invariant current $S^{\av}$ given by
$S^{\av}(\omega)=S(\omega^{\av})$.  
The antilinear involution $F$ of $\AS(V)^{\SS}$ leaves $\AS_{c}(V)^{\SS}$
invariant. 
It defines an antilinear involution on the space of $\SS$-invariant
complex currents  
$D(V)^{\SS}$, which we also denote by $F$, given by
\begin{equation}\label{eq:37}
(FS)(\eta)=\overline{S(F\eta)} \quad\quad (\eta \in A_c(V)^\SS).
\end{equation}
The space of $\SS$- and $F$-invariant currents is denoted by $D(V)^{\SS,F}$. 
Observe that $F$ respects the grading. 
Hence the spaces $D^{p,q}(V)^{\SS,F}$ are well-defined. 
Since $F$ is an involution, there is a canonical projection
$D(V)^{\SS}\to D(V)^{\SS,F}$. 
\end{rem}

\begin{defi}
The \emph{direct image of currents} is the map
\begin{displaymath}
  \trop_{\ast}\colon D^{p,q}(V)\longrightarrow
  D^{p,q}(U){\otimes_\R }\C
\end{displaymath}
given by 
$ \trop_{\ast}(S)(\eta)=S(\trop^{\ast}\eta)$
for each $\eta \in A_c^{n-p,n-q}(U)$. 
{By linearity, we extend the direct image to a $\C$-linear map 
$\trop_*\colon D(V) \to D(U){\otimes_\R }\C$, where $D(U) \coloneqq\oplus_{p,q}D^{p,q}(U)$ is the total 
space of real Lagerberg currents.}
\end{defi}

\begin{lem} \label{trop of complex current}
For  an open subset $U$ of $N_\Sigma$, $V \coloneqq \trop^{-1}(U)$ and $S \in D(V)$, we have the following properties. 
\begin{enumerate}
	\item \label{average}
	 $\trop_*(S)=\trop_*(S^\av)$
	\item \label{real push-forward}
	If $S\in D(V)^{{\SS},F}$, then $\trop_{\ast}(S)\in D(U)\subset D(U){\otimes_\R }\C$.
	\item \label{converse of real push-forward}
	If $U \subset N_\R$ and $S \in D(V)^\SS$,  then  $S$ is
	$F$-invariant if and only if  
	$ \trop _{\ast}(S) \in D(U)$.
\end{enumerate}
\end{lem}

\begin{proof}
Proposition \ref{prop:5} yields \ref{average}. 
To prove the remaining claims, let $S \in D(V)^\SS$. 
Using that $F$ is an involution, we see that \eqref{eq:37} is equivalent to $\overline{S(\eta)}=(FS)(F\eta)$ for every $\eta\in
{A_{c}(V)^\SS}$. 
To prove \ref{real push-forward}, we assume that $S$ is $F$-invariant. 
We have to show $S(\trop^*\omega)\in \R$ for all $\omega \in A_{c}(U)$. 
Using that $S$ and $\trop^{\ast}\omega $ are
$F$-invariant, property \ref{real push-forward} follows from
\begin{displaymath}
\overline {\trop_{\ast}S(\omega )}
=\overline {S(\trop^{\ast}\omega)}
=(FS)(F\trop^{\ast}\omega )
=S(\trop^{\ast}\omega )=\trop_{\ast}S(\omega ).
\end{displaymath}
To show \ref{converse of real push-forward}, let $U \subset N_\R$ and assume $ \trop _{\ast}(S) \in D(U)$. Then the same computation in reversed order shows $\overline{S}(\eta)=(FS)(F\eta)$ for every $\eta$ of the form $\trop^*\omega$ with $\omega \in A_{c}(U)$. 
Using $U\subset N_\R$, we  deduce from \eqref{added eqnr} and \eqref{eq:22} that the $\C$-span of such forms is $A_{c}(V)^\SS$ and hence $\overline{S}(\eta)=(FS)(F\eta)$ holds for all $\eta \in A_c(V)^\SS$ proving $F$-invariance of $S$ and \ref{converse of real push-forward}.
\end{proof}

\begin{rem} \label{rem:3}  
The map $\trop_{\ast}$ fits in the following  commutative diagram:
\begin{displaymath}
    \xymatrix{
D(V) \ar@(ur,ul)[rr]^{\trop_{\ast}}\ar[r] &D(V)^{\SS} \ar[r]_-{\trop_{\ast}}
&D(U)\underset{\R}\otimes \C\\
&D(V)^{\SS,F} \ar[r]^-{\trop_{\ast}}
\ar[u]&D(U).\ar[u]}
\end{displaymath}
If $U$ is an open subset of the dense orbit $N_\R$, 
then {Corollary \ref{top-coro}} yields that the map 
$D(V)^{\SS,F} \xrightarrow{\trop_{\ast}}	D(U)$ is an isomorphism. 
 If $U$ intersects the boundary, then $\trop^{\ast}$ is not a closed immersion and 
hence $\trop_{\ast}$ is not surjective (see  Examples
  \ref{exm:1} and \ref{exm:1'} below).
\end{rem}

\begin{lem}\label{lemm:8}
Given $\omega\in A^{p,q}(U)$ and $S\in D^{r,s}(V)$, we have
the projection formula
\begin{equation}\label{trop-projection-formula}
\omega\wedge \trop_*S=\trop_*\bigl(\trop^*(\omega)\wedge S\bigr).
\end{equation}  
\end{lem}
\begin{proof}
Given $\eta \in A_c^{n-p-r,n-q-s}(U)$, by
Remark \ref{product-partial-comp}\ref{product-partial-comp2} we have
\begin{align*}
S\bigl(\trop^*(\omega \wedge\eta )\bigr)=
(-1)^{(p+q)(r+s)}\bigl(\trop^*(\omega)\wedge S\bigr)(\trop^*\eta)
\end{align*}
as $\trop^*$ respects products and \eqref{trop-projection-formula} is an immediate consequence. 
\end{proof}

Recall that for $\omega \in A^{p,q}(U)$, 
we have an associated Lagerberg current $[\omega]= \int_U\omega \wedge .$ in $D^{p,q}(U)$.
Similarly, we have a complex current $[\eta] \in D^{p,q}(V)$
associated to a complex form
 $\eta \in A^{p,q}(V)$. 
From Lemma \ref{comparison of integration} and \ref{lemm:8}, 
we immediately deduce the following result.

\begin{cor}  \label{pushforward pullback}
  For every Lagerberg form $\omega \in A^{p,q}(U)$, we have
    \begin{displaymath}
      \trop_* [\trop^*(\omega)] = [\omega]\in D^{p,q}(U).
    \end{displaymath}
\end{cor}

\begin{defi}\label{definition positive currents on partial compactification}
   Let $V\subset \Xsigmaan$ be an open subset
and let $S\in D^{p,p}(V)$ be a current.
  The current $S$ is called 
\emph{strongly positive (resp.~positive, resp.~weakly positive)} 
if $S(\eta)\ge 0$ for every weakly positive (resp.~positive, resp.~strongly positive) form
$\eta\in \AS^{n-p,n-p}_{c}(V)$.
The space of  weakly positive $(p,p)$-currents will be denoted
$D^{p,p}_{w+}(V)$, the space of positive ones by 
$D^{p,p}_{+}(V)$ and  the space of strongly positive ones by 
$D^{p,p}_{s+}(V)$. 
{The current $S$ is called \emph{real} if
we have $S(\bar\eta)=\overline{S(\eta)}$ 
for each $\eta\in \AS^{n-p,n-p}_{c}(V)$}.

Let $U\subset N_{\Sigma }$ be an open subset and let 
$T\in D^{p,p}(U)$ be a Lagerberg current.
The Lagerberg current $T$ is called 
\emph{strongly positive (resp.~positive, resp.~weakly positive)} if it
is symmetric and for every weakly
positive (resp.~positive, resp.~strongly positive) Lagerberg form
$\eta\in  \AS^{n-p,n-p}_{c}(U)$
the condition $T(\eta)\ge 0$ holds.
The space of  weakly positive Lagerberg currents 
of type $(p,p)$ will be denoted
$D^{p,p}_{w+}(U)$, the space of positive ones by 
$D^{p,p}_{+}(U)$, and  the space of strongly positive ones by 
$D^{p,p}_{s+}(U)$.  
\end{defi}

\begin{lem}\label{lemm:6} 
Let $U$ be an open subset of $N_{\Sigma }$ and $V=\trop^{-1}(U)\subset \Xsigmaan$. 
Every strongly positive current in $D^{p,p}(V)$ (resp.~in $D^{p,p}(U)$) is 
positive, 
every positive current in $D^{p,p}(V)$ (resp.~in $D^{p,p}(U)$) is weakly positive {and every weakly positive current in $D^{p,p}(V)$ (resp.~in $D^{p,p}(U)$) is real (resp.~symmetric).} 
Moreover, if  $p=0,1,n-1,n$, then all three positivity notions agree in $D^{p,p}(V)$ (resp.~in $D^{p,p}(U)$).
\end{lem}

\begin{proof}
The positivity claims follow from  Lemma \ref{lemm:5}. By definition, a weakly positive current on $U$ is symmetric. 
{To see that a weakly positive current on $V$ is real, we use Lemma \ref{lemm:13}. The latter implies  that any real form on $V$ is a real linear combination of strongly positive ones which easily yields the claim.} 
\end{proof}

We will see in Theorem \ref{thm:5} that $\trop_*$ 
induces a bijection between the space of closed positive complex
currents on {$V= \trop^{-1}(U)$} that are invariant under $F$ and $\SS$ and the space of
closed positive Lagerberg currents on {$U$}. In particular, every closed
positive Lagerberg current is in the image of $\trop_{\ast}$.
The following two counterexamples show that one can drop neither the
closedness assumption 
nor the positivity assumption.
The first example is a positive Lagerberg current and the second
example is a closed
Lagerberg current, neither of which are in the image of $\trop_*$.

\begin{ex}\label{exm:1} 
Let $X_\Sigma= \mathbb A_\C^1$, $N_\Sigma =\Rsup$, 
$U=\{u \in \Rsup \mid u>0\}$ and $V=\trop^{-1}(U)=\{z \in \C \mid z\bar z <1\}$ as in 
Example \ref{exm:2}. 
Let $T\colon A^{1,1}_{c}(U)\to \R$ be the linear map 
\begin{align} \label{equation exm:1}
T(fd'u\land d''u)=\int_{{0}}^{\infty}e^{x^{2}}f(x)dx.
\end{align}
If $fd'u\land d''u\in A^{1,1}_{c}(U)$, then $f$ has compact support
contained in $U\setminus \{ \infty \}$.
Therefore the integral in (\ref{equation exm:1}) is finite. 
To show that $T$ is a continuous functional, we use Proposition 
  \ref{test-proposition}. Let $\omega _{n}=g_{n}d'u\land d''u$, $n\ge
1$ be a sequence of forms in $A^{1,1}_{c}(U)$ that converge to
zero. Then there is a neighborhood $V_{0}$ of $\infty$ such that
\begin{displaymath}
  \omega _{0,n}|_{V_{0}}=\pi _{\sigma ,0}^{\ast}(\omega _{\sigma ,n})=0
\end{displaymath}
for all $n$,
where $\sigma =\R_{\ge 0}$ is the cone corresponding to the point
$\infty$. Hence there is a compact $K\subset U\setminus \{\infty\}$
such that $\supp(g_{n})\subset K$ for all $K$. Moreover
$\|g_{n}\|_{K}$ converges to zero. Therefore $T(\omega _{n})$ also
converges to zero.
We conclude that $T$ is a Lagerberg current. 
Clearly it is positive.

Assume that $S$ is a current on $V$ such that $\trop_{\ast}S=T$. 
Let $f_{n}$ and $\varphi_{n}$ be as in Example
\ref{exm:2}.
Since the functions $\varphi_{n}$ converge to  $0$ in $A^{0,0}_{c}(V)$, we deduce that
\begin{align*}
\lim_{n\to \infty} S \left(\varphi_{n} \frac{idz \land d\overline z}
{4{\pi}z \overline z} \right) = 0.
\end{align*}
On the other hand, the assumption $\trop_{\ast}{S}=T$ yields the desired contradiction as follows:
\begin{displaymath}
{S} \left( \varphi_{n} \frac{idz \land d\overline z}
{4{\pi} z \overline z} \right) =
T \left( f_{n} d'u \land d''u \right)\\ 
=\int_{{0}}^{\infty} e^{x^{2}-n^{2}} \rho (x-n)dx >
\int_{n}^{n+1}e^{x^{2}-n^{2}}dx > 1.
\end{displaymath}
\end{ex}

\begin{ex}\label{exm:1'} 
We keep the setting from Example \ref{exm:1}. Then 
\begin{equation}\label{eq:6}
  T(f) \coloneqq \int _{{0}}^{\infty}e^{2e^{x}}f'(x)dx.
\end{equation}
defines a linear functional $T\colon A^{0,0}_{c}(U)\to \R$. Since any $f\in A^{0,0}_{c}(U)$ is {constant in a neighborhood of $\infty$}, 
the support of $f'$ is a compact subset of  $U\setminus \{\infty\}$.
As in Example \ref{exm:1}, we see 
that $T$ defines a current in $D^{1,1}(U)$.
This current is closed because it has top degree. 
We claim that $T$ is not of the form $\trop_{\ast}({S})$ for any 
${S}\in D^{1,1}(V)$.
We argue by contradiction and 
suppose that $T=\trop_{\ast}({S})$. Let $\rho $ and $\psi$ be as in
Example \ref{exm:2'} and write
$$
  \zeta_n=\psi(- \log \vert z \vert) e^{-1/|z|}\rho (-\log|z|-n), \quad 
  \zeta=\psi(- \log \vert z \vert) e^{-1/|z|}.
$$
Then $\zeta_n \to \zeta$ in $A^{0,0}_c(V)$. Moreover
$\zeta_{n}=\trop^{\ast}(f_{n})$ for
\begin{displaymath}
  f_{n}(x)=\psi(x) e^{-e^{x}}\rho (x-n)\in A^{0,0}_{c}(U).
\end{displaymath}
Therefore 
$
 \lim_{n\to \infty}{S}(\zeta_{n})={S}(\zeta) \in \C$. 
 This contradicts 
\begin{align*}
{S}(\zeta_{n}) = T(f_{n})=
\int_{{0}}^\infty e^{2e^x}\left(\psi(x)e^{-e^{x}}\rho (x-n)) \right)'dx
\end{align*}
converging to $-\infty$ for $n \to \infty$, as a direct computation
using integration by parts shows. 
\end{ex}

\begin{prop} \label{positivity and trop}
{Let $S\in D^{p,p}(V)^{\SS,F}$. 
If $S$ is real (resp.~positive, resp.~weakly positive), 
then $T\coloneqq\trop_*(S)$ is symmetric (resp.~positive, resp.~weakly positive).}
\end{prop}

\begin{proof}
{Recall from Lemma \ref{trop of complex current} that $T\in D^{p,p}(U)$.
If $S$ is real, then the Lagerberg current $T$ is symmetric.
Indeed for $q \coloneqq n-p$ and $\omega \in A_c^{q,q}(U)$, 
we deduce from \eqref{eq:18'} that 
\[
T(J(\omega))
=S\bigl(\trop^*(J(\omega))\bigr)
=(-1)^pS( \overline{\trop^*(\omega)})
=(-1)^p\overline{S( \trop^*(\omega))}
=(-1)^p{T(\omega)}.
\]
The positivity statements follow from Lemma \ref{lemm:7}.}	
\end{proof}

On the dense orbit, we also have the converse implication.

\begin{lem} \label{generic positivity and trop}
Let $U$ be an open subset of $N_\R$, 
put $V\coloneqq \trop^{-1}(U)$, and let ${S}\in D^{p,p}(V)^{\SS,F}$
be an invariant current such that $\trop_*({S})$ is a positive current on $U$.
Then ${S}$ is a positive current on $V$. 
\end{lem}

\begin{proof}
Given a positive form $\alpha\in A_c^{n-p,n-p}(V)$,
the averaged form $\alpha^{\rm av}\in A_c^{p,p}(V)^\SS$ 
from Remark \ref{rem:2} is again positive. 
From Lemma \ref{lemm:15}, we get that the invariant form
\begin{align*}
\alpha^{\rm inv}=(\alpha^{\rm av}+F(\alpha^{\rm av}))/2\in 
A_c^{n-p,n-p}(V)^{\SS,F}
\end{align*} 
is positive as well.
Since $U\subset N_\R$, we conclude by Proposition \ref{prop:4} that 
$\alpha^{\rm inv}=\trop^*(\beta)$ for a uniquely
determined Lagerberg form $\beta$ in $A_c^{p,p}(U)$.
Finally $\beta$ is positive by Lemma {\ref{lemm:4}}  and 
${S}(\alpha)={S}(\alpha^{\rm inv})=\trop_*({S})(\beta)\geq 0$ shows our claim.
\end{proof}

We will see in the next subsection that positive currents have measure coefficients. The key result to prove
this is captured in the following proposition.
{We refer the reader to Appendix \ref{facts from measure theory} for
the convention we use about Radon measures.}

\begin{prop}\label{positive_distribution_is_a_measure} 
Let $U$ be either an open subset of $\Xsigmaan$ or of $N_{\Sigma }$. 
Let $T \colon A^{0,0}_c(U) \to \R$ be a linear functional 
such that $T(f)\ge 0$ for every non-negative $f \in  A^{0,0}_c(U)$.
Then there is a unique positive Radon measure $\mu$ on $U$ with 
$T(f)=\int_U f \, d\mu$ for every $f \in A^{0,0}_c(U)$. 
\end{prop}

\begin{proof}
The following argument works for the complex and the Lagerberg case.

Let $K$ be a compact subset of $U$ and let $C^0_K(U,\R)$ 
be the set of continuous real functions with support in $K$. 
As usual, this real vector space is endowed with the supremum norm 
$\metr_{\rm sup}$. We also consider the subspace $A_K^{0,0}(U,\R)$ of 
smooth real functions with support in $K$
and its subspace $A_K^{0,0}(U,\R_{\geq 0})$ of non-negative functions. 

We claim that the restriction of $T$ to $A_K^{0,0}(U,\R)$ 
is continuous with respect to the sup-norm. 
Since we have smooth partitions of unity by Proposition 
\ref{extended-partition-unity}, there is a  non-negative 
 $\chi_K\in  A^{0,0}_c(U) $ with 
$\chi_K(x)=1$ for all $x \in K$. 
Now the positivity of $T$ yields 
\begin{align*}
-T(\chi_K)\|f\|_{\rm sup} \leq T(f) \leq T(\chi_K) \|f\|_{\rm sup}
\end{align*}
for all smooth functions $f$ with compact support in $K$.
This proves the claim. 

Let $f \in C^0_c(U,\R)$. 
To define $T(f)$, we use Corollary \ref{apply-stone-weierstrass} 
or its classical analogue to
get a sequence $(f_n)_{n \in \N}$ in 
$A_c^{0,0}(U,\R)$ which converges to $f$.
It follows from the proof of that corollary
as well as from
Proposition \ref{convergent sequences} that we can find a compact subset $K$ of
$U$ such that $f_n\in A_K^{0,0}(U,\R)$ for all $n\in \N$.
The sequence $(f_n)_{n \in \N}$ converges to $f$ with respect to the 
sup-norm on $K$.
By the continuity of $T$ shown above, it is clear that $T(f_n)$ converges 
to a real number. 
The limit $T(f)$ neither depends on the choice of the 
sequence nor on the choice of $K$.  
Note also that if $f \geq 0$, then 
$\chi_K\cdot(f_n+ \|f_n-f\|_{\rm sup})$ is a 
sequence in $C^0_{\supp(\chi_K)}(U,\R_{\geq 0})$ 
converging to $f$ with respect to the sup-norm.  
We conclude easily that the above defines a positive linear functional 
$T$ on $C^0_c(U,\R)$ which extends the given functional. This is equivalent 
to have a positive Radon measure $\mu$  on $U$ with $T(f)=\int_U f \, d\mu$ 
for every $f \in C^0_c(U,\R)$.

Note that uniqueness is obvious as a positive linear functional 
on $C^0_c(U,\R)$ is continuous on $C^0_K(U,\R)$ 
(by the same argument as above) 
and so our definition of $T(f)$ is forced.
\end{proof}

We can now prove that positive $(n,n)$-currents on
open subsets of $\Xsigmaan$ or $N_{\Sigma }$ are the same as positive
Radon measures.

\begin{prop} \label{positive currents in top degree}
Let $U$ be either an open subset of $\Xsigmaan$ or of $N_{\Sigma }$. 
For every positive current  $T \in D^{n,n}(U)$, there is a unique positive 
Radon measure $\mu$ on $U$ such that 
\begin{equation} \label{integral equation}
T(f)=\int_U f \, d\mu \quad \forall f \in A^{0,0}_c(U).
\end{equation}
Conversely,  every positive Radon measure 	
$\mu$ on $U$ induces a current $T \in D_+^{n,n}(U)$ by \eqref{integral equation}.
\end{prop}

\begin{proof}
If $T$ is a positive current in $D^{n,n}(U)$, then it is a positive 
linear functional on $A_c^{0,0}(U)$. 
Hence the first claim follows from Proposition \ref{positive_distribution_is_a_measure}. 
Conversely, we have seen in Example \ref{measure gives current in top degree}
that every real Radon measure $\mu$ induces an element $T \in D^{n,n}(U)$ with \eqref{integral equation}. 
Clearly, if $\mu$ is positive, then $T$ is positive as well.
\end{proof}

The following result is {a prototype} for 
our main correspondence result in Theorem \ref{thm:5}.

\begin{cor} \label{correspondence for top degree currents}
Let $U$ be an open subset of $N_\Sigma$ and let $V \coloneqq \trop^{-1}(U)$. 
Then the map $\trop_{\ast}$ induces a linear isomorphism 
$D^{n,n}(V)_+^\SS \to D^{n,n}(U)_+$ of real cones.	
\end{cor}

\begin{proof}
Since $N_\Sigma=\Xsigmaan/\SS$ (see Remark \ref{ass-toric-variety}), 
it is clear that $\trop^*$ induces an isomorphism between $C_c^0(U,\R)$ and the 
subspace of $C_c^0(V,\R)$ given by the $\SS$-invariant functions. 
Hence $\trop_*$ maps the cone of $\SS$-invariant positive Radon 
measures on $V$ isomorphically onto the cone of positive Radon measures on $U$. 
By Proposition \ref{positive currents in top degree}, we get the claim.
\end{proof}

\subsection{Co-coefficients of currents}  
\label{Co-coefficients subsection}

To extend Proposition \ref{positive_distribution_is_a_measure} to 
$(p,p)$-currents, we place ourselves in the local setting. 
We assume in this subsection that 
$N=\Z^{n}$ and that $\Sigma $ is the fan given by
the maximal cone $\R_{\geq 0}^n$ and its faces in
$N_\R=\R^n$. 
Then $N_\Sigma = \Rsup^n$ and $X_{\Sigma}$  
is the smooth toric variety $\A_\C^{n}$. 
Recall that any smooth toric variety can be covered by toric affine
varieties isomorphic to open subvarieties of this one.
The splitting $N=\Z^{n}$ induces coordinates $(z_{1},\dots,z_{n})$
in $ \Xsigmaan$
and $(u_{1},\dots,u_{n})$ in $N_{\R}$.

\begin{defi} \label{definition co-coefficitents C}
Let $V\subset \Xsigmaan = \C^n$ be an open subset and let 
${S}\in D^{p,p}(V)$, write $q\coloneqq n-p$ and let 
$I,J\subset \{1,\dots,n\}$ with $|I|=|J|=q$. 
Then the {\it co-coefficient} ${S}^{IJ}$ of ${S}$ is the 
distribution on $V$ defined by
\begin{displaymath}
  {S}^{IJ}(f)={S}(i^{q^2}f dz_{I}\land d\bar z_{J}).
\end{displaymath}
\end{defi}

For $I\subset \{1,\dots,n\}$, we consider the following union of strata
\begin{displaymath}
E^{I} \coloneqq \bigcup_{i\in I}  \{u_{i}=\infty\} \subset N_\Sigma = \Rsup^n.
\end{displaymath}

\begin{defi}\label{definition co-coefficitents L}
Let $U\subset N_{\Sigma } = \Rsup^n$ be  open  and
$T\in D^{p,p}(U)$. Let $q,I,J$ be as above.
Then the {\it co-coefficient} 
$T^{IJ}$ is the Lagerberg distribution  on 
$U\setminus E^{I\cup J}$ defined by
\begin{displaymath}
T^{IJ}(f)=T\bigl((-1)^{\frac{q(q-1)}{2}}f d'u_{I}\land d'' u_{J}\bigr).
\end{displaymath}
\end{defi}

\begin{rem}Note that we define $T^{IJ}\in D^{n,n}(U\setminus E^{I\cup J})$ 
because, by definition of Lagerberg forms, we can only plug
in functions that vanish in a neighborhood of 
$E^{I\cup J}$.
\end{rem}

\begin{rem} \label{co-coefficients determine currents}
Conversely, given  ${S}^{IJ} \in D^{n,n}(V)$ for all $I,J\subset \{1,\dots,n\}$ with
$|I|=|J|=q=n-p$, there is a unique complex current ${S} \in D^{p,p}(V)$ with co-coefficients ${S}^{IJ}$. 
  	
Similarly, one can show that given  $T^{IJ}\in D^{n,n}(U\setminus E^{I \cup J})$ 
for all $I,J\subset \{1,\dots,n\}$ with $|I|=|J|=q=n-p$,
there is a unique $T \in D^{p,p}(U)$ with co-coefficients $T^{IJ}$. 
\end{rem}

Positive currents have measure co-coefficients, instead of just distribution co-coefficients.
The complex case is given by \cite[Proposition III.1.14]{demailly_agbook_2012}:

\begin{prop}\label{prop:7}
Let $V\subset \Xsigmaan = \C^n$ and ${S}\in D^{p,p}(V)$ be a weakly positive current. 
Then the co-coefficients ${S}^{IJ}$ are complex Radon measures on $V$ 
that satisfy $\overline{{S}^{IJ}}={S}^{JI}$ for all multi-indices $I,J$ with $|I|=|J|=q$. 
Moreover, all ${S}^{II}$ are positive Radon measures and the total variation measures $|{S}^{IJ}|$ satisfy the estimates
\begin{equation}\label{eq:10}
\lambda _{I}\lambda _{J}|{S}^{IJ}| \le 2^{q}\sum_{M}\lambda ^{2}_{M}{S}^{MM}, 
\quad I\cap J \subset M  \subset I\cup J,
\end{equation}
for any collection of real numbers  $\lambda _{k}\ge 0$,
$k=1,\dots,n,$ and $\lambda _{I}=\prod_{k\in I}\lambda _{k}.$
\end{prop}

We have seen in Example \ref{too few strongly positive} 
that in the tropical case it is reasonable to restrict our attention
to positive Lagerberg forms and
hence to positive Lagerberg currents, as opposed to weakly positive currents.
    
\begin{prop}\label{prop:6}
Let $U\subset N_{\Sigma } = \Rsup^n$ be an open subset and $T\in D^{p,p}(U)$ a positive current. 
Then the co-coefficients $T^{IJ}$ are real Radon measures on $U\setminus E^{I\cup J}$ 
that satisfy $T^{IJ}=T^{JI}$ for all multi-indices $I,J$ with $|I|=|J|=q$. 
Moreover, all $T^{II}$ are positive Radon measures and the
total variation measures $|T^{IJ}|$ satisfy the estimates
\begin{equation}\label{eq:30}
\lambda _{I}\lambda _{J}|T^{IJ}| \le \frac{1}{2} (\lambda ^{2}_{I}T^{II}+ \lambda ^{2}_{J}T^{JJ}), 
\end{equation}
for any pair of real numbers  $\lambda _{I}, \lambda _{J}\ge 0$.
\end{prop}
\begin{proof}
By definition, Lagerberg currents are real valued. 
The condition
\begin{align} \label{eq:28}
T^{IJ}=T^{JI}
\end{align}
follows from the symmetry of $T$.
For every $f \in A_c^{0,0}(U\setminus E^{I\cup J})$, 
we have
\begin{align}\label{eq:26}
(T^{IJ}+T^{JI})(f) = 
T \left( f\cdot(-1)^{\frac{q(q-1)}{2}}  (d'u_{I}\wedge d''u_{J}+d'u_{J}\wedge d''u_{I}) \right).
\end{align}
Moreover 
$(-1)^{\frac{q(q-1)}{2}}(d'u_{I}\wedge d''u_{J}+d'u_{J}\wedge d''u_{I})$ 
can be written as the  difference of two positive Lagerberg forms
 \begin{align}\label{eq:27}
\nonumber (-1)^{\frac{q(q-1)}{2}} (d'u_{I}\wedge d''u_{J}+d'u_{J}\wedge d''u_{I}) &= 
(-1)^{\frac{q(q-1)}{2}} (d'u_{I}+d'u_{J}) \wedge (d''u_{I}+d''u_{J}) \\ 
&- (-1)^{\frac{q(q-1)}{2}}(d'u_{I}\wedge d''u_{I} + d'u_{J}\wedge d''u_{J}).
\end{align}
Since the product of a positive Lagerberg current by a positive Lagerberg form is a 
positive Lagerberg current by Remark \ref{rem:5}, 
we deduce from equations \eqref{eq:26}, \eqref{eq:27} 
and from Proposition \ref{positive_distribution_is_a_measure} 
that $T^{IJ}+T^{J,I}$ is a real Radon measure on $U\setminus E^{I\cup J}$. 
From equation \eqref{eq:28}, 
it follows that $T^{IJ}$ is also a real Radon measure on $U\setminus E^{I\cup J}$. 
In the special case $I=J$, there is no need for \eqref{eq:27} 
and we can directly deduce from the definition of the 
co-coefficient $T^{II}$ or from \eqref{eq:26} that 
$T^{II}$ is a positive Radon measure on $U\setminus E^{I}$.

Note that for every positive continuous function with compact support 
contained in $U\setminus E^{I\cup J}$ and 
for every pair of non-negative real numbers $\lambda _{I}, \lambda _{J}$, we have the inequalities
\begin{displaymath}
T\left(f\cdot (-1)^{\frac{q(q-1)}{2}}(\lambda _{I}d'u_{I}\pm \lambda _{J}d'u_{J})\wedge
(\lambda _{I}d''u_{I}\pm\lambda _{J}d''u_{J})\right)\ge 0.
\end{displaymath}
Expanding and using equation \eqref{eq:28}, we deduce
\begin{displaymath}
\pm 2 \lambda _{I}\lambda_{J} T^{IJ}(f) \le  
\lambda^{2}_{I}T^{I,I}(f)+ \lambda ^{2}_{J}T^{J,J}(f),  
\end{displaymath}
from which equation \eqref{eq:30} follows. 
\end{proof}

Observe that  \eqref{eq:30} is stronger 
than  \eqref{eq:10}. 
This is because in
Proposition \ref{prop:6} we are dealing with positive currents
while Proposition \ref{prop:7} deals with weakly positive currents.

\begin{ex}\label{measures and weakly positive Lagerberg currents1}
We show that there can be no analogue of Proposition \ref{prop:6} 
for weakly positive Lagerberg currents.

Example \ref{too few strongly positive} gives a non-zero Lagerberg form $\omega$  
of type $(2,2)$ on $\R^4$ such that $\pm \omega$ are both weakly positive. 
In fact $\omega$ was constructed such that $\omega\wedge \eta=0$ for each
strongly positive Lagerberg form $\eta$ of type $(2,2)$.
By definition, the associated Lagerberg currents $\pm [\omega]$ are
weakly positive.
We compute the co-coefficients of $T=[\omega]$. 
For $I=\{3,4\}$, $J=\{1,2\}$ and $f \in C_c^\infty(N_\R)$, we have 
\begin{align*}
T^{JI}(f)=T^{IJ}(f)
=-T(fd'u_3 \wedge d'u_4 \wedge d''u_1 \wedge d'' u_2)
=\int_{\R^4} f(u)du_1du_2du_3du_4
\end{align*}
by using the  formula for $\omega$ in 
Example  \ref{too few strongly positive}.
We conclude that $T^{JI}=T^{IJ}$ is   
the Lebesgue measure on $\R^4$ and
the same holds for $I=\{2,3\}$ and $J=\{1,4\}$. 
On the other hand, for $I=\{2,4\}$, $J=\{1,3\}$, we have that
$-T^{JI}=-T^{IJ}$ agrees 
with the Lebesgue measure on $\R^4$. 
All other co-coefficients are $0$. In particular, for all $M\subset
\{1,2,3,4\}$ with two elements, we have $T^{MM}=0$.     
We conclude that no estimate of the form \eqref{eq:30} or
\eqref{eq:10} holds.
\end{ex}

In Example \ref{measures and weakly positive Lagerberg currents1}, all 
co-coefficients were positive or negative Radon measures on $\R^4$.

\begin{ex}\label{measures and weakly positive Lagerberg currents2}
We construct from $\omega$ in Example \ref{too few strongly positive}
a weakly positive Lagerberg current ${T}$ whose 
co-coefficients are not all Radon measures.

Since $\omega \wedge \alpha = 0$ for all strongly positive 
forms $\alpha$, 
we have that ${\omega \wedge {T'}}$ is weakly positive for 
every {symmetric} current ${T'}$ of type $(0,0)$. 
This already shows that it is very unlikely that every weakly 
positive current has measure coefficients.
Indeed, one may check that for the Lagerberg currents
\begin{align*}
{T'}\colon A^{4,4}_c(\R^4)\longrightarrow \R,\,\,
{T'}(fd'u_1\wedge d''u_1\wedge\ldots\wedge d'u_4\wedge d''u_4)
=\frac{\partial f}{\partial u_1}(0,0,0,0)
\end{align*}
and $T=\omega \wedge {T'}$, and the sets 
$I = \{3, 4\}$ and $J = \{1,2\}$, we have
\begin{displaymath}
  T^{IJ}(f) = \frac{\partial f}{\partial u_1}(0,0,0,0).
\end{displaymath}
It is well-known that uniform convergence does not imply convergence
of derivatives.
Therefore the co-coefficient $T^{IJ}$ is 
\emph{not} a Radon measure. 
  \end{ex}

\section{Tropicalization of positive currents} 
\label{decomposition-invariant-boundary}

We consider a smooth fan $\Sigma$ in $N_\R$ with associated complex toric variety $X_\Sigma$ and
corresponding tropicalization $N_\Sigma$. 
In this section, we will describe precisely which  positive Lagerberg currents are in the image 
of positive complex currents with respect to $\trop_*$. 
To this end, in the first subsection, we will introduce a
canonical decomposition of positive currents in
the complex and in the Lagerberg case and use
it to give the desired characterization in the second subsection.

\subsection{Decomposition of positive currents along the
  strata} \label{subsection positive currents and the boundary}

In this subsection, we give a canonical decomposition 
of a positive current along the boundary strata. 
As usual, we handle the complex and the tropical case simultaneously. 
For simplicity, we will often give the arguments only in the case of
Lagerberg currents
as the complex case is completely similar and 
even easier as the exceptional sets $E^I$ for co-coefficients do not occur 
(see Subsection \ref{Co-coefficients subsection}). 
At the end of the subsection, we will study functoriality of the
canonical decomposition with respect to $\trop_*$.

\begin{rem} \label{stratifications of sets}
Recall that the torus action of $\torus$ yields stratifications of the
toric variety
\begin{align*}
X_\Sigma = \coprod_{\sigma \in \Sigma} O(\sigma) \; \text{ and } \;
X_\Sigma^\an = \coprod_{\sigma \in \Sigma} O(\sigma)^\an
\end{align*} 
into locally closed subsets given by the orbits $O(\sigma)$. 
Similarly, we have the stratification 
\begin{align*}
N_\Sigma = \coprod_{\sigma \in \Sigma} N(\sigma)
\end{align*}
given in Definition \ref{definition partial compactification}. 	
Note that the stratification of $N_\Sigma$ induces the one of $X_\Sigma^\an$ by 
$O(\sigma)^{\rm an}=\trop^{-1}(N(\sigma))$ for any $\sigma \in \Sigma$. 
For an open subset $V$ of $\Xsigmaan$ and an open subset $U$ of $N_\Sigma$, we get induced stratifications 
\begin{align*}
V= \coprod_{\sigma \in \Sigma} V \cap O(\sigma)^{\rm an}  \; \text{ and } \;  
U=\coprod_{\sigma \in \Sigma} U \cap N(\sigma).
\end{align*}
For any $\rho \in \Sigma$, the open subset 
$\coprod_{\nu \prec \rho} O(\nu)$  of $X_\Sigma$ is the affine 
toric variety associated with $\rho$ and 
its tropicalization $\coprod_{\nu \prec \rho} N(\nu)$ is  
open in $N_\Sigma$.
This gives rise to the open subsets 
\begin{align*}
V_\rho \coloneqq   \coprod_{\nu \prec \rho} V \cap O(\nu)^{\rm an} \; \text{ and } \; 
U_\rho \coloneqq \coprod_{\nu \prec \rho} U \cap N(\nu)
\end{align*}
of $V$ and $U$, respectively. 
For varying $\rho \in \Sigma$, 
these open  subsets  cover $V$ (resp.~$U$). 
\end{rem}

\begin{defi}\label{good-coordinates}
Since $\Sigma$ is smooth, a given cone $\rho$ in $\Sigma$ is 
generated by part of a $\Z$-basis $b_1, \dots, b_n$ of $N$. 
Using the corresponding coordinates, we may view  
$V_\rho$ (resp.~$U_\rho$) as an open subset of 
$\C^n$ (resp.~$\Rsup^n$). 
We call such an identification of 
$V_\rho$ (resp.~$U_\rho$)  with an open subset of $\C^n$ (resp.~$\R_\infty^n$) a
\emph{choice of toric coordinates on $V_\rho$ (resp.~$U_\rho$)}. We
usually denote the corresponding
complex coordinates by $z_1, \dots, z_n$ and the corresponding
tropical coordinates by $u_1, \dots, u_n$.
\end{defi}

Note that for a current $T$ of bidegree $(p,q)$ on $V$ (resp.~$U$), 
the choice of toric coordinates on $V_\rho$ (resp.~$U_\rho$) 
leads to well-defined co-coefficients $(T|_{V_{\rho} })^{IJ}$
(resp. $(T|_{U_{\rho} })^{IJ}$) for all  $I,J \subset
\{1, \dots, n\}$ with $|I|=|J|=q=n-p$. Recall that, 
{if $T$ is positive, then} in the complex
case the
co-coefficient $(T|_{V_{\rho} })^{IJ}$ is a Radon measure on the open
set $V_{\rho }$ while in the
Lagerberg case $(T|_{U_{\rho} })^{IJ}$ is a Radon measure on the open
set $U_\rho\setminus
E^{I\cup J}$ (see \S \ref{Co-coefficients subsection}).
This allows us to define the notion of a null set of a positive
current.

 \begin{defi}\label{def:nullsetcurrent}
{Let $U$ be an open subset of $X_{\Sigma }^{\an}$ (resp. $N_{\Sigma }$). 
	We say that $T\in D^{p,p}(U)$ has \emph{measure co-coefficients} if 
	for every $\rho \in \Sigma$ and some choice of toric coordinates on $U_\rho$, 
	the co-coefficients $T^{IJ}$ are complex (resp.~real) Radon measures on $U_\rho$ (resp.~on $U_\rho \setminus E^{I \cup J}$).} 
	
	{Let $T \in D^{p,p}(U)$ have measure co-coefficients and let $A$ be  a Borel subset of $U$. We
	say that $A$ is a \emph{null set} for $T$ if, for all $\rho \in \Sigma$ and some choice of toric coordinates on $U_\rho$, the set $A\cap U_{\rho }$
	(resp. $A\cap U_{\rho }\setminus E^{I\cup J}$) is a null set of
	the total variation measure 
	$\left|(T|_{V_\rho })^{IJ}\right|$
	(resp.~$\left|(T|_{U_\rho })^{IJ}\right|$) for all $I,J \subset \{1, \dots, n\}$ with $|I|=n-p$.}     
\end{defi}

Using change of variables, it is easy to see that the conditions in Definition \ref{def:nullsetcurrent} do not depend on the choice of toric coordinates.

\begin{lem} \label{uniqueness sequence}	
{Let $U$ be open in  $X_{\Sigma }^{\an}$ (resp.~in $N_{\Sigma }$)
	 and let  $A \subset U$  be closed in $U$. Then there is an increasing sequence 
	 of  smooth functions $\psi_k \geq 0$ in $A_c^{0,0}(U \setminus A)$ which converges pointwise 
	 to the characteristic function of $U \setminus A$.} 
	 
	 {For  any $\alpha \in A_c^{p,q}(U)$, the forms $\alpha_k \coloneqq \psi_k \alpha \in A_c^{p,q}(U)$ have compact support in $U \setminus A$. If $\alpha$ is a positive form, 
	 then   $\alpha_k$ and $\alpha- \alpha_k$ are both positive.}
	 
	{For every $T \in D^{n-p,n-q}(U)$ with measure co-coefficients, we have:
	\begin{enumerate}[label={(\rm \emph{\alph*})}]
		\item\label{item:10}
		if $A$ is a null set for $T$, then  
		$\lim_{k \to \infty} T(\alpha_k)=T(\alpha)$;
	    \item\label{item:11}
		if $U \setminus A $  is a null set for  
		$T$, 
		then $ \lim_{k \to \infty} T(\alpha_k)=0$.
	\end{enumerate}}
\end{lem}

\begin{proof}
  Let $\alpha \in A_c^{p,q}(U)$ with support contained in the compact
  subset $K$ of $U$.  Since $U$ is a metric space and since $A \cap K$
  is compact, there is a strictly decreasing sequence
  $(W_k)_{k \in \N}$ of relatively compact open subsets of $U$ with
  \begin{align*}
    \bigcap_{k \in \N} W_k = A \cap K \quad \text{ and } \quad 
    W_0 \supset \overline{W_1} \supset W_1 \supset \overline{W_2} \supset \cdots.
  \end{align*}
  Now a partition of unity {(see Proposition \ref{extended-partition-unity} for the Lagerberg case)} 
  gives the existence of a function
  $\varphi_k \in A_c^{0,0}(W_k)$  with $0 \leq \varphi_k \leq 1$ and
  $\varphi_k \equiv 1$ on $W_{k+1}$ for every $k \in \N$.  We note
  that the $\varphi_k$ form a decreasing sequence of smooth functions
  with compact support which converges pointwise to the characteristic
  function of $A \cap K$. {For the sequence $\varphi_k \coloneqq 1 - \psi_k$, we get the first claim.} 
  Evaluating $T(\alpha_k)$ in terms of the co-coefficients $T^{IJ}$
  and using that the latter are  Radon measures, we deduce \ref{item:10} and
  \ref{item:11} from the monotone convergence theorem.  If $\alpha$ is positive,
  then every {$\alpha_k=\psi_k \alpha$} 
  and every
  $\alpha-\alpha_k=\varphi_k\alpha$ is positive.
\end{proof}

\begin{rem} \label{invariance in complex case of approximation lemma}
In the complex case of Lemma \ref{uniqueness sequence}, if $A$ is
$\SS$-invariant, then we can choose the smooth functions $\psi_k$ to
be $\SS$-invariant.   Indeed, this follows easily from Lemma
\ref{uniqueness sequence} by averaging the $\psi_k$ with respect to
the Haar probability measure on $\SS$. 
\end{rem}

  The \emph{decomposition theorem} is the following result.

  \begin{thm}  \label{stratification of currents}
Let $U$ be an open subset of $\Xsigmaan$ (resp.~of $N_\Sigma$) 
and let $T$ be a positive current in $D^{p,p}(U)$. 
Then there is a unique  decomposition 
\begin{equation}\label{global-stratification-of-currents}
T = \sum_{\sigma \in \Sigma}T_\sigma
\end{equation}
such that, for every $\sigma \in \Sigma $, the following two 
conditions are satisfied
\begin{enumerate}
\item \label{item:8} $T_{\sigma }\in D^{p,p}(U)$ is positive;
\item \label{item:9} the set $U\setminus O(\sigma )$
  (resp. $U\setminus N(\sigma )$)
  is a null set for $T_{\sigma }$.
\end{enumerate}
Moreover, the support of the current $T_\sigma$ is contained in 
$V \cap \overline{O(\sigma)^{\an}}$ (resp.~$U \cap
\overline{N(\sigma)}$).
\end{thm}

We call \eqref{global-stratification-of-currents} the \emph{canonical
  decomposition} of the positive current $T$.

\begin{proof}
We write the proof only in the
  Lagerberg case as the complex case is analogous. The statement
  about the support is a direct consequence of statement \ref{item:9}
  so we only need to prove the existence and uniqueness of the
  decomposition.

Using a partition of unity provided by Proposition
\ref{extended-partition-unity}, the existence of such decomposition can
be checked locally. Moreover, since currents form a sheaf, the unicity
can also be checked locally. Therefore we choose a cone $\rho \in
\Sigma $ and replace $U$ by $U_{\rho }$. We also make a choice of
toric coordinates, so we can assume from now on that $U\subset
\R_\infty^n$.

For each pair of subsets $I,J$ with $|I|=|J|=p$ we write
\begin{displaymath}
  T^{IJ}=T^{IJ}_{+}-T^{IJ}_{-}
\end{displaymath}
with $T^{IJ}_{\pm}$ the positive and the negative part of the Radon
measure $T^{IJ}$. We denote by $\mu 
_{\pm}^{IJ}$ the corresponding Borel measures.  
For any  $\sigma \in \Sigma$, let us consider the immersion  
\begin{align*}
i_{\sigma}^{IJ}\colon U \cap N(\sigma) \setminus 
E^{I\cup J} \longrightarrow U \setminus E^{I\cup J}.
\end{align*}
We  define a new Borel measure 
$\mu_{\pm,\sigma}^{IJ}$ on $U \setminus E^{I \cup J}$ by first
restricting $\mu_\pm^{IJ}$ to 
the locally closed subset $U \cap N(\sigma) \setminus E^{I\cup J}$ and 
then using the image measure with respect to $i_\sigma^{IJ}$. 
Since the sets $N(\sigma)$ form a stratification of $N_\Sigma$, we get
the decomposition
\begin{equation} \label{decomposition of Borel measures}
\mu_\pm^{IJ}= \sum_{\sigma\in\Sigma} \mu_{\pm,\sigma}^{IJ}
\end{equation}
of Borel measures. 
It follows that $\mu_{\pm,\sigma}^{IJ} \leq \mu_\pm^{IJ}$. 
Note that in our setting, the Radon measures correspond to locally
finite Borel measures 
(see Appendix \ref{facts from measure theory}). 
We conclude that  $\mu_{\pm,\sigma}^{IJ}$ is a locally finite Borel measure and 
hence we get a real Radon measure $T_\sigma^{IJ}$ on $U \setminus E^{I\cup J}$ corresponding to 
$\mu_{+,\sigma}^{IJ}- \mu_{-,\sigma}^{IJ}$. By Remark \ref{co-coefficients determine
	currents}, we obtain a unique current $T_{\sigma } \in D^{p,p}(U)$ with co-coefficients $T_\sigma^{IJ}$.
{Note that $T_\sigma$ has measure co-coefficients.} 
By construction, the set $U \setminus (E^{I\cup J} \cup N(\sigma))$ 
is a null set 
with respect to the Borel measure  $\mu_{+,\sigma}^{IJ}+\mu_{-,\sigma}^{IJ}$ 
associated to the total variation measure $|T_\sigma|$.  
The decomposition $T = \sum_\sigma T_\sigma$ follows from
\eqref{decomposition of Borel measures}. It remains to show uniqueness
and property \ref{item:8}.

We prove uniqueness of $T_\sigma$ in the decomposition \eqref{global-stratification-of-currents} 
by induction with respect to the partial ordering $\prec$ on $\Sigma$. 
We take $\sigma \in \Sigma$ and we suppose that uniqueness of $T_\tau $
is known for all $\tau \not = \sigma $ with $\tau \prec \sigma$.
Recall from Remark \ref{stratifications of sets} 
that $U_\sigma$ is an open subset of $U$ and 
that $A_\sigma \coloneqq U \setminus U_\sigma$ is the closed subset of $U$ 
given by the disjoint union of all $N(\tau) \cap U$ with $\tau \in
\Sigma$ not a face of $\sigma$.
For $\alpha \in A_c^{n-p,n-p}(U)$,  
Lemma \ref{uniqueness sequence} gives a sequence $(\alpha_k)_{k \in
  \N}$ in $A_c^{n-p,n-p}(U)$ with 
\begin{align*}
\lim_{k \to \infty} T_\tau(\alpha_k)= T_\tau(\alpha) 
\end{align*}
for $\tau \prec \sigma$ and
\begin{align*}
\lim_{k \to \infty} T_\tau(\alpha_k)= 0
\end{align*}
for all $\tau \in \Sigma$ which are not faces of $\sigma$. 
Now the decomposition \eqref{stratification-of-currents} gives 
\begin{align*}
\lim_{k \to \infty} T(\alpha_k) = \sum_{\tau \prec \sigma} T_\tau(\alpha)
\end{align*}
and uniqueness of $T_\sigma$ follows from our induction hypotheses. 

For the positivity \ref{item:8}, we prove first that the current 
\begin{align*}
T_\sigma' \coloneqq T- \sum_{\tau \prec \sigma, \, \tau \neq \sigma}
  T_\tau 
\end{align*}
is positive. For this let now $A_\sigma'=A_{\sigma }\cup (N(\sigma )
\cap U)$ be the 
union of all $N(\tau) \cap U$ with $\tau$ ranging over all cones of
$\Sigma$ which are not proper faces of $\sigma$. 
Clearly, $A_\sigma'$ is also a closed subset of $U$. 
Let $\alpha' \in A_c^{n-p,n-p}(U)$ be positive. 
Then Lemma \ref{uniqueness sequence} gives a sequence 
$(\alpha_k')_{k \in \N}$ in $A_c^{n-p,n-p}(U)$ with both $\alpha_k'$
and $\alpha'-\alpha_k'$ positive.
Using the properties \ref{item:10} and \ref{item:11} of the sequence from Lemma
\ref{uniqueness sequence}, 
the same argument as before shows that
\begin{align*}
\lim_{k \to \infty} T(\alpha_k') = \sum_{\tau \prec \sigma, \, \tau
  \neq \sigma} T_\tau(\alpha')
\end{align*}
and hence 
\begin{align*} 
T_\sigma'(\alpha')= \lim_{k \to \infty} T(\alpha'-\alpha_k').
\end{align*}
Since $T$ and all $\alpha'-\alpha_k'$ are positive, we deduce that
$T_\sigma'$ is positive.

To prove that $T_\sigma$ is positive, we consider again the closed subset 
$A_\sigma $. 
Let $\alpha \in A_c^{n-p,n-p}(U)$ be positive. 
By Lemma \ref{uniqueness sequence}, there is a positive sequence  
$(\alpha_k)_{k \in \N}$ in  $A_c^{n-p,n-p}(U)$ such that 
\begin{align*}
\lim_{k \to \infty} T_\sigma(\alpha_k)= T_\sigma(\alpha)
\quad \quad 
\text{and}
\quad \quad
\lim_{k \to \infty} T_\tau(\alpha_k)= 0
\end{align*}
for all $\tau \in \Sigma$ which are not faces of $\sigma$. 
We conclude that 
\begin{align*}
\lim_{k \to \infty} T_\sigma'(\alpha_k)= T_\sigma(\alpha).
\end{align*} 
Using that $T_\sigma'(\alpha_k)\geq 0$ by the positivity of
$T_\sigma'$ and $\alpha_k$, 
we deduce that $T_\sigma$ is positive.
\end{proof}

In the complex case, we will show next that the canonical current 
$T_\sigma$ from the canonical 
decomposition {\eqref{global-stratification-of-currents}}
is not always the push-forward of a current on  $V \cap \overline{O(\sigma)}^{\rm an}$.

\begin{ex} \label{counterexample for push-forward}
Let us consider $V=\Xsigmaan = \C^2$ 
where $\Sigma$ is the  fan whose maximal cone is the
positive quadrant in $\R^2$
and let $T \in D^{1,1}(\C^2)$ be given by 
\[
T\biggl(\sum_{k,l} \alpha_{kl} idz_k \wedge d \overline z_l\biggr) 
= \alpha_{1,1}(0,1).
\]
Using that the coefficients $\alpha_{kk}$ of a positive $(1,1)$-form are 
non-negative, we see that $T$ is a positive current. 
We conclude that $T = T_\sigma$ for the cone $\sigma$ generated 
by $(1,0)$. 
Note that the push-forward $S$ of a current on the stratum closure 
$\overline{O(\sigma)}^{\rm an}=\{z_1=0\}$ can only have a non-zero 
co-coefficient {$S^{2,2}$}.
Since all the co-coefficients of $T$ are zero except {$T^{1,1}$} which 
is given by the Dirac measure $\delta_{(1,0)}$ in the point $(1,0)$, 
we conclude that $T_\sigma$ is not the push-forward of a current 
on  $\overline{O(\sigma)}^{\rm an}$.
\end{ex}

Note that the analogue of Example \ref{counterexample for push-forward} in the Lagerberg 
case gives no 
counterexample as $T$ is zero in  this case since 
the point $(\infty,0)$ belongs to $E^{\{1\}}$. 
This is explained by the following result.

\begin{prop} \label{pushforward decomposition}
Let $U$ be an open subset of  $\Xsigmaan$ (resp.~$N_\Sigma$). 
Let $T$ be a positive current of bidegree $(p,p)$ on $U$ with canonical 
decomposition  $T=\sum_{\sigma \in \Sigma} T_\sigma$. 
Let us fix $\sigma \in \Sigma$.
\begin{enumerate}[label={(\rm \emph{\alph*})}]
\item\label{item:1} If $U \subset \Xsigmaan$, then the complex current
  $T_\sigma$ is the push-forward of a positive current on
  $U \cap \overline{O(\sigma)}^{\rm an}$ if and only if for all
  $\rho \in \Sigma$ and all choices of toric coordinates on $U_\rho$
  the co-coefficients $T_\sigma^{IJ}$ are zero whenever
  $O(\sigma) \subset E^{I \cup J}$.
\item\label{item:2} If $U \subset N_\Sigma$, then there is a positive
  Lagerberg current
on $U \cap \overline{N(\sigma)}$ whose image on $U$ is 
$T_\sigma$.
\end{enumerate}
\end{prop}

We will see in the proof below that in the Lagerberg case, 
the necessary and sufficient condition in \ref{item:1} is always
satisfied which is the reason for \ref{item:2} to be true.

\begin{proof}
 We first deal with the complex case \ref{item:1}.  
Suppose that there is a current $S$ on $U$ such that 
$\iota_*(S)=T_{\sigma }$ for the closed immersion $\iota$ of 
$U \cap \overline{O(\sigma)}^{\rm an}$ 
into $U$. 
Let $\rho \in \Sigma$ and let us fix toric coordinates on $U_\rho$
such that we may view $U_\rho$ 
as an open subset of $\C^n$. 
Suppose that $N(\sigma) \subset E^{I \cup J}$. 
This means that there is $k \in I \cup J$ such that the coordinate 
$z_k$ vanishes on $O(\sigma)$. 
Hence the restrictions of $dz_k,d\overline z_k$ to $O(\sigma)$ 
are zero. 
This implies that $T_{\sigma }^{IJ}=0$ using the definition of co-coefficients 
and $\iota_*(S)=T_{\sigma }$.

Conversely, assume that for any $\rho \in \Sigma$ and any choice of
toric coordinates on $U_\rho$, 
we have  $T_\sigma^{IJ}=0$ on $U_\rho$ whenever 
$N(\sigma) \subset E^{I \cup J}$. 
The  coordinates $z_k$ of $O(\sigma)$ are 
precisely those such that $z_k \neq 0$ on $O(\sigma)$. 
We conclude that the co-coefficients of a current $S$  on 
$U_\rho \cap \overline{O(\sigma)}^{\rm an}$ 
can be labeled as $S^{IJ}$ 
with $I,J \subset \{1,\dots,n\}$ such that $|I|=|J|=p$ and such that 
$N(\sigma)$ is not contained in $E^{I \cup J}$. 
To construct such a current $S$ with $\iota_*(S)=T_{\sigma }$, we define its 
co-coefficients on $U_\rho \cap \overline{O(\sigma)}^{\rm an}$ 
by
\begin{align*}
S^{IJ}(f) \coloneqq T^{IJ}_{\sigma }(g)
\end{align*}
where $f$ is any smooth function with compact support in 
$U_\rho \cap \overline{O(\sigma)}^{\rm an}$ 
and where $g$ is any extension to a smooth function with compact support 
in $U_\rho$. 
It follows from Theorem \ref{stratification of currents} that the 
co-coefficients of $T_{\sigma }$ depend only on the restriction to 
$U_\rho \cap  \overline{O(\sigma)}^{\rm an}$.
Hence the definition of $S^{IJ}$ depends only on $f$ and not on the choice 
of the extension $g$. 
By Remark \ref{co-coefficients determine currents}, there is a unique 
current $S_\rho$ on $U_\rho \cap \overline{O(\sigma)}^{\rm an}$ with co-coefficients $S^{IJ}$. 
By definition, $T_{\sigma }^{IJ}$ is the push-forward of $S^{IJ}$. 
This shows that $T_{\sigma }|_{U_\rho}$ is the push-forward of $S_\rho$ with respect to the 
closed immersion of  $U_\rho \cap  \overline{O(\sigma)}^{\rm an}$ into $U_\rho$. 
It follows that the currents $S_\rho$ on the open covering $U_\rho$ of $U$ 
glue to a current $S$ on $U$  with 
$\iota_*(S)=T_{\sigma }$. 
Since $T_{\sigma }$ is positive and since we can extend compactly
supported positive
forms on $U \cap \overline{N(\sigma)}$ to compactly supported positive forms 
$U$, it is clear that also $S$ is positive. This proves \ref{item:1}.

In the case of a positive Lagerberg current on $U$, we note that 
Theorem \ref{stratification of currents} says that the 
support of $T_\sigma$ is contained in $U \cap \overline{N(\sigma)}$. 
For any $\rho \in \Sigma$ and any choice of toric coordinates on
$U_\rho$, the co-coefficient $T_{\sigma }^{IJ}$ on
$U_\rho \setminus E^{I\cup J}$  has support contained in $\supp(T_{\sigma })$.  
If $I,J$ satisfy  $N(\sigma) \subset E^{I \cup J}$, 
then the support of the restriction of $T_{\sigma }$ to $U_\rho$ is contained 
in $U_\rho \cap \overline{N(\sigma)} \subset E^{I\cup J}$ and 
hence $T_{\sigma }^{IJ}=0$ on $U_\rho \setminus E^{I\cup J}$. 
We conclude that the crucial condition in \ref{item:1} is always satisfied in
the Lagerberg case and an easy adaption of the argument of \ref{item:1} 
to the Lagerberg case gives a positive Lagerberg current $S$ on $U
\cap N(\sigma)$ 
whose image in $U$ is $T_\sigma$ proving \ref{item:2}.
\end{proof}

\begin{lem} \label{preparation-lemma}
{Let $U$ be an open subset of $\R_\infty^n$ and 
$V \coloneqq \trop^{-1}(U)\subset \C^n$. We consider 
$S \in D^{p,p}(V)^{\SS,F}$ and 
$T \coloneqq \trop_*(S)$. 
Let $I,J \subset \{1,\dots,n\}$ with $|I|=|J|=q\coloneqq n-p$.}
\begin{enumerate}
\item\label{preparation-lemma1}
{For any $f \in A_c^{0,0}(U \setminus E^{I\cup J})$, we have}
\begin{equation} \label{co-coefficients and trop}
{T^{IJ}(f)=\pi^{-q}2^{-2q} S^{IJ}
\bigl(z^{-I}\bar{z}^{-J}\trop^*(f)\bigr).}
\end{equation}
\item \label{preparation-lemma2}
{Assume that the current $S$ is positive.
Then $T$ is positive as well and the total variation measures $|S^{IJ}|$ and $|T^{IJ}|$
of the Radon measures $S^{IJ}$ and $T^{IJ}$ satisfy}
\begin{equation} \label{variance of co-coefficients and trop}
{|S^{IJ}|\bigl(\trop^*(g)\bigr)=
\pi^{-q} 2^{2q} |T^{IJ}| \Biggl(\prod_{i \in I} e^{-u_i} \prod_{j \in
  J} e^{-u_j} g \Biggr)}
\end{equation}
{for each $g \in A_c^{0,0}(U \setminus E^{I\cup J})$.} 	
\end{enumerate}
\end{lem}

\begin{proof} 
{Recall from Lemma \ref{trop of complex current} that $T \in D^{p,p}(U)$. To prove 
\ref{preparation-lemma1},  we use} \eqref{eq:25} to get 
\begin{multline*}
T^{IJ}(f)=T((-1)^{\frac{q(q-1)}{2}}f d'u_I \wedge d''u_J)
= S(i^{q^2-q}\trop^*(f d'u_I \wedge d''u_J)) \\
= S \left(i^{q^2}\pi^{-q}\trop^*(f) \frac{dz_{I}}{2^{q}z^{I}}\wedge
  \frac{i^{q}d\bar 
	z_{J}}{2^{q}\bar z^{J}} \right)
= \pi^{-q}2^{-2q} S^{IJ}(z^{-I}\bar{z}^{-J}\trop^*(f)).
\end{multline*}
To prove \ref{preparation-lemma2}, 
we recall from Proposition \ref{positivity and trop} that 
$T$ is positive if $S$ is positive.
Now equation \eqref{variance of co-coefficients and trop} follows 
from \eqref{co-coefficients and trop} by using
$
|z^{I}\bar z^{J}|=\prod_{i\in I}e^{-u_{i}}\prod_{j\in J}e^{-u_{j}}
$.
\end{proof}

The following result shows compatibility of the tropicalization map 
with the canonical decomposition of  positive currents from Theorem \ref{stratification of currents}.

\begin{prop} \label{extension by 0 and trop}
Let $U$ be an open subset of $N_\Sigma$ and  $V \coloneqq \trop^{-1}(U)$. If $S= \sum_{\sigma \in \Sigma} S_\sigma$ is the canonical
decomposition of the positive current $S \in D^{p,p}(V)^{{\SS,F}}$, {then $\trop_*(S)$ is a positive Lagerberg current on $U$ with canonical decomposition}
\begin{equation} \label{decomposition and tropicalization}
\trop_*(S)= \sum_{\sigma \in \Sigma} \trop_*(S_\sigma).
\end{equation} 
\end{prop}

\begin{proof}
{We first observe that with $S$ each summand $S_\sigma$ is  
again $\SS$- and $F$-invariant by the uniqueness property 
of the canonical decomposition.} 
To prove  \eqref{decomposition and tropicalization}, we may argue locally on the base and so we may assume 
that $U=U_\rho$ and $V=V_\rho$ for some $\rho \in \Sigma$. 
We choose toric coordinates to view $U$ as an open subset of {$\Rsup^n$} 
and $V$ as an open subset of $\C^n$. 
{We write $T\coloneqq \trop_*(S)$ and denote by
$T=\sum_{\sigma}T_\sigma$ its canonical decomposition.}
We pick $\sigma \in \Sigma$ and we may assume $S=S_\sigma$. 
Then we have  to show that $T=T_\sigma$. 
By our characterization of the canonical decomposition in Theorem
\ref{stratification of currents},  
it is equivalent to show that $U \setminus (N(\sigma) \cup E^{I \cup J})$ are null sets 
with respect to the Radon measures $|T^{IJ}|$. 
We note that  $\trop^{-1}(N(\sigma))= O(\sigma)$ and 
that $V \setminus O(\sigma)$ is a null set with respect to the Radon measure $S^{IJ}$ 
using $S=S_\sigma$ and Theorem \ref{stratification of currents}. 
Then \eqref{variance of co-coefficients and trop} yields 
that  $U \setminus (N(\sigma) \cup E^{I \cup J})$ is a null set  
with respect to  $|T^{IJ}|$ proving  \eqref{decomposition and tropicalization}.
\end{proof}

\subsection{Positive Lagerberg currents and local mass} \label{subsection local mass}
Equation 
\eqref{variance of co-coefficients and trop} 
gives a necessary condition 
for a positive Lagerberg current 
to be the image of {an $\SS$- and $F$-invariant} positive complex current. 
In fact, it naturally leads to the following definition 
which turns out to be  also a sufficient condition.

\begin{defi} \label{def complex local finite mass}
Let $U$ be an open subset of $N_{\Sigma }$ and 
let $T\in D^{p,p}(U)$ be a  
positive Lagerberg current. 
We say that $T$ has \emph{$\C$-finite local mass} if for 
all $\rho \in \Sigma$ 
and some choice of toric coordinates 
$u_1, \dots, u_n$ on $U_\rho$ as in Remark \ref{good-coordinates}, 
the corresponding co-coefficients $T^{IJ}$, which 
may be seen as real Radon measures on 
$U_\rho \setminus E^{I \cup J}$ with total variation measures
$|T^{IJ}|$, satisfy the condition that 
the Borel measures given as the image measures
\begin{displaymath}
j_\rho^{IJ}\Biggl( |T^{I,J}|\prod_{i\in I}e^{-u_{i}}\prod_{j\in J}e^{-u_{j}}\Biggr)
\end{displaymath}
with respect to the open  immersion 
$j_\rho^{IJ}\colon U_\rho \setminus E^{I \cup J} \to U_\rho$ 
are locally finite on $U_\rho$.
\end{defi}

Clearly, the above definition of $\C$-finite local mass 
does not depend on the choice of toric coordinates on $U_\rho$. 
Note that a Borel measure is locally finite on $U_\rho$ if and only if 
it is a Radon measure on $U_\rho$ (see Appendix \ref{facts from measure theory}).

\begin{prop}\label{prop:9}
Let $U\subset N_{\Sigma }$ be an open subset and 
$V=\trop^{-1}(U)$. Let $T\in D^{p,p}(U)$ be a positive Lagerberg 
current, then there exists a positive complex current 
$S\in D^{p,p}(V)^{\SS,F}$ such that $\trop_{\ast}(S)=T$ 
if and only if $T$
has $\C$-finite local mass.
\end{prop}
\begin{proof}
Suppose that $S\in D^{p,p}(V)^{\SS,F}$ is a positive current and let
$T=\trop_{\ast}(S)$. We will prove that $T$ has $\C$-finite local
mass. Indeed, it follows from \eqref{variance of co-coefficients and trop} that we have the identity 
\begin{equation*} 
\trop(|S^{IJ}|)= \pi^{-q}2^{2q} \prod_{i \in I} e^{-u_i} \prod_{j \in J} e^{-u_j} |T^{IJ}|
\end{equation*}
of positive measures on $U_\rho \setminus E^{I\cup J}$ for any
$\rho \in \Sigma$ and any choice of toric coordinates $u_1,\dots,u_n$ on $U_\rho$.  
Passing to image measures with respect to the
open immersion $j_\rho^{IJ}\colon  U_\rho \setminus E^{I\cup J} \to U_\rho$ and 
using that $\trop(|S^{IJ}|)$ is a Borel measure on $U_\rho$, we deduce that 
\begin{equation*} \label{variance and trop 2} 
\trop(|S^{IJ}|) \geq \pi^{-q} 2^{2q} j_\rho^{IJ} 
\Biggl(\prod_{i \in I} e^{-u_i} \prod_{j \in J} e^{-u_j} |T^{IJ}|\Biggr) \geq 0.
\end{equation*}
By Proposition \ref{prop:7},  the Borel measure $|S^{IJ}|$ is locally
finite on $V_\rho$ and
hence the Borel measure $\trop(|S^{IJ}|)$ is locally finite on $U_\rho$. 
We deduce  that $T$ has $\C$-finite local mass.

Conversely, let $T$ be a positive Lagerberg current in $D^{p,p}(U)$ with
$\C$-finite local mass. 
We consider the canonical decomposition $T= \sum_{\sigma \in \Sigma} T_\sigma$ 
from Theorem  \ref{stratification of currents}. 
First, we assume that $T=T_{\{0\}}$ for the minimal cone $\{0\} \in \Sigma$. 
By assumption, the co-coefficients are 
Radon measures $T^{IJ}$ on 
$U_\rho \setminus E^{I\cup J}$ such that the positive 
Borel measures 
\begin{displaymath}
j_\rho^{IJ}\biggl( |T^{IJ}|\prod_{i\in I}e^{-u_{i}}\prod_{j\in J}e^{-u_{j}}\biggr)
\end{displaymath}
are locally finite on $U_\rho$ for every $\rho \in \Sigma$. 
By Remark \ref{rem:3}, there is a unique current 
$R\in D^{p,p}(V \cap \torusan)^{\SS,F}$ 
with $\trop_*(R)=T|_{N_\R}$. 
Since $T$ is positive, it follows from 
Lemma \ref{generic positivity and trop}
that $R$ is a positive current. 
It follows from Proposition \ref{prop:7} that the co-coefficients $R^{IJ}$ 
of $R$ are complex Radon measures on $V \cap \torusan$. 
Using \eqref{variance of co-coefficients and trop}, we 
get the following identity of Borel measures
\begin{equation*} 
\trop(|R^{IJ}|)=\pi^{-q}2^{2q} \prod_{i \in I} e^{-u_i} \prod_{j \in J} e^{-u_j} |T^{IJ}|
\end{equation*}
on $U\cap N_\R$.
Using that $T$ has $\C$-finite local mass, we know that the 
image measure of the right hand side with respect to the 
open immersion $j_\rho\colon  U \cap N_\R \to U_\rho$ is a 
locally finite Borel measure on $U_\rho$. 
We conclude that the image measure of $|R^{IJ}|$ with 
respect to the open immersion 
$j_\rho'\colon V \cap \torusan \to V_\rho$ is a locally finite 
Borel measure on $V_\rho$ by using 
$j_\sigma \circ \trop = \trop \circ j_\sigma'$ on $V \cap \torusan$. 
Hence the Radon measure $R^{IJ}$ admits an image {Radon} measure  
 $j_\rho'(R^{IJ})$ which 
is a complex Radon measure  on $V_\rho$. 
By Remark \ref{co-coefficients determine currents} 
and Example \ref{measure gives current in top degree}, 
there is a unique current in $S \in D^{p,p}(V_\rho)$ with 
co-coefficients $S^{IJ}$ induced by $j_\rho'(R^{IJ})$. 
Since the restriction of the co-coefficients $S^{IJ}$ to the 
dense stratum $V \cap \torusan$ is $R^{IJ}$, 
it follows also from Remark  
\ref{co-coefficients determine currents} 
that $S|_{V \cap \torusan}= R$. 
By construction, the boundary $V_\rho \setminus N_\R$ is a null set
with respect to each Radon measure $j_\rho'(R^{IJ})$. 
By Lemma \ref{uniqueness sequence} {and Remark \ref{invariance in complex case of approximation lemma}}, there is an increasing sequence  
 $\psi_k \geq 0$  in $A_c^{0,0}(V \setminus N_\R)^\SS$ which converges pointwise 
to the characteristic function of $V \setminus N_\R$ such that for any $\alpha \in A_c^{q,q}(V_\rho)$ we have
\begin{equation} \label{limit from the dense stratum}
\lim_{k \to \infty} S(\alpha)= \lim_{k \to \infty} S(\psi_k\alpha) = \lim_{k \to \infty} R(\psi_k\alpha).
\end{equation}

It follows that the currents $S$, defined a priori only on $V_\rho$ for any 
$\rho \in \Sigma$, 
do not depend on the choice of toric coordinates on $V_\rho$ and hence 
agree on overlappings of the covering 
$(V_\rho)_{\rho \in \Sigma}$.  
They define a current on $V$ which we also denote by $S$. 
If $\alpha$ is positive, then 
$\psi_k\alpha$ is also positive. By positivity of $R$ and by \eqref{limit from the dense stratum}, the current $S$ is also positive. 
	Since both $\psi_k$ and $R$ are invariant with respect to $\SS$ and $F$, we deduce from \eqref{limit from the dense stratum} that $S \in D^{p,p}(V)^{\SS,F}$.

We have to check $\trop_*(S)=T$. 
Using $j_\sigma \circ \trop = \trop \circ j_\sigma'$ on $V \cap \torusan$, 
we have the following identities 
\begin{align*}
\trop(\pi^{-q} 2^{-2q} z^{-I}\bar{z}^{-J} S^{IJ})= 
\trop(j_\sigma'( \pi^{-q} 2^{-2q} z^{-I}\bar{z}^{-J} R^{IJ}))= 
j_\sigma(\trop(\pi^{-q} 2^{-2q} z^{-I}\bar{z}^{-J} R^{IJ}))
\end{align*}
of Radon measures on $U_\rho$. 
By \eqref{co-coefficients and trop}, 
we deduce the identity of Radon measures 
\begin{align*}
(\trop_*(S))^{IJ}= \trop(\pi^{-q}2^{-2q} z^{-I}\bar{z}^{-J} S^{IJ})= 
j_\sigma(\trop(\pi^{-q}2^{-2q} z^{-I}\bar{z}^{-J} R^{IJ}))
\end{align*}
on $U_\rho \setminus E^{I\cup J}$. Since $\trop_*(R)=T$ on $U \cap N_\R$, 
we deduce again from \eqref{co-coefficients and trop} 
\begin{align*}
(\trop_*(S))^{IJ}= j_\sigma(\trop(2^{-2q} z^{-I}\bar{z}^{-J} R^{IJ})) = T^{IJ}
\end{align*}
on $U_\rho \setminus E^{I\cup J}$, 
where in the last step we have used our assumption $T=T_{\{0\}}$. 
Then Remark \ref{co-coefficients determine currents} proves that  $\trop_*(S)=T$.

Now we skip the assumption $T=T_{\{0\}}$ from above and 
consider any positive current $T$ with canonical decomposition 
$T= \sum_{\sigma \in \Sigma} T_\sigma$ from Theorem \ref{stratification of currents}. 
By Proposition \ref{pushforward decomposition}, for every $\sigma \in \Sigma$, 
there is a positive Lagerberg current $P_\sigma$ on $U \cap \overline{N(\sigma)}$ 
such that $T_\sigma = (\iota_\sigma)_*(P_\sigma)$ for the closed immersion 
$\iota_\sigma\colon U \cap \overline{N(\sigma)} \to U$. 
Now we apply the above case to the current $P_\sigma$ and 
to the open subset $V \cap \overline{O(\sigma)}^{\rm an}$ of the toric variety $\overline{O(\sigma)}$. 
We conclude that there is a positive current $Q_\sigma$ on $V \cap \overline{O(\sigma)}^{\rm an}$ 
such that $\trop_*(Q_\sigma)=P_\sigma$. 
For the closed immersion $\iota_\sigma'\colon V \cap \overline{O(\sigma)}^{\rm an} \to V$, 
we have the obvious relation $\trop \circ \iota_\sigma' = \iota_\sigma \circ \trop$. 
Now we set $S_\sigma \coloneqq (\iota_\sigma')_*(Q_\sigma)$ which is a positive current on $V$.  
Then we get
\begin{align*}
\trop_*(S_\sigma)= \trop_*((\iota_\sigma')_*(Q_\sigma))
=(\iota_\sigma)_*(\trop_*(Q_\sigma))=(\iota_\sigma)_*(P_\sigma)= T_\sigma.
\end{align*}
Then $S \coloneqq \sum_{\sigma \in \Sigma} S_\sigma$ is a positive current on $V$ satisfying $\trop_*(S)=T$.
\end{proof}

Next we give an example of a positive Lagerberg current $T$ such that 
$T = \trop_{\ast}(S)$ for some complex current $S$, but for which no such $S$ is positive.

\begin{ex}\label{exm:3}
We consider the one-dimensional situation with $N=\Z$, 
$N_\Sigma = \Rsup$ and $X_\Sigma= \mathbb A_\C^1$. Write $U\subset
N_{\Sigma }$ given by   
$U=\{u \in \Rsup \mid u>0\}$ and $V=\trop^{-1}(U)=\{z \in \C \mid z\bar z <
1\}$.
Let $T\in D^{0,0}(U)$ be the current
\begin{displaymath}
T(gd'u\land d''u)=\int_{0}^{\infty} g(x) e^{2x}dx.
\end{displaymath}
Since the function $1=e^{2x}e^{-2x}$ does not have locally finite
mass around $\infty$, this current is not of the form
$\trop_{\ast}(S)$ for a positive current $S$ in
$V$. Nevertheless $T$ is in the image of $\trop_{\ast}$. Indeed,
let $S\in D^{0,0}(U)$ be the current given by
\begin{displaymath}
S(f dz \land d\bar z)= \int _{V}\frac{f(z)-f(0)}{z\bar
z}dz\land d\bar z.
\end{displaymath}
Clearly this current is not weakly positive.
Let now $gd'u\land d''u$ be a Lagerberg form on $U$. Then $g(u)=0$ for
$u\gg 0$ and hence
\begin{multline*}
\trop_{\ast} S (gd'u\land d''u) = S(\trop^{\ast}(gd'u\land d''u))
\,=\, S \left( g ( -\log|z| )\frac {idz\land d\bar z} {4 \pi z \bar z} \right)\\
=\frac{1}{\pi }\int _{V}\frac{g(-\log|z|)}{4 (z\bar z)^2}idz\land d\bar z
\,=\, \int_{0}^{1}g(-\log r)\frac{r dr}{r^{4}}
=
\int_{1}^{\infty}g(x)e^{2x}dx.
\end{multline*}
This shows $T=\trop_{\ast}S$. 
\end{ex}

\section{The correspondence theorem for closed positive currents} 
\label{correspondance-invariant-boundary}

In this section, we consider a smooth fan $\Sigma$ with associated toric 
variety $X_\Sigma$ and partial compactification $N_\Sigma$. 
After considering positivity of currents in the previous sections, 
we add here the additional condition that the currents are closed. 
Recall that a complex $(p,p)$-current $T$ (resp.~a Lagerberg $(p,p)$-current
$S$) is called
\emph{closed} if $\partial T = \bar{\partial}T = 0$ (resp.~$d'S=d''S=0$).
In the first subsection, we show our main theorem.
It states that the 
tropicalization map induces a bijective correspondence between $\SS$- and 
$F$-invariant closed positive currents on $X_\Sigma$ and closed 
positive Lagerberg currents on $N_\Sigma$. 
In the second subsection, we
derive from our main theorem a tropical version 
of the Skoda--El Mir theorem.

\subsection{Closed positive invariant currents} \label{closed positive
	invariant currents}

For $M \subset M' \subset \{1,\dots,n\}$, we consider the inclusions $\Rsup^M
\subset \Rsup^{M'} \subset \Rsup^n$ obtained 
by adding zero at the missing components.  
For $v \in \R_{\geq 0}^M$, we consider the parallelepiped  
\begin{align*}
P(v) \coloneqq \{u \in {\R^M} \mid 0 \leq u_i \leq v_i \, \forall i \in M\}.
\end{align*}

\begin{lem} \label{disturbation lemma}
	Let $T$ be a closed current in $D_+^{p,p}(U)$ 
	for an open subset $U$ of $\Rsup^n$.
	Fix $M\subset \{1,\ldots,n\}$ with $|M|=n-p$ and 
	with complement $M^c$. Fix also 
	a compact subset $K$ of  $U$, and a function $\chi \in
	C_c^\infty(\Rsup^{M^c})$.  
	Then there is a constant $C \in \R_{\geq 0}$ such that for all 
	$f \in C_c^\infty({\R}^M)$, and for all 
	$v \in \R_{\geq 0}^M$ such that $(\supp(f)+P(v))\times \supp(\chi) \subset K$, 
	we have
	\[
	\Bigl|T\bigl(f(u_M)\chi(u_{M^c})d'u_M \wedge d''u_M\bigr)
	-T\bigl(f(u_M-v)\chi(u_{M^c})d'u_M \wedge d''u_M\bigr)\Bigr| 
	\leq C \cdot \|v\| \cdot \|f \|_{\rm sup}.
	\]
\end{lem}

\begin{proof}
	Let $(e_1,\ldots,e_n)$ denote the standard basis of $\R^n$
	and write $v=\sum_{i\in M}t_ie_i$. 
	Using a telescope argument, one easily reduces to the case where
	$v=t e_i$ for some fixed $i\in M$.
	For a given {$t \in [0,v_i]$}, we consider the auxiliary function
	\begin{displaymath}
	g_t\colon \R^M\longrightarrow \R,\,\,
	g_t(u_M)=\int_{-\infty}^{0}f(u_M+se_i)
	-f(u_M+(s-t)e_i)ds
	=\int_{-t}^{0}f(u_M+se_i)ds.
	\end{displaymath}
	From $(\supp(f)+P(te_i))\times \supp(\chi) \subset K$, we get 
	\begin{equation}\label{reference in proof}
	\supp(g_t)\times\supp(\chi)\subset K,\,\,\,\,
	\|g_t\|_{\sup}\le t \|f\|_{\sup},\,\,\,\,
	\frac{\partial g_t}{\partial u_i}(u)
	=f(u_M)-f(u_M-t e_i).
	\end{equation}
	We put $M'=M\setminus \{i\}$ and obtain
	\begin{align*}
	d'\bigl(g_t(u_M)\chi(u_{M^c})d'u_{M'}\land d''u_{M})
	&=\sum_{j\in M^c}g_t(u_M)\frac{\partial\chi}{\partial u_{j}}(u_{M^c})
	d'u_{j}\land  d'u_{M'}\land d'' u_{M}\\
	&+\bigl(f(u_M)-f(u_M-t e_i)\bigr)\chi(u_{M^c})
	d'u_{i}\land d'u_{M'}\land d'' u_{M}.
	\end{align*}
	Since {$g_{t}$ has compact support in $\R^{M}$} and, by
	\eqref{reference in proof} $g_t(u_{M})\chi(u_{M^c})$ has compact
	support in $U$, we deduce that
	\begin{displaymath}
	g_t (u_{M})\chi(u_{M^c})d'u_{M'}\land d''u_{M}
	\in A^{n-p-1,n-p}_{c}(U).
	\end{displaymath}
	Hence, since $T$ is closed, we have
	\begin{displaymath}
	T\Bigl(d'\bigl(g_t (u_{M})\chi(u_{M^c})d'u_{M'}\land d''u_{M}\bigr)\Bigr)=0.
	\end{displaymath}
	Therefore, using that $T$ has measure co-coefficients (see Proposition
	\ref{prop:6}), we get
	\begin{align*}
	\bigl|T\bigl((f(u_M)-f(u_M-t e_i))&\chi(u_{M^c})d'u_{M}\land d''u_{M}\bigr)\bigr| 
	\le \sum_{j\in M^c}\bigl| T^{M'\cup\{j\},M}\bigr|
	\left(\left|\frac{\partial\chi}{\partial u_{j}}(u_{M^c})g_t(u_M)\right|\right)\\
	&\le  \sum_{j\in M^c}\left | T^{M'\cup\{j\},M}\right|
	\left(\left|\frac{\partial\chi}{\partial u_{j}}(u_{M^c})\right|\right) \|g_t\|_{\sup}
	\le C\cdot t \cdot \|f\|_{\sup}
	\end{align*}
	for some constant $C$.
\end{proof}

\begin{prop}\label{prop:8}
	Let $U\subset N_{\Sigma }$ be an open subset and $T \in D^{p,p}(U)$. 
	If $T$ is positive and closed, then it has $\C$-finite local mass. 
\end{prop}

\begin{proof}
	{The question is local and so we may assume that $U$ is an open subset of $\Rsup^n$.}
	
	If $p=n$, then 
	the only co-coefficient of $T$ is $T^{\emptyset,\emptyset}$ 
	which is a positive {Radon} measure on all of $U$ {by Proposition \ref{prop:6}}. 
	Hence $T$ has $\C$-finite local mass.
	
	Now assume that $p<n$. 
	We want to prove {that $T$ has $\C$-finite local} 
	mass around $u \in U$. 
	{We will check the relevant condition in Definition \eqref{def complex local finite mass} 
		first for the co-coefficients $T^{MM}$.}
	For notational simplicity, we assume $M=\{1,\dots,n-p\}$ and 
	$u=(\infty, \dots, \infty)$ which is no real restriction of generality as 
	finite components are easier to handle.
	We choose a compact neighborhood $K$ of $u$.   
	There is $R \in \N_{\geq 1}$ such that $[R-1,\infty]^n\subset K$. 
	There is a $C^\infty$-function $f$ on $\R^M$ with the properties:
	\begin{enumerate}
		\item 
		$0 \leq f \leq 1$;
		\item 
		$f \equiv 1$ on the unit cube $[R,R+1]^M$;
		\item 
		$f$ has compact support in the open cube $(R-\frac{1}{2},R+ \frac{3}{2})^M$.
	\end{enumerate}
	We choose $\chi \in C^\infty(\Rsup^{M^c})$ with compact support 
	in $(R-\frac{1}{2},\infty]^{M^c}$ such that 
	$0\leq \chi\leq 1$ and $\chi\equiv 1$ on
	$[R,\infty]^{M^c}$. 
	For $\underline n \in \N^M$, we consider the positive Lagerberg form 
	\begin{displaymath}
	\alpha  _{M,\underline n}
	={(-1)^{\frac{(n-p)(n-p-1)}{2}}}\chi(u_{M^c})f(u_M-\underline n)
	{d'u_{M}\wedge d''u_{M}.}
	\end{displaymath}
	Lemma \ref{disturbation lemma} implies that
	\begin{displaymath}
	T(\alpha  _{M,\underline n})= T(\alpha  _{M,\underline 0}) + O(\|\underline n\|).
	\end{displaymath}
	For any $\varepsilon >0$, we prove now that the positive measure
	\begin{displaymath}
	T^{MM}\prod_{i\in M}\frac{1}{u_{i}^{2+ 2\varepsilon}}
	\end{displaymath}
	has finite mass in the region $S\coloneqq [R,\infty)^M \times [R,\infty]^{M^c}$.
	We have 
	\begin{eqnarray*}
		\left(T^{MM}\prod_{i\in M}\frac{1}{u_{i}^{2+2\varepsilon}}\right)(1_{S})&\le&
		\sum _{\underline n\in (\N_{\geq 1})^{M}}
		\left(\prod_{i\in M}\frac{1}{n_{i}^{2+ 2\varepsilon}}\right)
		T(\alpha  _{M,\underline n})\\ 
		&=& \sum_{\underline n\in {(\N_{\geq 1})^{M}}}
		{\left(\prod_{i\in M}\frac{1}{n_{i}^{2+ 2\varepsilon}}\right)O(\|\underline n\|)}\\
		& \leq& O(1)  \sum_{\underline n\in {(\N_{\geq 1})^{M}}}\prod_{i\in M}
		\frac{1}{n_{i}^{1+2\varepsilon}} < \infty.
	\end{eqnarray*}
	For 
	$I,J\subset \{1,\dots,n\}$ with $|I|=|J|=n-p$, equation  \eqref{eq:30} yields 
	\begin{displaymath}
	|T^{IJ}|\prod_{i\in I}\frac{1}{u_{i}^{1+\varepsilon}}\prod_{j\in
		J}\frac{1}{u_{j}^{1+\varepsilon}}\le \frac{1}{2}\left(T^{II}\prod_{i\in
		I}\frac{1}{u_{i}^{2+2\varepsilon}}
	+T^{JJ}\prod_{j\in J}\frac{1}{u_{j}^{2+2\varepsilon}}\right)
	\end{displaymath}
	in the region $[R,\infty)^{I \cup J} \times [R,\infty]^{(I \cup J)^c}$. 
	Therefore the left hand side has finite
	mass in this region. Finally, using that $e^{-u}$ decreases faster
	than $u^{-1-\varepsilon}$, 
	we deduce that $T$ has $\C$-finite local mass.
\end{proof}

\begin{prop}\label{prop:10}
	Let $U\subset N_{\Sigma}$ be an open subset and $V=\trop^{-1}(U)$.
	Let $T',T''\in {D^{p,p}(V)^{\SS,F}}$ be positive and closed currents.
	If $\trop_{\ast}(T')=\trop_{\ast}(T'')$, then $T'=T''$.
\end{prop}

\begin{proof}
	This can be proved locally on $N_\Sigma$.
	Hence we may assume that $\Sigma $ is a fan containing a
        single maximal dimensional cone with all its faces, so that   
	$N_\Sigma = \Rsup^n$ 
	as in Definition \ref{good-coordinates}. 
	Then the cones of $\Sigma$ are given by the faces 
		$$\sigma_M \coloneqq \{u \in \R_{\geq 0}^n \mid u_m  = 0 \ \forall m \notin M\}$$
		of $\R_+^n$, where $M$ is ranging over all subsets of $\{1,\dots,n\}$.
		We have the corresponding strata
	\begin{displaymath}
	S_M\coloneqq  O(\sigma_M)^{\rm an} \cap V =\{(z_{1},\dots,z_{n})\in V\,\mid\, z_{i}=0
	\Leftrightarrow i\in M\}
	\end{displaymath}
	of $V$.
	We consider now the canonical decompositions of $T'$ and $T''$ from Theorem \ref{stratification of currents}: 
	\begin{displaymath}
	T'= \sum_{\sigma \in \Sigma} T'_{\sigma}, \quad
	T''=\sum_{\sigma \in \Sigma} T''_{\sigma}.
	\end{displaymath}
	For $M\subset \{1,\dots,n\}$ and  $\sigma \coloneqq \sigma_M \in \Sigma$, we recall that $T_\sigma'$ and $T_\sigma''$ are $\SS$-invariant,
	positive currents with support contained in the closure 
	$\overline{S_M}$  
	of the stratum $S_M$. 
	The co-coefficients of the current 
	$T_\sigma\coloneqq T_\sigma'-T_\sigma''$ are  obtained by first 
	restricting the co-coefficient measures of $T\coloneqq T'-T''$ to 
	$S_M$ and then by extending these
	measures by zero to $V$.
	
	We claim that all currents $T_\sigma, T_\sigma'$ and $T_\sigma''$ are closed. 
	It is clear that $T'$ agrees with $T_{\{0\}}'$ on the dense
        stratum
		$V \cap \torus^\an$ and hence $T_{\{0\}}'|_{V \cap \torus^\an}$ is a closed 
	positive current on $V \cap \torus^\an$. 
Since $T_{\{0\}}'$ is a positive current on $V$, it is clear that the restriction of $T_{\{0\}}'$ to $V \cap \Tan$ has locally finite mass near the boundary $\partial V \coloneqq V \setminus (V \cap \torus^\an)$.
By the Skoda--El Mir Theorem \cite[Theorem III.2.3]{demailly_agbook_2012} and using that $\partial V$ is a null set for $T_{\{0\}}'$, it follows that $T_{\{0\}}'$ is closed on $V$. 
Using induction on $\dim(\sigma)$, we can prove similarly that $T'_\sigma$ is closed for any $\sigma \in \Sigma$.  
The same arguments show that $T_\sigma''$ is closed. 
Then $T_\sigma$ is closed as well.
	
	From Proposition \ref{extension by 0 and trop}, we get the
	equality of canonical decompositions
	\[
	\sum_{\sigma \in \Sigma} \trop_*(T'_{\sigma})=\trop_*(T')
	=\trop_*(T'')=\sum_{\sigma \in \Sigma} \trop_*(T''_{\sigma})
	\]
	which implies $\trop_*(T_\sigma)=0$ for all $\sigma \in \Sigma$.
	
	To prove the proposition, it is enough to show for each
        $\sigma = \sigma_M\in \Sigma$ that $T_{\sigma}$ is zero , or
        equivalently that all  
	co-coefficients $T_\sigma^{IJ}$ are zero for all subsets
        $I,J\subset \{1,\dots,n\}$
	with $|I|=|J|=n-p$  (see Remark \ref{co-coefficients determine currents}). 
	We prove this by induction on the cardinality of $M\cap (I\cup J)$.
	For the initial step, we consider subsets $M,I,J\subset \{1,\dots,n\}$ 
	such that $M\cap (I\cup J)=\emptyset$.  
	Let $f$ be a real valued 
	smooth function on $U$ with compact support in 
	$U\setminus E^{I\cup J}
	=\{u \in U \mid u_i \neq \infty \; \forall i \in I \cup J\}.$ 
	Using $\trop_*(T_\sigma)=0$ and \eqref{co-coefficients and trop},  
	we obtain
	\begin{equation}\label{cocoff-vanish}
	T_\sigma^{IJ}\bigl(z^{-I}\bar{z}^{\, -J}\trop^*(f)\bigr)=0.
	\end{equation}	
	Using that $T_\sigma^{IJ}$ is a Radon measure, we conclude 
	by a standard approximation argument that \eqref{cocoff-vanish} holds 
	for all 
	$f \in C_c^0(U \setminus E^{I\cup J})$. 
	
	For any $g \in C_c^0(V \setminus \trop^{-1}(E^{I\cup J}))$, 
	let $g^{\rm av}$ be the natural projection onto the $\SS$-invariant 
	functions given by averaging over the fibers of $\trop$ with respect 
	to the probability Haar measures. 
	By construction, there is a unique $f \in C_c^0(U \setminus E^{I\cup J})$ 
	with $\trop^*(f) = g^{\rm av}$. 
	The $\SS$-invariance of $T_\sigma$ implies that
	\begin{displaymath}
	T_\sigma^{IJ}(z^{-I}\bar z^{\,-J}g)= T_\sigma^{IJ}(z^{-I}\bar z^{\,-J}g^{\rm av})
	= T_\sigma^{IJ}\bigl(z^{-I}\bar z^{\,-J}\trop^{\ast}(f)\bigr)=0.
	\end{displaymath}
		This means that the restriction of $T_\sigma^{IJ}$ to the open 
	subset $V \setminus \trop^{-1}(E^{I\cup J})$ is identically zero. 
	Since $V \setminus S_M$ is a null set with respect to the Radon measure 
	$T_{{\sigma }}^{IJ}$ and since our
          assumption $M\cap (I\cup J)=\emptyset$ yields 
	$S_M \subset V \setminus \trop^{-1}(E^{I\cup J})$, 
	we deduce that $T_{\sigma}^{IJ}=0$.
	
	For the inductive step, the induction hypothesis is that $T_{\rho}^{IJ}=0$
		whenever $\rho=\sigma_L \in \Sigma
		$ with $|L\cap(I\cup J)|< k$ for some $k\geq 1$. 
	Let $M,I,J \subset \{1,\dots, n\}$ with $|M\cap(I\cup J)|= k$ and
	$m\in M\cap (I\cup J)$. 
	We have to show that $T_\sigma^{IJ}(f)=0$ for $\sigma=\sigma_M$ and any $f \in C_c^\infty(V)$. 
	Using that $V \setminus S_M$
	is a null set with respect to the Radon measure $T_\sigma^{IJ}$,
	we may assume that, in a neighborhood of 
	$\overline{S_M}$, the
	function $f$ depends only the
	variables $z_{j}$, $j\not \in M$. 
	By symmetry, we may also assume without loss of generality that $m\in I$.
	Write $I'=I\setminus \{m\}$. Let $g\in C_c^\infty(V)$ be the function
	given by $g=z_{m}f$. By the assumptions on $f$, there is a  
	neighborhood of $\overline{S_M}$ where 
	\begin{displaymath}
	\frac{\partial g}{\partial z_{m}}=f. 
	\end{displaymath}
	and in this neighborhood, $g$ depends  only on the variables $z_{j}$,
	$j\not \in M$, and $z_m$.
	Consider the smooth form $\eta  = g dz_{I'}\wedge d\bar z_{J}$
	on $V$ which has compact support.  
	Since $T_{\sigma}$ is closed and has support on $\overline{S_M}$, 
	the assumptions on $g$ lead to 
	\begin{equation}\label{eq:290}
	0 = T_{\sigma}(\partial \eta)
	= T_{\sigma}(\pm f dz_{I}\wedge d\bar z_{J}) 
	+ \sum _{j\not \in M\cup I}
	T_{\sigma}\Bigl(\frac{\partial g}{\partial z_{j}} dz_{j}\wedge dz_{I'}\wedge d\bar z_{J} \Bigr),
	\end{equation}
	where the sign depends on the position of $m$ in $I$. 
	Let us first consider the case where $m\not\in J$.
	For each $j\not \in M\cup I$, we get
	$|M\cap \bigl((I'\cup \{j\})\cup J\bigr )|<k$.
	Hence the induction hypothesis gives
	\[
	T_{\sigma}^{(I'\cup\{j\})J}\Bigl(\frac{\partial g}{\partial z_{j}}\Bigr)=0
	\]
	and \eqref{eq:290} implies $T_{\sigma}(f dz_{I}\wedge d\bar z_{J})=0$ and
	\[
	T_\sigma^{IJ}(f) = i^{q^2}T_\sigma(f dz_{I}\land d\bar z_{J})=0.
	\]
	Now consider the case where $m\in J$. 
	We put $J'\coloneqq J\setminus \{m\}$.
	For each $j\not \in M\cup I$, we write
	\begin{displaymath}
	h_{j}=\bar z_{m}\frac{\partial g}{\partial z_{j}}\in C_c^\infty(V).
	\end{displaymath}
	This function depends only on the variables $(z_{j'})_{j' \not \in M}$
	and $z_m$ and, in a neighborhood of $\overline{S_M}$, satisfies  
	\begin{displaymath}
	\frac{\partial h_{j}}{\partial \bar z_{m}}
	= \frac{\partial g}{\partial z_{j}}.  
	\end{displaymath}
	For the compactly supported smooth form 
	$\alpha _{j}=h_{j}dz_{j}\wedge dz_{I'}\wedge d\bar z_{J'}$ on $V$, 
	we deduce as above that 
	\begin{displaymath}
	0=T_{\sigma}(\bar \partial \alpha _{j})=
	T_{\sigma} \left( \pm\frac{\partial g}{\partial z_{j}} d z_j\wedge dz_{I'}\wedge d\bar z_{J} \right)
	+\sum_{j'\not \in M\cup J} T_{\sigma} \left( \frac{\partial h_j}{\partial \bar z_{j'}}
	d \bar z_{j'}\wedge dz_{j} \wedge dz_{I'} \wedge d\bar z_{J'} \right).
	\end{displaymath}
	For $j\not\in M\cup I$ and $j'\not \in M\cup J$, we get
	$|M\cap \bigl((I'\cup \{j\})\cup (J\cup \{j'\}\bigr )|<k$.
	Hence our induction hypothesis gives
	\begin{align*}
	T_{\sigma}\left( \frac{\partial h_j}{\partial \bar z_{j'}}
	d \bar z_{j'}\wedge   dz_{j} \wedge dz_{I'}\wedge d\bar z_{J'} \right) =0.
	\end{align*}
	This implies
	\begin{align*}
	T_{\sigma}\left( \frac{\partial g}{\partial z_{j}}{ d\bar z_{j}} 
	\wedge dz_{I'}\wedge d\bar z_{J} \right) = 0.
	\end{align*}
	for all $j\not\in M\cup I$
	which implies $T_\sigma^{IJ}=0$ by \eqref{eq:290} as before.
	This completes the induction and proves the result.
\end{proof}

In general the map $\trop_{\ast}\colon D^{p,p}_{+}(V)\cap
D^{p,p}(V)^{\SS,F}\to D^{p,p}_{+}(U)$ is not injective as the
following example shows. 
\begin{ex} \label{tropicalization of currents is not injective}
	Consider $\P^{1}_\C$ as a toric variety. Then the $(0,0)$
	current that sends the form $f(z)dz\land i d\bar z$ to the value
	$f(0)$ is a non-zero positive invariant {current}, but its image by
	$\trop_{\ast}$ is zero. Thus, in Proposition \ref{prop:10}  the closedness
	condition is necessary.
\end{ex}

\begin{thm} \label{thm:5}
	Let $U\subset N_{\Sigma }$ be an open subset,
	$V=\trop^{-1}(U)$. Denote by $D^{p,p}(U)_{\cl,+}$ the cone of
	positive closed currents on $U$ and by
	$D^{p,p}(V)^{\SS,F}_{\cl,+}$ the cone of
	$\SS$- and $F$-invariant positive closed currents on $V$. Then
	the map $\trop_{\ast}$ induces {a linear isomorphism} 
	\begin{equation}   \label{main iso}
	\trop_{\ast}\colon D^{p,p}(V)^{\SS,F}_{\cl,+}
	\stackrel{\sim}{\longrightarrow} D^{p,p}(U)_{\cl,+}
	\end{equation} 
	of real cones. 
\end{thm}

\begin{proof}
	By Lemma \ref{trop of complex current}, Proposition \ref{prop:5} and
	Proposition \ref{positivity and trop}, the linear map in \eqref{main
		iso} is well-defined. Surjectivity follows from  Propositions
	\ref{prop:9} and \ref{prop:8}, while  injectivity was proven in
	Proposition \ref{prop:10}.
\end{proof}

The isomorphism in Theorem \ref{thm:5} respects the support of currents.

\begin{prop}\label{prop:16}
	For $T\in D^{p,p}(V)^{\SS,F}_{\cl,+}$,   we have 
	$\trop^{-1}(\supp(\trop_{\ast} T )) = \supp(T)$. 
\end{prop}

\begin{proof}
	Recall that the support of a current $T$ is defined as the
	complement of the maximal open set $W$ such that $T|_{W}=0$.
	Let $C \coloneqq \supp(\trop_{\ast}(T))$. Write $U'=U\setminus C$ and
	$V'=\trop^{-1}(U')$. Then
	$\trop_{\ast}(T)|_{U'}=0$ and hence Proposition
	\ref{prop:10} yields $T|_{V'}=0$ which means 
	\begin{displaymath}
	\supp(T)\subset \trop^{-1}(\supp(\trop_{\ast}(T))).
	\end{displaymath}
	Since $T$ is $\SS$-invariant, then $C'\coloneqq \supp(T)$ is
	also $\SS$-invariant. Write $V''=V\setminus C'$. By the
	invariance of $C'$, we have $V''=\trop^{-1}(U'')$ for an open
	subset $U''\subset U$ (see Remark \ref{ass-toric-variety}). Since $T|_{V''}=0$, it is clear that 
	$\trop_{\ast}(T)|_{U''}=0$. Therefore
	\begin{displaymath}
	\trop^{-1}(\supp(\trop_{\ast}(T)))\subset  \supp(T),          
	\end{displaymath}
	concluding the proof of the proposition. 
\end{proof}

\begin{ex}\label{exm:5}
In the case $U= N_\R$ and $V=\T^\an$ our 
Correspondence Theorem \ref{thm:5} gives a new interpretation of the 
\emph{tropical currents} introduced by Babaee and Huh  \cite[Chapter 3]{babaee2014},
\cite[Section 2]{babaee-huh2017}.

{Let $\Delta$ be a polyhedron in $N_\R$
of dimension $p$ which is \emph{integral $\R$-affine}, i.e.~a polyhedron 
given by finitely many inequalities $\varphi\geq c$ with 
$\varphi\in \mathrm{Hom}_\Z(N,\Z)$ and $c\in \R$.} 
The argument map from polar coordinates induces an $\SS$-equivariant 
fibration of the $\SS$-invariant subset $\trop^{-1}(\Delta)\subset \torusan$ 
over a real torus of dimension $n-p$ with fibers of complex dimension $p$.  
Integration of a complex $(p,p)$-form over the fibers and 
then integrating the resulting function on the real torus with respect to the probability Haar 
measure defines a complex current $T_\Delta\in {D^{n-p,n-p}}(\T^\an)$.
For details of the construction of $T_\Delta$,
we refer to \cite[Definition 2.3]{babaee-huh2017}.

Let $\delta_\Delta$ denote the Lagerberg current of integration
over $\Delta$ defined in \cite[3.6]{gubler-forms}.
The complex current $T_\Delta$ is by construction $\SS$-invariant. 
We get furthermore $\trop_*(T_\Delta)=\delta_\Delta$ by a similar argument 
as in the proof of Lemma \ref{comparison of integration}.
Then Lemma \ref{trop of complex current}\ref{converse of real push-forward} yields that $T_{\Delta }$ is F-invariant.

If $C=(\mathscr C,m)$ is a weighted integral 
$\R$-affine polyhedral complex of pure {dimension $p$} with weights $m$ 
in the sense of \cite[3.1, 3.3]{gubler-forms}, then one defines
\[
T_C\coloneqq \sum_{\genfrac{}{}{0pt}{}{\Delta\in \mathscr C}{\dim\Delta=p}}
m_\Delta T_\Delta\in {D^{n-p,n-p}(\T^\an)},
\,\,\,
\delta_C\coloneqq \sum_{\genfrac{}{}{0pt}{}{\Delta\in \mathscr C}{\dim \Delta=p}}
m_\Delta\delta_\Delta\in {D^{n-p,n-p}(N_\R)}.
\]
Since $\trop_*(T_\Delta)=\delta_\Delta$, we clearly have $\trop_*(T_C)=\delta_C$.

If $C$ is an effective tropical cycle, then 
$T_C$ is closed and positive by \cite[Theorem 2.9]{babaee-huh2017}. 
The same is true for $\delta_C$ by  \cite[3.7]{gubler-forms}.
We conclude that $T_C$ is the unique closed positive current in {$D^{n-p,n-p}(\torusan)^{\SS,F}$} 
such that $\trop_*(T_C) = \delta_C$.
\end{ex}

\subsection{The {analogue of the} Skoda--El Mir Theorem for tropical toric varieties} 
\label{section:tropical Skoda El Mir}

We will prove a tropical analogue of the Skoda--El Mir Theorem.
In this subsection, $U$ is an open subset of $N_\Sigma$ and 
$E$ is the intersection of $U$ with a union of strata closures.

Let $T \in D^{p,p}(U \setminus E)$ be a positive Lagerberg current. We
pick $\rho \in \Sigma$ and choose toric coordinates on $U_\rho$ as in
Definition \ref{good-coordinates}. By Proposition \ref{prop:6}, the
co-coefficients $T^{IJ}$  
are real Radon measures on $U_\rho \setminus (E^{I \cup J} \cup E)$
with {total variation measure} $|T^{IJ}|$. 
We will consider the open immersion 
\[{i_E^{IJ}}\colon U_\rho \setminus (E^{I \cup J} \cup E) 
  \longrightarrow U_\rho \setminus E^{I\cup J}.\]

\begin{defi} \label{extension by 0} Let
  $T \in D^{p,p}(U \setminus E)$ be a positive Lagerberg current.  We
  say that \emph{$T$ is extendable by zero to $U$} if for any
  $\rho\in \Sigma$,   any toric coordinates on $U_\rho$ and all subsets $I,J$ of
  $\{1,\dots,n\}$ with $|I|= |J|=n-p$, the Radon measure $T^{IJ}$
  admits an image {Radon} measure with respect to $i_E^{IJ}$ (see
  Appendix \ref{facts from measure theory}). In other words, there is
  a (unique) Radon measure $i_E^{IJ}(T^{IJ})$ on
  $U_\rho \setminus E^{I\cup J}$ which agrees with $T^{IJ}$ on
  $U_\rho \setminus (E^{I \cup J} \cup E)$ such that
  $U_\rho \cap E \setminus E^{I \cup J}$ is a null set with respect to
  $i_E^{IJ}(T^{IJ})$.
\end{defi}

\begin{lem} \label{current extends by zero} If $T$ is extendable by
  zero to $U$, then there is a unique Lagerberg current
  $\widetilde T \in D^{p,p}(U)$ such that for any $\rho \in \Sigma$
  and all toric coordinates on $U_\rho$, the co-coefficient
  $(\widetilde T)^{IJ}$ is induced by the Radon measure
  ${i_E^{IJ}}(T^{IJ})$ for all $I,J$. Moreover, we have
  $\widetilde T|_{U \setminus E} = T$ and the Lagerberg current
  $\widetilde T$ is positive.
\end{lem}

In the above situation, we say that $\widetilde T$ is the extension of $T$ by zero to $U$.

\begin{proof}
  For $\rho \in \Sigma$ and a choice of toric coordinates, Remark
  \ref{co-coefficients determine currents} shows that there is a
  unique current $\widetilde T \in D^{p,p}(U_\rho)$ with
  co-coefficients induced by the Radon measures ${i_E^{IJ}}(T^{IJ})$.
  Clearly, we have
  $\widetilde T|_{U_\rho \setminus E} = T|_{U_\rho \setminus E}$ and
  $\widetilde T$ does not depend on the choice of toric
  coordinates. Moreover, the extensions $\widetilde T$ constructed on
  the open covering $(U_\rho)_{\rho \in \Sigma}$ agree on overlapping.
  By glueing, we get a Lagerberg current on $U$ also denoted by
  $\widetilde T$.  Using Lemma \ref{uniqueness sequence} for
  $A \coloneqq E \cap U_\rho$, we deduce easily that $\widetilde T$ is
  again a positive Lagerberg current.
\end{proof}

\begin{defi} \label{def complex local finite mass generalized}
	Let $T\in D^{p,p}(U \setminus E)$ be a  positive Lagerberg current. 
	Generalizing Definition \ref{def complex local finite mass}, 
	we say that  $T$ has \emph{$\C$-finite local mass on $U$} if for all $\rho \in \Sigma$ the 
	co-coefficients $T^{IJ}$, 
	which may be seen as real Radon measures on $U_\rho \setminus (E \cup  E^{I \cup J})$, 
	satisfy the condition that {the  Borel measures on $U_\rho$, given as} 
	the image measures
	\begin{displaymath}
	j_E^{IJ}\biggl(|T^{I,J}|\prod_{i\in I}e^{-u_{i}}\prod_{j\in J}e^{-u_{j}}\biggr)
	\end{displaymath}
	with respect to the open immersion $j_E^{IJ}\colon U_\rho \setminus (E \cup E^{I \cup J}) \to U_\rho$, are locally finite on $U_\rho$.
\end{defi}

\begin{thm} \label{tropical El Mir} Let $T \in D^{p,p}(U \setminus E)$
  be a closed positive Lagerberg current which has $\C$-finite local
mass on $U$.  Then $T$ is extendable by zero to $U$ and the
  extension $\widetilde T$ of $T$ by zero to $U$ from Lemma
  \ref{current extends by zero} is a closed positive Lagerberg current
  on $U$.
\end{thm}

\begin{proof}
	{We may assume that $U=U_\rho$ for some $\rho \in \Sigma$ and we choose toric coordinates as in Definition \ref{good-coordinates}.} 
	Since the function  $\prod_{i\in I}e^{-u_{i}}\prod_{j\in J}e^{-u_{j}}$ 
	is locally finite on $U \setminus E^{I \cup J}$  
	{and since $T$ has $\C$-finite local mass on $U$, the Borel measure $i_E^{IJ}(|T^{IJ}|)$} 	is locally finite on $U \setminus E^{I\cup J}$ and 
		hence $T$ is extendable by zero to $U$. 
		By Lemma \ref{current extends by zero},    
		the extension $\widetilde T$ of $T$ by zero to $U$ is a  positive Lagerberg current on $U$. 
	
	Let $V \coloneqq \trop^{-1}(U)$ and let $D \coloneqq \trop^{-1}(E)$. 
	Note that $D$ is the intersection of the open subset $V$ of $\Xsigmaan$ with a union of strata closures and 
	hence is a closed analytic subset of the complex toric manifold $\Xsigmaan$. 
	By our correspondence theorem (Theorem \ref{thm:5}), 
	there is a unique closed positive current 
	$S \in D^{p,p}(V \setminus D)^{\SS,F}_{\cl,+}$ with $\trop_*(S)=T$. 
	Using that $T$ has $\C$-finite local mass on $U$, 
	it follows from \eqref{variance of co-coefficients and trop} that $S$ has finite local mass on $V$. 
	The complex Skoda--El Mir Theorem (see \cite[Theorem III.2.3]{demailly_agbook_2012}) 
	shows that the extension $\widetilde{S}$ of $S$ by zero to $V$ is a closed positive current on $V$. 
	By construction, we have $\trop_*(\widetilde S)= \widetilde T$ and 
	hence $\widetilde T$ is a closed Lagerberg current.
\end{proof}

\begin{cor} \label{decomposition theorem and closed}
  Let $T \in D^{p,p}(U)$ be  positive and closed with canonical
  decomposition $T=\sum_{\sigma \in \Sigma} T_\sigma$ from Theorem
  \ref{stratification of currents}. Then every Lagerberg current
  $T_\sigma$ is positive and closed.
\end{cor}

\begin{proof}
  By Theorem \ref{stratification of currents}, every $T_\sigma$ is a
  positive Lagerberg current. We know from Proposition \ref{prop:8}
  that the closed positive Lagerberg current $T$ has $\C$-finite local
  mass. These two facts show that every $T_\sigma$ has $\C$-finite
  local mass.

    We prove that $T_\sigma$ is closed
    by induction on the dimension of $\sigma $. By construction
    $T_{\{0\}}$ is the extension by zero of $T|_{N_\R\cap U}$. 
    Hence $T_{\{0\}}$ is closed
    by Theorem \ref{tropical El Mir}. We assume now that $T_{\tau }$
    is closed for all $\tau $ with $\dim(\tau )<\dim (\sigma )$. Hence
    \begin{displaymath}
      T_{\sigma }'\coloneqq T-\sum_{\dim(\tau )<\dim(\sigma )}T_{\tau }
    \end{displaymath}
    is closed. Since
    $
      T_{\sigma }|_{U_{\sigma }}= T'_{\sigma }|_{U_{\sigma }}
    $
    we deduce that $T_{\sigma }|_{U_{\sigma }}$ is closed. As
    $T_{\sigma }$ is the extension by zero of this last current, Theorem
    \ref{tropical El Mir} implies that $T_{\sigma }$ is closed.
\end{proof}

\begin{appendix}

\section{Reminder about Radon and regular Borel measures}
\label{facts from measure theory}

For the convenience of the reader, we gather
the used conventions about Radon measures and some basic facts.   
In this paper, we deal only with measures on locally compact Hausdorff 
spaces which have a countable basis, so let us consider  such a space $Y$.

For $\K\in \{\R,\C\}$ and a compact subset $Z$ of $Y$, we write 
$C_Z^0(Y,\K)$ for the space of $\K$-valued continuous functions 
on $Y$ with support
in $Z$ equipped with the topology induced by the supremum norm.
The space $C^0_c(Y,\K)$ of $\K$-valued continuous functions on $Y$ with compact support
is the direct limit of the spaces $C_Z^0(Y,\K)$.
We equip $C^0_c(Y,\K)$ with the direct limit topology in the category of
locally convex topological vector spaces 
(see \cite[III \S1 n\textsuperscript{o}1]{bourbaki-integration-1-4}).

We define the space of \emph{real Radon measures} on  $Y$ as the topological dual of $C^{0}_{c}(Y,\R)$.
Observe that our Radon measures are precisely the measures
considered by Bourbaki 
\cite[III \S1 n\textsuperscript{o}3+5]{bourbaki-integration-1-4}.
If $\mu\colon C^{0}_{c}(Y,\R) \to \R$ is linear with 
$\mu(f) \geq 0$ for all $0 \leq f \in C^{0}_{c}(Y,\R)$, then we call 
$\mu$ a {\it positive linear functional}. 
Note that $\mu$ is then continuous and gives rise to a 
\emph{positive Radon measure}. 
The Riesz representation theorem shows that there is a unique regular 
Borel measure $\mu^B$ on $Y$ such that
\begin{align*}
\mu(f) = \int_Y f \, d\mu^B
\end{align*}
for all  $f \in C^{0}_{c}(Y,\R)$. 
	Recall that a Borel measure $\lambda$ on $Y$ 
is regular if and only if it is \emph{locally finite}, 
i.e.~any point has an open neighborhood $U$ with $\lambda(U)< \infty$
\cite[Prop.~7.2.3 and 7.2.5]{cohn1993}.
In contrast to Radon measures, we always assume that a Borel measure
is positive.

A real Radon measure $\mu\colon C_c^0(Y,\R)\to \R$
can be written as the difference $\mu=\mu_+-\mu_-$
of positive Radon measures $\mu_\pm$
which are induced by locally finite (positive) Borel 
measures $\mu^B_\pm$.
The Hahn--Jordan decomposition theorem tells us that the Borel measures 
$\mu^B_\pm$ are unique if one requires them to be minimal.
The positive Radon measure $|\mu| \coloneqq \mu_+ + \mu_-$ is 
called  the \emph{{total variation measure of $\mu$}}. 
We call $|\mu|^B(Y)=\mu_+^B(Y)+\mu_-^B(Y)$ the \emph{total variation
of $\mu$}. 
Observe that $\mu^B_+-\mu^B_-$ 
\emph{is not necessarily a signed Borel measure} as it has only 
well defined finite values on $\mu^B_\pm$-finite Borel sets.	
The standard example is the real Radon measure on $\R$ 
with density function $\sin(x)$ with respect to the 
Lebesgue measure which is not a signed Borel measure as the 
total space $\R$ has no well-defined mass.

A \emph{complex Radon measure} on $Y$ is defined as
a continuous linear functional 
$\mu\colon C^{0}_{c}(Y,\C) \to \C$. 
It is clear that the real and imaginary part of $\mu$ are 
real Radon measures. 
From this, we conclude as above 
that $\mu$ corresponds locally on an open subset $U$ of $Y$ to
a complex Borel measure $\mu_U^B$.
 The \emph{{total variation measure
of $\mu_U^B$}} is the unique Borel measure $|\mu_U^B|$ on $U$ 
such that there is a Borel measurable function $\theta\colon U \to \R$ 
with $d\mu_U^B = e^{i\theta}d|\mu_U^B|$ \cite[6.1, 6.12]{rudin1966}.
The Borel measures $|\mu_U^B|$ glue to a 
Borel measure on $Y$ which is associated to a Radon measure $|\mu|$ on
$Y$  called 
the \emph{{total variation measure }of the Radon measure $\mu$}.
We call $|\mu|^B(Y)$ the \emph{total variation of $\mu$}.

Let $g\colon X\to Y$ be a continuous map of locally compact {Hausdorff} 
spaces which admit a countable basis. 
We say that a real or complex Radon measure $\mu$ on $X$ \emph{admits
an image {Radon} measure under $g$}  if the image measure 
{$g(|\mu|^B)$ is a locally finite Borel measure.}

Assume that the real or complex Radon measure $\mu$ 
on $X$ admits an image {Radon} measure under $g$.
If $\mu$ is a real Radon measure on $X$, then 
we write $\mu=\mu_+-\mu_-$ as above.
By our assumption the image measures $f(\mu^B_+)$ and
$f(\mu^B_-)$ are locally finite and define Radon measures 
$f(\mu)_+$ and $f(\mu)_-$ on $Y$.
We call the real Radon measure $f(\mu)\coloneqq f(\mu)_+-f(\mu)_-$ 
the \emph{image measure of $\mu$ under $f$}.
If $\mu$ is a complex Radon measure $\mu$ on $X$, we define the 
\emph{image measure $f(\mu)$} by treating the real and 
the imaginary part of $\mu$ separately.

It is straightforward to check that real
and complex Radon measures always admit image {Radon} measures under
proper maps.

\end{appendix}

\bibliographystyle{alpha}

\end{document}